\documentclass[%
  12pt, 
]{amsart}
  \usepackage{aliascnt}
  \usepackage{mathtools}
  \usepackage{keytheorems}
  \usepackage{dsfont} 
  \usepackage{amssymb} 
  \usepackage{yhmath} 
  \usepackage[%
      shortlabels,
      inline,
  ]{enumitem} 
  \usepackage{mathrsfs}
  \usepackage[Symbol]{upgreek}
  \usepackage{xfrac}
  \usepackage{accents}
  \usepackage[dvipsnames]{xcolor} 
  \usepackage{interval} 
  \usepackage{mleftright} 
  \usepackage{fixdif}
  \usepackage{leftindex} 
  \usepackage{mathcommand} 
  \usepackage{tikz} 
  \usepackage[all,cmtip]{xy} %
  \usepackage{zref-clever}
  \usepackage[%
      pagebackref, 
  ]{hyperref}
  \usepackage{caption} 
  \numberwithin{equation}{section}
  \mathtoolsset{showonlyrefs}
  \let\realItem\item 
\makeatletter
\NewDocumentCommand\myItem{ o }{%
   \IfNoValueTF{#1}%
      {\realItem}
      {\realItem[#1{\MakeLinkTarget[item]{}}]\def\@currentlabel{#1}}
}
\makeatother
\setlist[enumerate,1]{
  before=\let\item\myItem, 
  label = \textup{(\alph*)}, 
  ref = \textup{(\alph*)}
}
  
  \zcsetup{%
      cap,
      abbrev,
      nameinlink,
  }
  \hypersetup{%
      bookmarksnumbered=true,
      colorlinks=true,
      linkcolor=blue,
      citecolor=cyan,
      pdfstartview=FitBH,
  }
  \renewcommand*{\backrefalt}[4]{%
      \ifcase #1 %
        No citations.%
      \or
        $\uparrow$ #2%
      \else
        $\uparrow$ #2%
      \fi
  }
  \providecommand{\noopsort}[1]{} 

  \newkeytheorem{mainthm}[%
      name=Theorem,%
      counter-format=\Alph{mainthm},%
  ]

  \newkeytheorem{%
      theorem,%
      proposition,%
      lemma,%
      corollary,%
      conjecture,%
  }[sharenumber=equation]

  \newkeytheorem{%
      definition,%
      remark,%
      notation,%
      example,%
      warning,%
      construction,%
      convention,%
  }[sharenumber=equation,style=definition]

  \newkeytheorem{deflem}[%
      name={Definition/Lemma},%
      sharenumber=equation,%
      style=definition]

      \newkeytheorem{defprop}[%
      name={Definition/Proposition},%
      sharenumber=equation,%
      style=definition]

  \NewDocumentCommand \ListInThm { m O{\textup} O{(\roman{enumi})} O{.} }
    {
      \AtBeginEnvironment{#1}
        {
          \zcsetup{ countertype={enumi=#1} }
          \setlist[enumerate,1]
            {
              label={#2{#3}},
              ref={\csname the#1\endcsname#4{#3}}
            }%
        }
    }

  \ListInThm{theorem}
  \ListInThm{corollary}

  \usetikzlibrary{%
      cd,%
      arrows,%
      backgrounds,%
      calc,%
      decorations,%
      decorations.pathmorphing,%
      shapes,%
      tikzmark%
  }
  \tikzcdset{%
      scale cd/.style = %
          {%
            every label/.append style = {scale = #1},%
            cells = {nodes = {scale = #1}}%
          }%
  }

  \NewDocumentMathCommand\txand%
      { O{\quad} }{#1\text{and}#1}
  \NewDocumentMathCommand\txforall%
      { O{\quad} }{#1\text{for all }}
  \NewDocumentMathCommand\txforsome%
      { O{\quad} }{#1\text{for some }}
  \NewDocumentCommand\SetSymbol{o}{\nonscript\:#1\vert\allowbreak\nonscript\:\mathopen{}}  
  \DeclarePairedDelimiterX\Set[1]\{\}
      {#1} 
  \NewDocumentMathCommand\textSet{m}%
      {\Set*{\text{#1}}}
  \DeclarePairedDelimiterX\GSet[1]\langle\rangle
      {#1} 
  \NewDocumentCommand\placeholder{}{\:\cdot\:} 
  \NewDocumentCommand\NewPairedDelimiterS{mmm}{%
      \DeclarePairedDelimiterX{#1}[1]{#2}{#3}%
          {\ifblank{##1}{\placeholder}{##1}}%
  }
  \NewDocumentCommand\NewPairedDelimiterSS{mmmO{,}}{%
      \DeclarePairedDelimiterX{#1}[2]{#2}{#3}%
          {\ifblank{##1}{\placeholder}{##1}%
            #4%
          \ifblank{##2}{\placeholder}{##2}}%
  }
  \NewPairedDelimiterS\abs\lvert\rvert 
  \NewPairedDelimiterS\norm\lVert\rVert 
  \NewPairedDelimiterS\lrangle\langle\rangle 
  \NewPairedDelimiterS\paren\lparen\rparen 
  \NewPairedDelimiterS\lrbrack\lbrack\rbrack 
  \NewPairedDelimiterS\lrbrace\lbrace\rbrace 
  \NewPairedDelimiterS\ceil\lceil\rceil 
  \NewPairedDelimiterS\floor\lfloor\rfloor 
  \NewPairedDelimiterS\bra\langle\vert 
  \NewPairedDelimiterS\ket\vert\rangle 
  \NewPairedDelimiterSS\braket\langle\rangle[%
      \,\delimsize\vert\,\mathopen{}%
  ] 
  \NewPairedDelimiterSS\pairing\langle\rangle
  \NewPairedDelimiterSS\inner\lparen\rparen
  \NewPairedDelimiterSS\Liebracket\lbrack\rbrack
  \DeclarePairedDelimiterX\bracket[3]\langle\rangle%
      {\ifblank{#1}{\placeholder}{#1}%
      \,\delimsize\vert\,\mathopen{}%
      \ifblank{#2}{\placeholder}{#2}
      \,\delimsize\vert\,\mathopen{}%
      \ifblank{#3}{\placeholder}{#3}}%
  \NewDocumentMathCommand\dparen{m}%
      {\lparen\!\lparen{#1}\rparen\!\rparen}
  \NewDocumentMathCommand\dbrack{m}%
      {\lbrack\!\lbrack{#1}\rbrack\!\rbrack}
  \NewDocumentMathCommand\dangle{m}%
      {\langle\!\langle{#1}\rangle\!\rangle}
  \NewDocumentCommand \fun { m e{^_} d() }
      {%
        \operatorname{#1}%
        \IfValueT{#2}{\sp{#2}}%
        \IfValueT{#3}{\sb{#3}}%
        \IfNoValueTF{#4}{}{\mleft(#4\mright)}%
      }
  \LoopCommands{%
      NZQRCKkWM%
  }[#1]{\declaremathcommand#2{\mathbb#1}}
\declaremathcommand\kk{\Bbbk}
  
  \LoopCommands{%
      {Aut}%
      {an}%
      {Coker}%
      {Cotor}%
      {coker}%
      {Der}%
      {End}%
      {Ext}%
      {Gal}%
      {Gr}%
      {gl}%
      {GL}%
      {gr}%
      {Ho}%
      {Hom}%
      {Id}%
      {id}%
      {Ind}%
      {Ker}%
      {Mor}%
      {Ob}%
      {Pic}%
      {Proj}%
      {pr}%
      {Res}%
      {SL}%
      {Sq}%
      {Spc}%
      {Spf}%
      {Spec}%
      {Spm}%
      {supp}%
      {Tor}%
      {tr}%
      {wt}%
      {Tot}%
  }[#1]{\declaremathcommand#2{\fun{#1}}}
    \renewmathcommand\le\leqslant
    \renewmathcommand\ge\geqslant
    \declaremathcommand\one{{\scriptstyle\mathfrak{I}}}
    \declaremathcommand\unit{\mathbb{1}}
    \declaremathcommand\iu{\mathtt{i}}
    \declaremathcommand\e{\mathsf{e}}
    \declaremathcommand\Rf{\operatorname{\mathbf{R}}}
    \declaremathcommand\qbar{\mathscr{q}}
    \declaremathcommand\SG{\mathfrak{S}}
    \declaremathcommand\Gm{\mathbb{G}_{\mathtt{m}}}
    \declaremathcommand\Ga{\mathbb{G}_{\mathtt{a}}}
    \declaremathcommand\vac{\vac}
    \declaremathcommand\cch{\mathscr{c}}
    \declaremathcommand\cfv{\upomega}
    \declaremathcommand\hollowcolon{{}^{\circ}_{\circ}}
    \declaremathcommand\O{\mathscr{O}}
    \declaremathcommand{\H}{\fun{H}}
    \declaremathcommand\et{\text{\'et}}
    \declaremathcommand\pt{\mathsf{p}}
    \declaremathcommand\tpt{\tilde{\mathsf{p}}}
    \declaremathcommand\qt{\mathsf{q}}
    \declaremathcommand\Moduli{\mathcm{M}}
    \declaremathcommand\shHom{\fun{\mathscr{Hom}}}
    \declaremathcommand\V{\mathbb{V}}
    \declaremathcommand\uAut{\operatorname{\underline{Aut}}}
    \declaremathcommand{\G}{\mathbb{G}}
    \declaremathcommand{\vac}{\mathds{1}}
        
    \newmathcommand\Span{\operatorname{span}}
    \newmathcommand\ch{\operatorname{ch}}
    \newmathcommand\Image{\operatorname{Im}}
    
    \newmathcommand\Lie\Liebracket
    \NewDocumentMathCommand\odv{m}%
        {\frac{\d}{\d{#1}}}%
    \NewDocumentMathCommand\pdv{m}%
        {\frac{\partial}{\partial{#1}}}
    \NewDocumentMathCommand\dual{m}
        {\ifblank{#1}{(\:\cdot\:)}{#1}^{\ast}}
    \NewDocumentMathCommand\rldual{m}
        {\ifblank{#1}{(\:\cdot\:)}{#1}^{\dagger}}
    \NewDocumentMathCommand\opp{m}
        {\ifblank{#1}{(\:\cdot\:)}{#1}^{\mathrm{op}}}
    \NewDocumentMathCommand\invo{m}%
        {\prescript{\theta}{}{#1}}%
    \NewDocumentMathCommand\vect{m}%
        {\boldsymbol{#1}}
    \NewDocumentMathCommand\grp{m}%
        {\fun{\mathsf{#1}}}%
    \NewDocumentMathCommand\cat{m}%
        {\operatorname{\mathsf{#1}}}%
      
    \NewDocumentMathCommand\vo{mm}%
        {#1_{(#2)}}%
    \NewDocumentMathCommand\lo{mm}%
        {#1_{[#2]}}%
    \NewDocumentMathCommand\VL{m}%
        {L_{#1}}%

    \NewDocumentMathCommand\Nn{O{\bullet}}%
        {\mathsf{N}^{#1}}
    \NewDocumentMathCommand\NL{O{\bullet}}%
        {\mathsf{N}_{\mathsf{L}}^{#1}}
    \NewDocumentMathCommand\NR{O{\bullet}}%
        {\mathsf{N}_{\mathsf{R}}^{#1}}
    \NewDocumentMathCommand\NLR{O{\bullet}}%
        {\mathsf{N}^{#1}}
    \NewDocumentMathCommand\cNL{O{\bullet}}%
        {\prescript{\mathrm{c}}{}{\mathsf{N}}_{\mathsf{L}}^{#1}}
    \NewDocumentMathCommand\cNR{O{\bullet}}%
        {\prescript{\mathrm{c}}{}{\mathsf{N}}_{\mathsf{R}}^{#1}}
    \NewDocumentMathCommand\cNLR{O{\bullet}}%
        {\prescript{\mathrm{c}}{}{\mathsf{N}}^{#1}}

    \NewDocumentMathCommand\normord{m}%
        {\mathopen{\hollowcolon}\mathinner{#1}\mathclose{\hollowcolon}}%

\newmathcommand\bPhi{\mathbf{\Phi}}

\makeatletter
\newmathcommand{\ostar}{\mathbin{\mathpalette\make@circled\star}}

\newcommand{\make@circled}[2]{%
  \ooalign{$\m@th#1\smallbigcirc{#1}$\cr\hidewidth$\m@th#1#2$\hidewidth\cr}%
}
\newcommand{\smallbigcirc}[1]{%
  \mathcal{V}enter{\hbox{\scalebox{0.7}{$\m@th#1\bigcirc$}}}%
}

\newmathcommand{\halfstar}{\mathbin{\tikz@halfstar}}
\newcommand{\tikz@halfstar}{%
  \tikzstyle{scorestars}=[star, star points=5, star point ratio=2.25, draw, inner sep=0.2ex, anchor=outer point 3]%
  \begin{tikzpicture}[baseline]%
    \node[scorestars] {};
    \path node[scorestars,fill=black] (s) {} [clip] (s.south west) rectangle (s.north);
  \end{tikzpicture}%
}
\makeatother

\def\cal{\mathcal}
\def\frak{\mathfrak}

\def\eqdef{\stackrel{\textup{def}}{=}}
\def\Im{\operatorname{Im}}

\def\op{{\operatorname{op}}}
\DeclareMathOperator{\im}{im}

\DeclareMathOperator{\Vir}{Vir}

\def\Z{{\mathbb{Z}}}
\def\C{{\mathbb{C}}}
\def\Q{{\mathbb{Q}}}
\def\N{{\mathbb{N}}}
\def\P{{\mathbb{P}}}
\def\F{{\mathbb{F}}}
\def\A{{\mathbb{A}}}

\declaremathcommand{\sfL}{{\mathsf{L}}}
\declaremathcommand{\sfR}{{\mathsf{R}}}
\newcommand{\Uu}{\mathit{U}}

\newcommand{\UV}{\mathscr{U}}
\newcommand{\UVf}{\mathscr{U}^{\mathsf{f}}}
\newcommand{\UVR}{\mathscr{U}^{\mathsf{R}}}
\newcommand{\UVL}{\mathscr{U}^{\mathsf{L}}}

\newcommand{\UuR}{\Uu^{\mathsf{R}}}

\newcommand{\hUu}{\widehat{\Uu}}

\newcommand{\frakL}{\mathfrak{L}}

\newcommand{\LV}{{\mathfrak{L}}(V)}
\newcommand{\LVf}{{\mathfrak{L}}(V)^{\mathsf{f}}}

\newcommand{\LVR}{\LV^{\mathsf{R}}}
\newcommand{\LVL}{\LV^{\mathsf{L}}}

\newcommand{\PhiL}{\Phi^\mathsf{L}}
\newcommand{\PhiR}{\Phi^\mathsf{R}}

\newcommand{\Ac}{\mathfrak{A}}

\newcommand{\Aa}{\mathsf{A}}

\newcommand{\Cs}{\mathscr{C}}

\usepackage{stackengine}

\newcommand{\wh}{\widehat}

\newmathcommand\ctensor{\mathbin{\mathop{\widehat{\otimes}}}}

\newcommand{\mrm}{\mathrm}

\DeclareMathOperator{\Char}{char}

\begin{document}


\title{Conformal blocks for modular vertex algebras}

\begin{abstract}
Given a vertex operator algebra $V$ over $\C$, it is known how to define so-called sheaves of coinvariants and conformal blocks on the moduli space $\overline{\mathcal{M}}_{g,n}$ of stable, $n$-pointed, genus $g$ curves. We extend this theory to vertex algebras equipped with a notion of change of coordinates over any algebraically closed field. As a by-product of our arguments, we also show that a number of standard assumptions in these constructions can be removed, such as being $C_1$-cofinite or CFT-type.
\end{abstract}

\date{\today}
\author[Griffin]{Colton Griffin}

\address{David Rittenhouse Lab, University of Pennsylvania, 209 South 33rd Street, Philadelphia, PA 19104-6395}


\setcounter{tocdepth}{1}

\maketitle

\section{Introduction}

Vertex operator algebras (VOAs) over the complex numbers have a rich and interesting geometric representation theory. By assigning a VOA module to each marked point of a stable algebraic curve, one may form the associated mutually dual spaces of coinvariants and conformal blocks. These give rise to sheaves on the moduli space $\overline{\mathcal{M}}_{g,n}$ of stable, $n$-pointed curves of genus $g$ \cite{TUY89,NT05,DGT21,DGT23,DGK23}. If $V$ is of CohFT type,\footnote{A VOA $V$ is of \textit{CohFT} type if it is of CFT-type, rational, $C_2$-cofinite, simple, and self-contragredient.} then they are vector bundles on $\overline{\mathcal{M}}_{g,n}$, and their Chern classes form semisimple cohomological field theories \cite{DGT22}. Since it is common to study algebraic curves over fields other than $\C$ or even more general bases, it is natural to ask if the theory of coinvariants and conformal blocks on $\overline{\mathcal{M}}_{g,n}$ over $\C$ extends to a more general base.

Here, we consider the theory of coinvariants and conformal blocks associated to vertex algebras over any algebraically closed field $\kk$ of any characteristic. The sheaves we construct are quasi-coherent and as such rely on certain forms of descent data (\zcref{sec:descent}). The structure of a VOA in positive characteristic may not permit descent to $\overline{\mathcal{M}}_{g,n}$ (\zcref{rmk:poschardescent}). 
To handle this challenge, we study $\N$-graded vertex algebras equipped with an action of the group $\Aut\mathcal{O}$ of continuous automorphisms of the power series ring $\kk[\![t]\!]$.

With our assumptions, we give generalizations of the propagation of vacua and factorization properties for coinvariants and conformal blocks (\zcref{thm:propagation of vacua}, \zcref{thm:factorization theorem}) and show that, under certain assumptions on modules, their associated sheaves on $\overline{\mathcal{M}}_{g,n}$ are quasi-coherent (\zcref{sheaves of coinvariants descent}). Assuming that the underlying VOA satisfies the so-called strong identity condition associated to its mode transition algebra (\zcref{def:stronglyunital}), we show that the sheaves satisfy smoothing (\zcref{smoothing property}).

Our results may be of interest to readers that are exclusively concerned with VOAs over $\C$ since, as a by-product of our arguments, a number of prior assumptions in the construction of these algebraic structures can be relaxed. For instance, one no longer needs the vertex algebra to be of CFT-type to define the sheaves of coinvariants on $\overline{\mathcal{M}}_{g,n}$. Moreover, the assumption that the vertex algebra is $C_1$-cofinite is not needed for the factorization property \cite[Proposition 6.2.1]{DGT23}, \cite[Lemma 4.4.4]{DGK23}. This allows one to consider conformal blocks associated to a more general class of vertex algebras such as the $\beta\gamma$ (or bosonic ghost) vertex algebra, which is a VOA not of CFT-type. The mode transition algebra of this VOA has been shown to satisfy the strong identity condition, so we have the smoothing property \cite{Barron}.

Over $\C$, the vertex algebra sheaf associated to a nodal curve $C$ and a VOA $V$ is known to admit a flat logarithmic connection. Given a vertex algebra satisfying our assumptions, we show that the vertex algebra sheaf $\mathcal{V}_C$ on a smooth curve $C$ admits a natural $\mathcal{D}$-module structure (\zcref{PD-connection smooth curve}).
More precisely, we have a \textit{stratification} or a so-called \textit{PD-connection} from \cite{BerthelotOgus}. If the curve has a nodal singularity, then we require the compatible structure of a so-called \textit{M\"obius vertex algebra} (or $\mathcal{H}$-module vertex algebra from \cite{Li18}), in which case we have a logarithmic version of a PD-connection (\zcref{PD-connection nodal curve}).

Aside from the geometric challenges in studying vertex algebras in positive characteristic, there are also some missing details in the vertex algebra literature that need to be filled in for our purposes. In particular, there are many well-known statements about $C_2$-cofiniteness that have to be reproven in positive characteristic to establish certain properties of coinvariants, such as coherence. We recover many of the results of \cite{KL98, GN03, BuhlSpanning, AN2} for CFT-type vertex algebras over arbitrary rings, and we apply these to show that the sheaves of coinvariants associated to a $C_2$-cofinite, CFT-type vertex algebra are coherent (\zcref{cor:CoherenceOnMgn}). Many of the other tools required to study sheaves of coinvariants already exist in the literature, such as the Zhu algebra $\Aa(V)$ from \cite{Z} and its generalizations in \cite{DLM2, DongRen14, Ren17}; the mode transition algebra $\mathfrak{A}(V)$ from \cite{DGK23}; or generalizations of M\"obius vertex algebras in \cite{Li18}.

Possible directions of research related to modular conformal blocks include the following:
\begin{enumerate}
    \item It was shown in \cite{DGT23} that rational, $C_2$-cofinite VOAs over $\C$ give rise to finite rank vector bundles over $\overline{\mathcal{M}}_{g,n}$. It would be interesting to see if this extends to sheaves of coinvariants over any algebraically closed field.  This would provide new avenues for studying $\overline{\mathcal{M}}_{g,n}$ in positive characteristic. If the vertex algebra is $C_2$-cofinite, then we show in \zcref{cor:CoherenceOnMgn} that the sheaves of coinvariants are coherent using a similar proof to \cite{AN2}. If $V$ is rational, then it satisfies the strong identity condition \cite{Griffin26}, so we have the smoothing property (\zcref{smoothing property}). Using a similar argument to \cite[Corollary 5.2.6]{DGK23}, it remains to show that the sheaves of coinvariants are flat. For this, it suffices to provide a twisted logarithmic $\mathcal{D}$-module structure.

    \item VOAs of CohFT type are closely related to modular tensor categories and the Verlinde formula \cite{Huang2005}. While we work exclusively with algebro-geometric methods, it would be interesting to see if there are non-Archimedean analytic methods of studying vertex algebras applicable for modular forms or modular tensor categories. This perspective was proposed in \cite{Mason-p-adic}, where they work with analytic versions of vertex algebras over $\Q_p$. 
    We wonder if our results can be stated in a rigid geometric framework when the base field is non-Archimedean. In this way, one could study non-Archimedean versions of 2D CFT using sewing operations on $\P^1_\kk$ as in \cite{Huang1997}. In the same vein, it may be possible to study the relationship between modular forms and VOAs in a non-Archimedean setting.
    In some sense, our work gives a partial confirmation of these ideas. Indeed, given an non-Archimedean, algebraically closed field $\kk$, rigid GAGA implies that the coherent sheaves on $\overline{\mathcal{M}}_{g,n}$---such as those arising from $C_2$-cofinite vertex algebras---are in equivalence with coherent sheaves on the rigid analytification $\overline{\mathcal{M}}_{g,n}^{\mathrm{an}}$, which parametrizes stable, $n$-pointed, rigid analytic curves of genus $g$.
\end{enumerate}
We hope that our work on modular conformal blocks will stimulate interest in VOAs and their geometric applications over more arbitrary fields.

\subsection{Plan of the paper}
In \zcref{part:prelims}, we develop the general theory of vertex algebras that we will use throughout the paper. 
In \zcref{sec:Basics}, we define vertex algebras and their modules, studying the modules in terms of the universal enveloping algebra and the associated (higher) Zhu algebras. 
In \zcref{sec:contragredient}, we give the basic tools necessary to define an involution on the universal enveloping algebra that gives the left $V$-module structure on the contragredient dual $M'$ of an admissible module $M$. In \zcref{sec:PBW}, we develop the theory of $C_n$-cofiniteness for CFT-type vertex algebras. We recover most of the main results of \cite{KL98,GN03,BuhlSpanning,AN2} for arbitrary base fields.
In \zcref{sec:Aut O}, we study vertex algebras equipped with an action of $\Aut\mathcal{O}$. This generalizes the notion of a quasi-conformal vertex algebra $V$.

In \zcref{part: sheaves}, we define and study the vertex algebra sheaf associated to a vertex algebra with an $\Aut\mathcal{O}$-action over an arbitrary algebraically closed field $\kk$.
In \zcref{sec:vertex alg sheaf}, we define the vertex algebra sheaf $\mathcal{V}_C$ and give a natural construction of its logarithmic $\mathcal{D}$-module structure.
In \zcref{sec:chiralLieAlg}, we study the chiral Lie algebra $\mathcal{L}_{C\setminus P_\bullet}(V)$ associated to a vertex algebra $V$ with an $\Aut\mathcal{O}$-action and a stable, $n$-pointed curve $C$.
In \zcref{sec:coinvariants}, we define the sheaves of coinvariants and conformal blocks associated to a stable relative curve $\mathcal{C}\to S$. We prove analogues of propagation of vacua (\zcref{thm:propagation of vacua}) and the factorization theorems (\zcref{thm:factorization theorem}) in full generality, with no additional assumptions beyond what we have required so far. We also characterize the space of intertwining operators using conformal blocks associated to the smooth relative curve $\P^1_{\kk[z,z^{-1}]} \to \Spec\kk[z,z^{-1}]$ (\zcref{thm:intertwining}).
In \zcref{sec:finiteness-smoothing}, we study certain properties of sheaves of coinvariants given certain assumptions on the underlying vertex algebra. We show that sheaves of coinvariants associated to a $C_2$-cofinite vertex algebra with an $\Aut\mathcal{O}$-action are finite-dimensional over $\kk$, and we show that coinvariants satisfy smoothing as in \cite{DGK23} provided that the mode transition algebras $\Ac_d$ are strongly unital for all $d\in \N$.

\subsection*{Acknowledgements}
We thank Angela Gibney and Danny Krashen for the discussions that led to our main results. The author was supported by the NSF GRFP.

\part{Preliminaries on vertex algebras}\label{part:prelims}
In this part, we give an overview of vertex algebras and their properties over an arbitrary ring. In particular, we will reprove many well-known properties of vertex operator algebras for vertex algebras over an arbitrary ring. Much of this part of the paper is discussed in more generality in \cite{Griffin26}.

\section{Basics}\label{sec:Basics}
In this section, we recall definitions and facts about vertex algebras. Many of our definitions are given over an arbitrary ring $\kk$, but in practice $\kk$ will be a field.

\subsection{Main definitions}
Given $n\in \Z$ and $k\in \N$, we set
\[(z+w)^n = \sum_{k\ge 0}\binom{n}{k}z^{n-k}w^k,\qquad \binom{n}{k} = \frac{n(n-1)\cdots (n-k+1)}{k!}.\]
\begin{definition}
    Let $\kk$ be a commutative unital ring. A \textit{vertex $\kk$-algebra} is a tuple $(V,\vac,Y)$ consisting of a $\kk$-module $V$, an element $\vac\in V$, and a $\kk$-linear map
    \begin{align*}
        Y(\cdot,z) \colon V&\to \End(V)[\![z,z^{-1}]\!],\\
        a&\mapsto Y(a,z) = \sum_{k\in \Z}a_{(k)}z^{-1-k},
    \end{align*}
    satisfying the following properties:
    \begin{enumerate}[label = (V\arabic*)]
        \item \label{V1} For all $a,b\in V$, there exists $K=K_{a,b}\in \Z$ such that $a_{(k)}b = 0$ if $k\ge K$.
        \item \label{V2} $Y(\vac,z)=\id_V$, or equivalently $\vac_{(k)}a = \delta_{k,-1}a$.
        \item \label{V3} For all $a\in V$, we have $Y(a,z)\vac \in a + zV[\![z]\!]$, or equivalently
        \[a_{(-1)}\vac = a,\quad \text{and}\quad \forall k\ge 0,\ a_{(k)}\vac = 0.\]
        \item \label{V4} For all $a,b,c\in V$, we have the \textit{Jacobi identity}:
        \begin{align*}
            &z^{-1}\delta\left(\frac{x-y}{z}\right)Y(a,x)Y(b,y)c - z^{-1}\delta\left(\frac{-y+x}{z}\right)Y(b,y)Y(a,x)c\\
            &=y^{-1}\delta\left(\frac{x-z}{y}\right)Y(Y(a,z)b,y)c,
        \end{align*}
        where $\delta(z) = \sum_{k\in \Z}z^k$.
    \end{enumerate}
\end{definition}
There are many properties of vertex algebras one can prove, such as the following:
\begin{defprop}\label{prop:VA properties}
    Let $V$ be a vertex $\kk$-algebra. We let
    \[e^{zT} = \sum_{n\ge 0}z^nT^{(n)} \in \End(V)[\![z]\!]\]
    be defined so that $e^{zT}a = Y(a,z)\vac$. That is, $T^{(n)}a = a_{(-1-n)}\vac$. We call $\{T^{(n)}\}_{n\ge 0}$ the \textit{(divided power) translation operators}. Then the following identities hold:
    \begin{enumerate}
        \item (Divided power structure): We have $T^{(0)} = \id_V$ and $e^{zT}e^{wT} = e^{(z+w)T}$, or equivalently for all $n,m\ge 0$
        \[T^{(n)}T^{(m)} = \binom{n+m}{n}T^{(n+m)}.\]
        \item (Translation covariance) For all $a\in V$, we have
        \[e^{wT}Y(a,z)e^{-wT} = Y(e^{wT}a,z) = Y(a,z+w)\eqdef e^{w\partial_z}Y(a,z) = \sum_{k\ge 0}w^k\partial_z^{(k)}Y(a,z).\]
        \item (Skew-symmetry) For all $a,b\in V$, we have
        \[Y(a,z)b = e^{zT}Y(b,-z)a.\]
        \item (Commutator formula) For all $a,b\in V$ and $r,s\in \Z$, we have
        \begin{equation}\label{eq:commutator}
            [a_{(r)},b_{(s)}] = \sum_{k\ge 0}\binom{r}{k}(a_{(k)}b)_{(r+s-k)}.
        \end{equation}
        \item (Iterate/associator formula) For all $a,b\in V$ and $r,s\in \Z$, we have
        \begin{equation}\label{eq:associator}
            (a_{(r)}b)_{(s)} = \sum_{k\ge 0}(-1)^k\binom{r}{k}\left(a_{(r-k)}b_{(s+k)} - (-1)^rb_{(r+s-k)}a_{(k)}\right).
        \end{equation}
    \end{enumerate}
\end{defprop}
One may check that the commutator formula and the associator formula together are equivalent to the Jacobi identity.

\begin{remark}\label{D-module remark}
    Define $\cal D$ to be the $\kk$-algebra 
    \[\cal D = \kk\left[D^{(n)} \mid n\ge 0\right]\big/\left(D^{(0)} - 1,\, D^{(n)}D^{(m)} - \tbinom{n+m}{n}D^{(n+m)} \mid n,m\ge 0\right).\]
    Then the above proposition states that the operators $\{T^{(n)}\}_{n\ge 0}$ define a representation of $\cal D$ on $V$. If $\kk$ is an algebra over the finite field $\F_p$ for a prime $p$, then Lucas' theorem implies that $\cal D$ is generated over $\kk$ by the set $\{D^{(p^n)}\}_{n\ge 0}$, and we have $(D^{(p^n)})^p = 0$ for all $n\ge 0$.

    Let $\widetilde{V}$ be the quasi-coherent sheaf corresponding to the $\kk$-module $V$ on $\Spec \kk$. Then a $\widehat{\G}_a$-equivariant structure on $\widetilde{V}$, where $\G_a = \A^1_\kk = \Spec \kk[t]$, is equivalent to a representation of $\cal D$ on $V$.
\end{remark}

Throughout the paper, we will require gradings on vertex algebras:
\begin{definition}\label{Z-grading}
    Let $V$ be a vertex $\kk$-algebra. A \textit{$\Z$-grading} on $V$ is a $\Z$-grading on the underlying $\kk$-module
    $V = \bigoplus_{s\in \Z}V_{s}$
    such that:
    \begin{enumerate}
        \item The vacuum element $\vac$ is of degree 0, that is $\vac\in V_{0}$.
        \item For all $a\in V$ homogeneous of degree $\deg(a)$ and $k\in \Z$, the mode $a_{(k)}$ is a homogeneous operator of degree $\deg(a)-1-k$.
    \end{enumerate}
    We say that a $\Z$-grading on $V$ is \textit{lower-truncated} if there exists $K\in \Z$ such that $V_{s} = 0$ unless $s\ge K$. We say that $V$ is \textit{$\N$-graded} if the grading is lower-truncated with $K=0$.

    We say an $\N$-graded vertex algebra $V$ over $\kk$ is of \textit{CFT-type} if the $\N$-grading $V = \bigoplus_{s\ge 0}V_{s}$ is such that $V_{0} = \kk\vac$.
\end{definition}
While we will mostly consider $\N$-graded vertex algebras, we will not require that they be CFT-type. We will also generally not assume that the graded components are finite-dimensional, though in practice this will often be the case.

One way gradings can arise is as part of a larger structure, like a conformal element.
\begin{definition}
    A \textit{vertex operator algebra} (or \textit{VOA}) over a ring $\kk$ is an $\N$-graded vertex $\kk$-algebra $V$ with a distinguished element $\omega\in V_2$, called the \textit{conformal element}, such that the following hold:
    \begin{enumerate}
        \item Each graded part of $V$ is finitely generated and projective over $\kk$.
        \item Denote by $T(z)$ the field corresponding to $\omega$:
        \[T(z) = Y(\omega,z) = \sum_{n\in \Z}\omega_{(n)}z^{-1-n} = \sum_{m\in \Z}L_{m}z^{-2-m}.\]
        We require that the modes $\{L_m\}_{m\in \Z}$ satisfy the commutator relations for the Virasoro algebra such that the central element $\mathbf{c}$ acts as a fixed scalar $c\in \kk$. We refer to $c$ as the \textit{central charge}.
        \item The mode $L_0$ acts as the grading operator on $V$. That is, $L_0v = \deg(v)v$ for all homogeneous $v$.
        \item The mode $L_{-1}$ acts as the first translation operator $T^{(1)}$.
    \end{enumerate}
\end{definition}
A vertex operator $\kk$-algebra has an obvious action of $\mathfrak{sl}_2(\kk) = \Span_\kk\{L_{-1},L_0,L_1\}$, but this will not be sufficient to define contragredient modules unless $\kk$ is a $\Q$-algebra. We discuss the replacement of this action in \zcref{sec:contragredient}.
\begin{remark}
    Note that the conformal element does not determine the higher-order translation operators $\{T^{(n)}\}_{n\ge 1}$ unless $\kk$ is a $\Q$-algebra. This is one subtle reason that VOAs are often easier to study over $\C$.
\end{remark}

Now we recall modules for vertex algebras.
\begin{definition}
    Let $V$ be a vertex $\kk$-algebra. A \textit{weak $V$-module} is a $\kk$-module $W$ together with a $\kk$-linear map
    \begin{align*}
        Y^W(\cdot,z)\colon V&\to \End(W)[\![z,z^{-1}]\!]\\
        a&\mapsto Y^W(a,z) = \sum_{k\in \Z}a^W_{(k)}z^{-1-k},
    \end{align*}
    satisfying the following properties:
    \begin{enumerate}[label = (M\arabic*)]
        \item For all $a\in V$ and $b\in W$, there exists $K=K_{a,b}\in \Z$ such that $a^W_{(k)}b = 0$ if $k\ge K$.
        \item We have $Y^W(\vac,z) = \id_W$.
        \item For all $a,b\in V$ and $c\in W$, we have
        \begin{align*}
            &z^{-1}\delta\left(\frac{x-y}{z}\right)Y^W(a,x)Y^W(b,y)c - z^{-1}\delta\left(\frac{-y+x}{z}\right)Y^W(b,y)Y^W(a,x)c\\
            &=y^{-1}\delta\left(\frac{x-z}{y}\right)Y^W(Y(a,z)b,y)c.
        \end{align*}
    \end{enumerate}
\end{definition}
\begin{definition}
    Let $V$ be a $\Z$-graded vertex $\kk$-algebra. Let $\Lambda$ be a subset of $\Q$ that is closed under the action of $(\Z,+)$. A \textit{$\Lambda$-grading} on a weak $V$-module $W$ is a $\Lambda$-grading on the underlying $\kk$-module $W = \bigoplus_{\ell\in \Lambda}W_{\ell}$ such that for all $a\in V$ homogeneous of degree $\deg(a)$ and $k\in \Z$, the mode $a^W_{(k)}$ is a homogeneous operator of degree $\deg(a)-1-k$ on $W$. That is, $a_{(k)}^W\cdot W_\ell \subset W_{\ell + \deg(a)-1-k}$ for all $\ell\in \Lambda$.

    We say $W$ is \textit{admissible} if $W$ admits an $\N$-grading, that is, a $\Z$-grading such that $W_{s} = 0$ unless $s\ge 0$.
\end{definition}
We will be primarily concerned with the case where $\Lambda = c_W + \Z$ for some $c_W\in \Q$. We say that $W$ has \textit{conformal dimension $c_W$} when this is the case. 

Every lower-truncated $\Z$-grading on a weak $V$-module $W$ may be shifted so that we may regard it as an admissible $V$-module such that $W_0\neq 0$. The morphisms in the category of admissible $V$-modules are weak $V$-module morphisms. That is, a morphism $W^1\to W^2$ must intertwine the operators $Y^{W^i}(a,z)$, but it need not respect the grading (but in practice it usually does).

\begin{notation}
    Given a $\Lambda$-graded $\kk$-module $W$, it admits a canonical left filtration given by $\mrm G_\ell W = \bigoplus_{s\le \ell}W_{s}$ for $s\in \Z$. This filtration is exhaustive and separated. There is a right filtration given by replacing ``$s\le \ell$'' with ``$s\ge \ell$'' in the direct sum.
\end{notation}
From here on in the paper, we assume that all $\Z$-gradings are lower-truncated.

\begin{definition}\label{def:rational}
    Let $V$ be a $\Z$-graded vertex $\kk$-algebra over a field $\kk$. We say that $V$ is \textit{quasi-rational} if every admissible $V$-module is semisimple (a direct sum of simple admissible $V$-modules). We say that $V$ is \textit{rational} if it is quasi-rational and the graded components of every simple admissible $V$-module are finite-dimensional over $\kk$.
\end{definition}
Being rational is a much more useful condition than being quasi-rational. These conditions are known to be equivalent for a VOA over $\C$ \cite[Theorem 8.1(c)]{DLM1}. We will discuss the case over an arbitrary algebraically closed field in \zcref{sec:Aut O}.

\subsection{The universal enveloping algebra}
Now we discuss the (finite) ancillary Lie algebra associated with a vertex $\kk$-algebra $V$. This Lie algebra and its associated constructions control the module theory for $V$.
\begin{definition}
    Given a vertex $\kk$-algebra $V$, we define the \textit{finite, left, and right ancillary Lie algebras} as
    \begin{align*}
        \LVf &= \frac{V\otimes_{\kk}\kk[t,t^{-1}]}{\left(\Im \nabla^{(n)}\right)_{n> 0}}\cong \frac{V\otimes_{\kk}\kk[t,t^{-1}]}{\left(\Im\left(\id\otimes \partial_{t}^{(n)} - (-1)^nT^{(n)}\otimes \id\right)\right)_{n> 0}},\\
        \LVL &= \frac{V\otimes_{\kk}\kk(\!(t)\!)}{\left(\Im \nabla^{(n)}\right)_{n> 0}}\cong \frac{V\otimes_{\kk}\kk(\!(t)\!)}{\left(\Im\left(\id\otimes \partial_{t}^{(n)} - (-1)^nT^{(n)}\otimes \id\right)\right)_{n> 0}},\\
        \LVR &= \frac{V\otimes_{\kk}\kk(\!(t^{-1})\!)}{\left(\Im \nabla^{(n)}\right)_{n> 0}}\cong \frac{V\otimes_{\kk}\kk(\!(t^{-1})\!)}{\left(\Im\left(\id\otimes \partial_{t}^{(n)} - (-1)^nT^{(n)}\otimes \id\right)\right)_{n> 0}},
    \end{align*}
    where $\nabla^{(n)} = \sum_{0\le k\le n}T^{(k)} \otimes \partial_{t}^{(n-k)}$.
\end{definition}
The $\kk$-module isomorphism above follows from iteratively substituting the identity $\nabla^{(n-j)}\equiv 0$ into $\nabla^{(n)}$.
We write $a_{[n]}$ for the image of $a\otimes t^{n}$ in the above spaces.
The Lie brackets on $\LVf$, $\LVL$, and $\LVR$ are given by
\begin{equation}
    [a\otimes f(t),b\otimes g(t)] = \sum_{k\ge 0}a_{(k)}b \otimes \partial_t^{(k)}[f(t)]g(t)
\end{equation}
for $a,b\in V$ and $f,g\in \kk[t,t^{-1}]$, $\kk(\!(t)\!)$, and $\kk(\!(t^{-1})\!)$ respectively.
For $\LVf$ specifically, this bracket is equivalently given by the formula
\begin{equation}\label{Lie bracket}
[a_{[n]},b_{[m]}] = \sum_{k\ge 0}\binom{n}{k}(a_{(k)}b)_{[n+m-k]},
\end{equation}
where $a,b\in V$ and $n,m\in \Z$. This bracket satisfies skew-symmetry and the Jacobi identity. 
\begin{warning}
    We will generally refer to $\LVf$, $\LVL$, $\LVR$, and related constructions as Lie algebras even though Lie algebras are usually required to satisfy $[x,x]=0$ for all $x$, which is a slightly stronger condition. This is equivalent to skew-symmetry if 2 is invertible in $\kk$, but not in general.
\end{warning}

We leave the following observation from \cite[\textsection 2.2]{DGK23}:
\begin{lemma}
    Let $V$ be a $\Z$-graded vertex $\kk$-algebra. Then the Lie algebras $\LVf$, $\LVL$, and $\LVR$ admit a natural $\Z$-grading, left filtration, and right filtration respectively by setting $\deg(a_{[k]}) = \deg(a)-1-k$ for $a\in V$ homogeneous.
\end{lemma}
There is a natural notion of a universal enveloping algebra $\UV$ associated to $V$, which we briefly describe here.
Let $U$ be the universal enveloping algebra of $\LVf$. This naturally inherits a $\Z$-grading from the $\Z$-grading on $\LVf$. The \textit{left and right canonical seminorms} on $U$ are defined to be
\[\NL[m]U = UU_{\le -m} = \sum_{i\le -m}UU_i,\qquad \NR[m]U = U_{\ge m}U = \sum_{i\ge m}U_iU.\]
These both define exhaustive right filtrations of $U$, though their roles are distinct.
We can restrict these seminorms to various filtered and graded parts of these algebras in a canonical way. We write $\NLR[d]_{\bullet}U_{\le p} = (\NLR[d]_{\bullet}U)\cap U_{\le p}$ and similarly so for $\NLR[d]_{\bullet}U_{p}$:
\[
    \NL[m]\Uu_{p} = (\Uu \Uu_{\le -m})_{p} = \sum_{j\le -m} \Uu_{p-j}\Uu_{j}, 
    \qquad \NR[m]\Uu_{p} = (\Uu_{\ge m} \UuR)_{p} = \sum_{i\ge m} \Uu_{i}\Uu_{p-i}.
\]
We have $\NR[m+p]\Uu_p = \NL[m] \Uu_p$. Restricting the seminorm to these subspaces, we may define a graded completion of $\Uu$:
Specifically, we define
\[
\hUu_d := \varprojlim_m \dfrac{\Uu_{ d}}{\NL[m]\Uu_{d}} = \varprojlim_m \dfrac{\Uu_{ d}}{\NR[m+d]\Uu_{d}},\qquad \hUu := \bigoplus_d \hUu_d.
\]
We set $J$ to be the two-sided ideal in $\wh{U}$ generated by the Jacobi relations and the vacuum relations. Let $\overline J$ be the respective topological closure of $J$. The resulting quotient algebra
\begin{equation}\label{eq:UV}
    \UV = \hUu/\overline J
\end{equation}
is a graded associative algebra with an almost canonical seminorm.
\begin{definition}
    We call $\UVf = \UV = \UV(V)$ from \zcref[noname]{eq:UV} the \textit{(finite) universal enveloping algebra} of $V$.
\end{definition}
We refer the reader to \cite[\textsection 2--3]{DGK23} to see a more explicit construction of $\UV = \UVf$ and the left and right completions $\UVL$ and $\UVR$.
\begin{lemma}\label{lem:weak V mods equiv}
    Let $V$ be a $\Z$-graded vertex $\kk$-algebra. The category of weak $V$-modules is equivalent to the category of continuous left $\UV$-modules.
\end{lemma}

\subsection{The Zhu algebras and mode transition algebras}
Here we define the Zhu algebras and mode transition algebras associated to a $\Z$-graded vertex algebra $V$ over $\kk$. The Zhu algebras are usually defined as a quotient of $V$ (\cite{DLM2, Ren17}), but we opt to define these as quotients of the degree zero part of the universal enveloping algebra $\UV$.

\begin{definition}\label{Zhu algs}
Let $V$ be a $\Z$-graded vertex $\kk$-algebra.
The $d$-th \textit{(higher) Zhu algebra} is the quotient $\kk$-algebra
\begin{equation}
    \Aa_d(V) = \UV_0\big/\NLR[d+1]\UV_0.
\end{equation}
If $V$ is clear in context, we will just write $\Aa_d$ instead of $\Aa_d(V)$. We also write $\Aa_0 = \Aa$, and we refer to the zeroth Zhu algebra as just the \textit{Zhu algebra}. There are natural surjections $\Aa_d\twoheadrightarrow \Aa_{d-1}$ for all $d\ge 1$, and $\Aa_d = 0$ unless $d\ge 0$.
\end{definition}
This definition is equivalent to the ``quotient of $V$'' construction $A_d(V)$, which is more common in the literature. See \cite{Griffin26}, \cite{HeHigherZhu}, or \cite{Xu2026} for a proof of this. In this paper, we are primarily concerned with the zeroth Zhu algebra $\Aa(V)$.
Here is the main reason we study Zhu algebras:
\begin{lemma}
    Let $V$ be a $\Z$-graded vertex $\kk$-algebra, and let $M$ be an admissible $V$-module. Then the degree $d$ part $M_d$ canonically has the structure of an $\Aa_d$-module via the action of $\UV_0$.
\end{lemma}
\begin{proof}
    This comes from the fact that $M$ is equivalently a continuous $\N$-graded module. The image of the action of $\NLR[d+1]\UV_0 = \UV\UV_{\le -d-1}\cap \UV_0$ is then zero on $M_d$.
\end{proof}
There are several functors relating $V$-modules and modules for the above algebras. See \cite{Griffin26} for more details.
\begin{definition}
    Let $n,m\ge 0$. We define the \textit{Zhu bimodule} $\Aa_{n,m}(V)$ to be
    \[\Aa_{n,m}(V) = \UV_{n-m}/\NL[m+1]\UV_{n-m}.\]
\end{definition}
Note that $\Aa_d(V) = \Aa_{d,d}(V)$ for all $d\ge 0$, and $\Aa_{n,m} = 0$ unless $n\ge 0$ and $m\ge 0$.
\begin{lemma}
    For all $n,m,\ell\ge 0$, there is a natural morphism of $(\Aa_{n},\Aa_{\ell})$-bimodules
    \[\mu_{n,\ell}^m \colon \Aa_{n,m}\otimes_{\Aa_m}\Aa_{m,\ell} \to \Aa_{n,\ell}\]
    given by multiplication in $\UV$.
\end{lemma}
The maps above give $\Aa_{n,m}(V)$ the structure of a $(\Aa_n(V),\Aa_m(V))$-bimodule.

\begin{definition}
    We define the \textit{mode transition algebra}
    \[\Ac(V) = \bigoplus_{n,m\ge 0}\Ac_{n,-m}(V) = \bigoplus_{n,m\ge 0}\Aa_{n,0}\otimes_{\Aa_0}\Aa_{0,m}.\]
    We denote $\Ac_n \eqdef \Ac_{n,-n}$ for all $n\ge 0$.
\end{definition}
In principle, one could study the tensor products $\Aa_{n,d}\otimes_{\Aa_d}\Aa_{d,m}$ for any $d\ge 0$, but for our purposes we only need $d=0$.
\begin{definition}
    Given an $\Aa_d$-module $\mathsf{M}$, we define the \textit{induced Verma module} $\PhiL_d(\mathsf{M})$ to be
    \[\PhiL_d(\mathsf{M}) = \UV/\NL[d+1]\UV \otimes_{\Aa_d(V)}\mathsf{M}.\]
    We assign the grading 
    \[\PhiL_d(\mathsf{M})_{m} = \UV_{m-d}/\NL[d+1]\UV_{m-d} \otimes_{\Aa_d(V)}\mathsf{M} = \Aa_{m,d}\otimes_{\Aa_d(V)}\mathsf{M},\]
    and it is naturally a continuous left $\UV$-module. We have $\PhiL_d(\mathsf{M})_{m} = 0$ for $m < 0$.
\end{definition}
We will only need to consider $d=0$.

\begin{defprop}\label{def:stronglyunital}
    For all $n,m,d\ge 0$, there are natural morphisms of $(\Aa_{n},\Aa_{m})$-bimodules
    \[\star\colon \Ac_{n,-d}\otimes_{\Aa_d}\Ac_{d,-m} \to \Ac_{n,-m},\qquad \mu_{n,m}\colon \Ac_{n,-m} \to \Aa_{n,m}\]
    given by multiplication in $\UV$.
    We refer to the first morphism as the \textit{star product}. We say that $\Ac_n$ is \textit{strongly unital} if there is an element $\one_n \in \Ac_n$, called a \textit{strong identity element}, such that for all $m\ge 0$, $\frak{a} \in \Ac_{m,-n}$, and $\frak{b} \in \Ac_{n,-m}$ we have
    \[\frak{a} \star \one_n = \frak{a},\qquad \one_m\star\frak{b} = \frak{b}.\]
    We say that $V$ satisfies the \textit{strong identity condition} if $\Ac_d$ is strongly unital for all $d\ge 0$.
\end{defprop}
Mode transition algebras are very useful for studying the higher Zhu algebras.
\begin{proposition}[\cite[Theorem B.3.3]{DGK23}]\label{Zhu decomposition unital}
    Suppose that $\Ac_d$ admits an identity element (not necessarily strong). Then there is a $\kk$-algebra decomposition
    \[\Aa_d \cong \Ac_d \times \Aa_{d-1}.\]
\end{proposition}
Another way of expressing strong identity elements is using a set of equations.
\begin{definition}\label{def:strong identity eqns}
We say that a sequence $(\one_d)_{d\in \N}$, with $\one_d\in \Ac$, satisfies the \textit{strong identity equations} if for every homogeneous $a\in V$ and $n\in \Z$ such that $d+n\ge 0$, we have
\begin{equation}\label{strong identity eqns}
 J_{-n}(a) \one_d = \one_{d+n}J_{-n}(a).
\end{equation}
\end{definition}
As the name suggests, the strong identity equations are closely related to strong identity elements. We omit proofs of the following results since the same proof as the original source works here without modification.
\begin{lemma}[\cite[Lemma 5.1.5]{DGK23}]\label{lem:strong identity eqns equiv}
Suppose we have a collection of elements $\one_d\in \Ac_d$ for each $d\ge 0$, with $\one_0 = 1_{\Ac_0}$. Then $\one_d$ is a strong identity element in $\Ac_d\subset \Ac$ for all $d\in \N$ if and only if the sequence $(\one_d)_{d\in \N}$ satisfies the strong identity equations.
\end{lemma}
We do not need to assume anything about the underlying vertex algebra $V$, aside from the grading that is necessary to define the universal enveloping algebra $\UV$.
Note that in \zcref{def:strong identity eqns} there is no assumption on the bidegrees of $\one_d$. One may resolve this as follows:
\begin{lemma}[\cite[Lemma 5.1.4]{DGK23}]\label{lem:strong identity eqns reparam}
     If the sequence $(\one_d)_{d\in \N}$ with $\one_d\in \Ac$ satisfies the strong identity equations, then so does the sequence $(\one_d')_{d\in \N}$, where $\one_d' \eqdef (\one_d)_{d,-d}$.
\end{lemma}

\section{Contragredient duals and symmetric bilinear forms}\label{sec:contragredient}
The purpose of this section is to give a generalization of some of the constructions of \cite{Li18} from an algebraically closed base field to any base ring. Most of our statements here will be given without proof since one may apply the original paper's arguments verbatim.

\subsection{Main definitions}\label{sec:H-mod}
Consider the Lie algebra $\mathfrak{sl}_2(\kk)$ over a ring $\kk$ with the basis $\{L_{-1},L_0,L_1\}$, satisfying the relations
\[[L_1,L_{-1}] = 2L_0,\quad [L_0,L_{\pm 1}] = \mp L_{\pm 1}.\]
Given $n\in \N$, we set
\[L_{\pm 1}^{(n)} = \frac{L_{\pm 1}^n}{n!},\qquad L_0^{(n)} = \binom{-2L_0}{n}.\]
We define $U(\mathfrak{sl}_2)_\Z$ to be the noncommutative ring generated by the symbols $\{L_{\pm 1}^{(n)},L_{0}^{(n)}\}_{n\in \N}$ with the relations
\begin{align*}
    L_{\pm 1}^{(0)} &= L_0^{(0)} = 1,\\
    L_0^{(m)}L_0^{(n)} &= \sum_{j=0}^m\binom{m}{j}\binom{n+j}{m}L_0^{(n+j)},\\
    L_{\pm 1}^{(m)}L_{\pm 1}^{(n)} &= \binom{m+n}{n}L_{\pm 1}^{(m+n)},\\
    L_0^{(m)}L_{\pm 1}^{(n)} &= L_{\pm 1}^{(n)}\binom{-2L_0\pm 2n}{m} = \sum_{i=0}^m\binom{\pm 2n}{i}L_{\pm 1}^{(n)}L_0^{(m-i)},\\
    L_1^{(m)}L_{-1}^{(n)} &= \sum_{i=0}^{\min\{m,n\}}\sum_{j=0}^i \binom{-m-n+2i}{j}(-1)^i L_{-1}^{(n-i)}L_0^{(i-j)}L_1^{(m-i)},\\
    L_{\pm 1}^{(m)}L_0^{(n)} &= \sum_{i=0}^n\binom{\mp 2m}{i}L_0^{(n-i)}L_{\pm 1}^{(m)}.
\end{align*}
The last identity is redundant. We also define $\mathcal{H} = U(\mathfrak{sl}_2)_\Z\otimes_\Z\kk$.
We define the generating series
\[e^{zL_{\pm 1}} = \sum_{k\ge 0}z^kL_{\pm 1}^{(k)},\qquad (1+z)^{-2L_0} = \sum_{k\ge 0}z^kL_0^{(k)}.\]
It is helpful to write the above identities in terms of these generating series. The identities in the second, third, fourth, and fifth lines above are respectively equivalent to the following:
\begin{align*}
    (1+z)^{-2L_0}(1+w)^{-2L_0} &= (1+z+w+zw)^{-2L_0},\\
    e^{zL_{\pm 1}}e^{wL_{\pm 1}} &= e^{(z+w)L_{\pm 1}},\\
    (1+z)^{-2L_0}e^{wL_{\pm 1}} &= e^{(1+z)^{\mp 1}wL_{\pm 1}}(1+z)^{-2L_0},\\
    e^{xL_1}e^{zL_{-1}} &= e^{(1-xz)^{-1}zL_{-1}}(1-xz)^{-2L_0}e^{(1-xz)^{-1}xL_1}.
\end{align*}
The algebra $\mathcal{H}$ also has the structure of a bialgebra. For $k\in \{-1,0,1\}$ and $n\in \N$, we set
\begin{equation*}
    \varepsilon(L_{k}^{(n)}) = \delta_{n,0},\qquad
    \Delta(L_{k}^{(n)}) = \sum_{i=0}^nL_{k}^{(n-i)}\otimes L_{k}^{(i)}.
\end{equation*}
Furthermore, there is a natural anti-automorphism $\theta\colon \mathcal{H}\to \mathcal{H}^\op$ given by
\[\theta(L_{\pm 1}^{(n)}) = L_{\mp 1}^{(n)},\quad \theta(L_0^{(n)}) = L_0^{(n)}\quad \forall n\in \N.\]
There is also a natural $\Z$-grading on $\mathcal{H}$ given by $\deg(L_{\pm 1}^{(n)}) = \mp n$ and $\deg(L_0^{(n)}) = 0$ for all $n\in \N$.

Now we study modules for $\mathcal{H}$. While all modules we consider are admissible (i.e., $\N$-graded), it may not be the case that $L_0^{(n)}$ respects the $\N$-grading in the obvious sense. Thus, it is better to work with modules equipped with a $\Lambda$-grading, where again $\Lambda$ is a subset of $\Q$ that is closed under the action of $(\Z,+)$. For example, $\Lambda$ could be $c_W+\Z\subset \Q$, so $W$ has conformal dimension $c_W\in \Q$.
\begin{definition}
    A \textit{$\Lambda$-graded weight $\mathcal{H}$-module} is a $\Lambda$-graded $\mathcal{H}$-module $W = \bigoplus_{\ell\in \Lambda}W_\ell$ such that $L_0^{(n)}$ acts as $\binom{-2\deg}{n}$ for all $n\in \N$ in a well-defined manner. That is, for all $\ell\in \Lambda$ such that $W_\ell\neq 0$, the denominator of $\ell$ is invertible in $\kk$ and
    \[L_0^{(n)}|_{W_\ell} = \binom{-2\ell}{n}.\]
    If $W_\ell = 0$, then we set $L_0^{(n)}\cdot W_\ell = 0$.
\end{definition}
\begin{definition}
    A \textit{M\"obius vertex $\kk$-algebra} (or \textit{$\mathcal{H}$-module vertex $\kk$-algebra}) is a $\Z$-graded vertex $\kk$-algebra $V$ with the structure of a $\Z$-graded weight $\mathcal{H}$-module such that:
    \begin{enumerate}
        \item $V_n=0$ for $n\ll 0$.
        \item $L_1^{(n)}\vac = \varepsilon(L_1^{(n)})\vac = \delta_{n,0}\vac$.
        \item For $a\in V$, we have
        \[e^{zL_1}Y(a,z_0)e^{-zL_1} = Y\left(e^{z(1-zz_0)L_1}(1-zz_0)^{-2L_0}a,\frac{z_0}{1-zz_0}\right).\]
    \end{enumerate}
\end{definition}
We can use these structures to define contragredient modules.
\begin{definition}\label{contragredient}
    Let $V = \bigoplus_{r\ge K}V_{r}$ be a M\"obius vertex $\kk$-algebra. Let $W = \bigoplus_{\ell\in \Lambda} W_{\ell}$ be a $\Lambda$-graded $V$-module whose grading is lower-truncated. We define its \textit{contragredient module} to be the graded dual $W' = \bigoplus_{\ell\in \Lambda} W_{\ell}^\lor$, where $W_{\ell}^\lor = \Hom_{\kk}(W_{\ell},\kk)$, and
    \[Y^{W'}(-,z)\colon V\to \End(W')[\![z,z^{-1}]\!]\]
    is the linear map determined by the relationship
    \[\left\langle Y^{W'}(a,z)\psi,w\right\rangle = \left\langle \psi,Y^{W}(e^{zL_1}(-z^{-2})^{L_0}a,z^{-1})w\right\rangle\]
    for $a\in V$, $\psi\in W'$, $w\in W$, where $\langle \cdot,\cdot\rangle$ is the natural dual pairing. If $a\in V$ is homogeneous, then this identity reads as
    \[\langle a^{W'}_{(n)}\psi,w\rangle = (-1)^{\deg(a)}\sum_{i\ge 0}\langle\psi,(L_1^{(i)}a)^W_{(2\deg(a)-n-i-2)}w\rangle.\]
\end{definition}
Since $L_1^{(i)}$ is a degree $-i$ operator and the $\Z$-grading on $V$ is assumed to be lower-truncated, for $a\in V$ and $i \gg 0$ we have $L_1^{(i)}a = 0$. This shows that the above expression is well-defined. We also have the following results from \cite{FHL} as their proof methods work identically.
\begin{lemma}[\cite[Proposition 3.8]{Li18}, \cite[Theorem 5.2.1]{FHL}]
    The above data defines a $\Lambda$-graded $V$-module $W'$ over $\kk$ whose grading is lower-truncated.
\end{lemma}
\begin{proposition}[\cite[Proposition 5.3.1]{FHL}]
    Given an admissible $V$-module $W$, there is a natural map of $V$-modules $W\to W''$. If $\kk$ is a field, then this morphism is an isomorphism provided that the graded components of $W$ are all finite-dimensional over $\kk$.
\end{proposition}
The results we cite in the above statements are not stated in our generality, but the same proof arguments work.

Now we consider $V$-modules with a compatible $\mathcal{H}$-module structure.
\begin{definition}
    Let $V$ be a M\"obius vertex $\kk$-algebra. A \textit{$\Lambda$-graded $(V,\mathcal{H})$-module} is a $\Lambda$-graded weight $\mathcal{H}$-module $W = \bigoplus_{\ell\in \Lambda}W_\ell$ with the structure of a $\Lambda$-graded $V$-module such that the following hold:
    \begin{align*}
        e^{zL_{-1}}Y^W(a,x)e^{-zL_{-1}} &= Y^W(e^{zL_{-1}}a,x)\\
        e^{zL_1}Y^W(a,z_0)e^{-zL_1} &= Y^W\left(e^{z(1-zz_0)L_1}(1-zz_0)^{-2L_0}a,\frac{z_0}{1-zz_0}\right).
    \end{align*}
\end{definition}
From the definition, a M\"obius vertex $\kk$-algebra $V$ is naturally a $\Z$-graded $(V,\mathcal{H})$-module.

\subsection{Involutions}\label{sec:involution}
Given a M\"obius vertex $\kk$-algebra $V$, we have a linear map
\[\theta\colon V\otimes_{\kk}\kk[t,t^{-1}]\to V\otimes_{\kk}\kk[t,t^{-1}],\]
which is given for homogeneous elements $a\in V$ by
\[a\otimes \sum_{i}c_it^i \mapsto (-1)^{\deg(a)}\sum_{j\ge 0} L_1^{(j)}a\otimes \sum_{i}c_i t^{2\deg(a)-i-j-2},\]
and extended $\kk$-linearly. The above infinite sum is actually a finite sum since $L_1^{(j)}$ is a degree $-j$ operator, and the grading on $V$ is assumed to be lower-truncated.
The map $\theta$ is related to the involution $\gamma = e^{L_1}(-1)^{L_0}\colon V\to V$ defined on homogeneous elements $a\in V$ by
\[a\mapsto (-1)^{\deg(a)}\sum_{k\ge 0}L_1^{(k)}a,\]
again extended linearly. To state the relationship, we define for $a\in V$ homogeneous
\[J_n(a) = a_{[\deg(a)-1+n]} = \overline{a\otimes t^{\deg(a)-1+n}}.\]
\begin{lemma}[\cite[Lemma 3.4.2]{DGK23}]
    For $a\in V$ homogeneous, we have $\theta(J_n(a)) = J_{-n}(\gamma(a))$. Equivalently,
    \[\theta(a_{[j]}) = (-1)^{\deg(a)}\sum_{i\ge 0}(L_1^{(i)}a)_{[2\deg(a)-j-i-2]}.\]
\end{lemma}
\begin{lemma}[\cite[Lemma 3.4.3]{DGK23}]
    The map $\theta$ defines a Lie algebra involution on $\LVf$ such that
    \begin{align*}
        \theta(\LVf_{d}) &= \LVf_{-d}.
    \end{align*}
\end{lemma}
\begin{lemma}[\cite[Lemma 3.4.4]{DGK23}]
    The Lie algebra involution $\theta\colon \LVf\to \LVf$ induces a $\kk$-algebra involution $U(\LVf)\to U(\LVf)^\op$, which extends to an involution of the universal enveloping algebra $\UV$ of $V$.
\end{lemma}
This involution is precisely what allows us to identify contragredient duals as left $V$-modules instead of right $V$-modules.

\section{PBW-type spanning sets}\label{sec:PBW}
The following section is necessary to prove a specific result about the coherence of sheaves of coinvariants in positive characteristic, namely \zcref{prop:fin-dim-coinvars} and \zcref{cor:CoherenceOnMgn}.
In the construction and analysis of coinvariants over $\C$, one often uses facts about the structure of $V$ as a vector space given certain finiteness assumptions. These are the $C_n$-cofiniteness properties, with the most important cases being $n=1$ and $n=2$. We will try to avoid these assumptions to make our results as general as possible, but there are times where it is difficult to avoid them.
We will reprove some standard results about $C_1$- and $C_2$-cofiniteness for vertex operator algebras over $\C$ for a general CFT-type vertex algebra $V$ over a ring $\kk$. We will not make any finiteness assumptions about the graded components unless stated otherwise. In many cases, the proofs carry over with little to no modification. For brevity, we will omit proofs when the original argument works.

\subsection{On $C_1$-cofiniteness}
In this section, we generalize several results of \cite{KL98} on $C_1$-cofiniteness to the setting of a CFT-type vertex algebra over an arbitrary ring $\kk$. 
In \cite{KL98}, they write $a(n)$ instead of $a_{(n)}$ for the $n$-th mode associated to $a\in V$.
\begin{definition}
    Let $V$ be a CFT-type vertex $\kk$-algebra, and let $V_+ = \bigoplus_{s\ge 1}V_{s}$. We define
    \[C_1(V) = \Span_\kk\{a_{(-1)}b,\ T^{(k)}c \mid a,b\in V_+,\ c\in V,\ k\ge 1\}\subseteq V_+.\]
    We say that $V$ is $C_1$-cofinite if $V/C_1(V)$ is finitely generated over $\kk$.
\end{definition}
\begin{example}
    Consider the Heisenberg vertex $\kk$-algebra $\pi$. It is straightforward to see that $\alpha_{(-n)}\vac \in C_1(\pi)$ for $n\ge 2$. Moreover, for any $v\in V_+$ we have $\alpha_{(-n)}v \in C_1(V)$ for all $n\ge 1$. With these observations, we can see that $\pi/C_1(\pi)$ is freely generated over $\kk$ by $\vac$ and $\alpha = \alpha_{(-1)}\vac$, hence it is $C_1$-cofinite.
\end{example}

Let $U$ be a graded subspace of $V_+$. We define $P(U)$ to be the subspace of $V$ generated by elements of the form
\begin{equation}\label{eq:PBW-type}
    a = a^1_{(-k_1)}\cdots a^r_{(-k_r)}\vac,
\end{equation}
where $r\in \N$, $a^i\in U$, and $k_i\ge 1$. We define a filtration on $P(U)$ given by setting $P(U)_{\le s}$ to be the subspace spanned by elements of the above form with $r\le s$.
\begin{lemma}[\cite[Lemma 3.3]{KL98}]\label{lem:KLlem1}
    Let $V$ be CFT-type. Suppose that $\sum_{i=0}^nV_{i}\subseteq P(U)$ for some $n\in \N$. Let $a,c\in V_+$ and $r\ge 1$ be such that $a_{(-r)}c \in V_{n+1}$. Then $a_{(-r)}c \in P(U)$.
\end{lemma}
\begin{proof}
    The proof is the same as the original proof over $\C$. We leave it here for convenience.
    
    First note that since $\deg(a)+\deg(c) + r-1 = n+1$ and $\deg(a),\deg(c) > 0$, we have $\deg(a),\deg(c)\le n$, so $a,c\in P(U)$.
    
    We use induction on $s$ to prove that $a_{(-r)}c\in P(U)$ for $a\in P(U)_{\le s}$ and $r\ge 1$. It is clearly true for $s=0,1$, so assume it is true for a given $s\ge 1$ and let $a\in P(U)_{\le s+1}$. By the definition of $P(U)$, without loss of generality we may assume that $a = a'_{(-k)}b$ for $a'\in U$, $b\in P(U)_{\le s}$, and $k\ge 1$. By the associator formula \zcref[noname]{eq:associator}, we have
    \begin{align*}
        (a'_{(-k)}b)_{(-r)}c &= \sum_{j\ge 0}\binom{-k}{j}\left((-1)^j a'_{(-k-j)}b_{(-r+j)} - (-1)^{-k+j}b_{(-k-r-j)}a'_{(j)}\right)c.
    \end{align*}
    Note that $a'\in U\subseteq V_+$ and
    \[n+1 = \deg(a_{(-r)}c) = \deg(a'_{(-k-j)}b_{(-r+j)}c) = (\deg(a')+k+j-1) + \deg(b_{(-r+j)}c),\]
    so we have $\deg(b_{(-r+j)}c) \le n$. But $b_{(-r+j)}c \in P(U)$ by assumption, so we must have $a'_{(-k-j)}b_{(-r+j)}c \in P(U)$.
    
    Similarly, we have
    \[n+1 = \deg(b_{(-k-r-j)}a'_{(j)}c) = (\deg(b) + k+r+j - 1) + \deg(a'_{(j)}c),\]
    so $\deg(a'_{(j)}c)\le n$. By the induction hypothesis, we have $b_{(-k-r-j)}a'_{(j)}c \in P$. This completes the induction step, so we are done.
\end{proof}
\begin{proposition}[\cite[Proposition 3.4, Theorem 3.5]{KL98}]\label{strongly finitely generated}
    Let $U$ be a graded subspace of $V_+$. Then $V_+ = U+C_1(V)$ if and only if $V=P(U)$.
\end{proposition}
\begin{proof}
    The proof is more or less the same as the original proof over $\C$, but we replace $\frac{1}{n!}T^n$ with $T^{(n)}$ since we are over an arbitrary ring.
    
    Assume that $V_+ = U+C_1(V)$.
    Clearly we have $V_{0} = \kk \vac \subseteq P$. Now assume $\sum_{j=0}^nV_{j}\subset P$ for $n\ge 0$. Let $w \in V_{n+1}$. Since $V_+ = U + C_1(V)$, we may write $w\in C_1(V)$ as a sum of elements of the form $a_{(-1)}b$ for $a,b\in V_+$ and $T^{(j)}w'$ for $w'\in V_{n+1-j}$. By \zcref{lem:KLlem1}, we have $a_{(-1)}b \in P$. By the induction assumption, we have $w'\in P$. By translation covariance, we have $T^{(j)}P \subseteq P$ for all $j\ge 1$. This shows the inductive step, so $V = P(U)$.

    Conversely, suppose that $V = P(U)$, and consider an element of the form
    \[a = a^1_{(-k_1)}\cdots a^r_{(-k_r)}\vac \in P(U).\]
    Let $b= a^2_{(-k_2)}\cdots a^r_{(-k_r)}\vac$ so that $a = a^1_{(-k_1)}b$ (if $r=1$, then $b=\vac$). Then we have
    \[a^1_{(-k_1)}b = (T^{(k_1-1)}a^1)_{(-1)}b.\]
    If $r\ge 2$, then $a \in C_1(V)$. If $r=1$, then $a = T^{(k_1-1)}a^1$. If $k_1\ge 2$, then $a\in C_1(V)$ again; otherwise, we have $a\in U$. This shows the other direction.
\end{proof}


We also have the following properties of the universal enveloping algebra $\UV = \UV(V)$. The proof is the exact same as it is done over $\C$.
\begin{theorem}[\cite[Theorem 3.10]{KL98}]
    Let $U$ be a graded subspace of $V$ such that $V = P(U)$, and let $\{u^i\}_{i\in I}$ be a generating set for $U$. Given any total order $\le$ on the set of symbols
    \[\{u^i_{(m)} \colon i\in I,\ m\in \Z\},\]
    the subspace of $\UV$ spanned by all nonincreasing monomials\footnote{As in \cite{KL98}, we say that a string $X = a^{i_1}_{(n_1)}\cdots a^{i_r}_{(n_r)}$ is \emph{nonincreasing} if $a^{i_1}_{(n_1)}\ge \cdots \ge a^{i_r}_{(n_r)}$ with respect to the given ordering $\le$.} is dense in $\UV$.
\end{theorem}
\begin{corollary}[\cite[Corollary 3.12]{KL98}]\label{cor:lowest weight spanning}
    Any $\N$-graded $V$-module $W$ that is generated by the $V$-action on $W_{0}$ is generated by elements of the form
    \[u^{i_1}_{(n_1)}\cdots u^{i_r}_{(n_r)} w,\]
    where $w\in W_{0}$ and $u^{i_j} \in U$ with $\deg (u^{i_j}_{(n_j)}) > 0$.
\end{corollary}

\subsection{On $C_2$-cofiniteness}
The purpose of this section is to reprove the results of \cite{BuhlSpanning} and \cite{GN03} on $C_2$-cofinite vertex algebras over a general ring $\kk$. In order to keep the discussion brief, we will generally omit proofs, as most of their arguments apply in our more general setting with no modification.
\begin{definition}
    For $n\ge 2$, we define
    \begin{equation}\label{Cn(V)}
        C_n(V) = \Span_{\kk}\{a_{(-k)}b \mid a,b\in V,\ k\ge n\}.
    \end{equation}
    We say that $V$ is \textit{$C_n$-cofinite} if the quotient $\kk$-module $V/C_n(V)$ is finitely generated over $\kk$.
\end{definition}
\begin{remark}
    Note that $C_n(V)\subset C_m(V)$ for $n\ge m$, so there is a surjection $V/C_n(V) \twoheadrightarrow V/C_m(V)$. Therefore, if $V$ is $C_n$-cofinite, then it is $C_m$-cofinite.
\end{remark}

Here we outline the contents of Buhl's paper. Let $\overline X\subset V$ be a subset of homogeneous elements such that the image of $\overline X$ under $V\to V/C_2(V)$ is a spanning set of $V/C_2(V)$. Without loss of generality, we may suppose that $\vac \in \overline X$. We define $X = \overline X \setminus \{\vac\}$, and we set $B = \sup_{x\in X}\{\deg(x)\}$.
\begin{definition}\label{def L}
    Let $W$ be a weak $V$-module, and let $w\in W$. For each $x\in X$, there exists $l_x\ge 0$ such that $x_{(l_x)}w \neq 0$ but $x_{(l_x+k)}w = 0$ for all $k>0$. We define $L = L_{W,w,X} \eqdef \sup_{x\in X}\{l_x\}+1$, writing $L = \infty$ if the set $\{l_x\}_{x\in X}\subset \N$ is unbounded.
\end{definition}
If $V$ is $C_2$-cofinite, then $X$ is a finite set, implying $L$ is finite. Note that we may set $(W,w) = (V,\vac)$, in which case we have $L = 0$ for any representative generating set $X$ since $a_{(\ell)}\vac = 0$ whenever $\ell\ge 0$.

\begin{construction}[Canonical filtration]\label{constr:filtration}
    Let $W$ be a weak $V$-module generated by an element $w\in W$. Suppose $L = L_{W,w,X}$ as in \zcref{def L} is finite. We define $\cal F_kW$ to be the submodule generated by elements of the form $u^1_{(-n_1)}\cdots u^1_{(-n_s)}w$, where $u^i\in V$ is homogeneous for all $i$, $\sum_{i=1}^s\deg(u^i)\le k$, and $n_1\ge n_2\ge \cdots\ge n_s > -L$.

    Note that $\cal F_{k-1}W\subset \cal F_kW$, and $W = \bigcup_k\cal F_kW$ since $W$ is generated by $w$. Moreover, we have $\bigcap_k \cal F_kW = 0$. Moreover, suppose that $V$ is lower-truncated; say that $V = \bigoplus_{s\ge S}V_s$ for some $S\in \Z$. Then $\cal F_kW = 0$ unless $k\ge S$.

    Let $u = u^1_{(-n_1)}\cdots u^1_{(-n_s)}w \in W_k$. Consider what happens when we transpose two modes:
    \begin{align*}
        & u^1_{(-n_1)}\cdots u^{i+1}_{(-n_{i+1})}u^i_{(-n_i)}\cdots u^1_{(-n_s)}w\\
        & = u - u^1_{(-n_1)}\cdots [u^i_{(-n_i)},u^{i+1}_{(-n_{i+1})}]\cdots u^1_{(-n_s)}w\\
        & = u - \sum_{j\ge 0}\binom{-n_i}{j}u^1_{(-n_1)}\cdots (u^i_{(j)}u^{i+1})_{-n_i-n_{i+1}-j}\cdots u^1_{(-n_s)}w.
    \end{align*}
    Note that $\deg(u^i_{(j)}u^{i+1}) = \deg(u^i) + \deg(u^{i+1})-1-j < \deg(u^i) + \deg(u^{i+1})$. Therefore the second set of terms on the right-hand side all belong to $\cal F_{k-1}W$. Therefore the string $u$ is independent of the arrangement of the modes modulo $\cal F_{k-1}W$.
    
    Using the same $u$ as above, we consider what happens when we replace $u^i$ by its representative modulo $C_2(V)$. We let $u^i = x^i + c$, where $c\in C_2(V)$. Without loss of generality, we may assume that $c$ is a finite sum of the form $\sum_{\ell\ge 1}v^\ell_{(-1-\ell)}y^\ell$. Then
    \begin{align*}
        u &= u^1_{(-n_1)}\cdots(x^i+c)_{(-n_i)}\cdots u^s_{(-n_s)}w\\
        &= u^1_{(-n_1)}\cdots x^i_{(-n_i)}\cdots u^s_{(-n_s)}w + u^1_{(-n_1)}\cdots c_{(-n_i)}\cdots u^s_{(-n_s)}w.
    \end{align*}
    Note that $\deg(v^\ell_{(-1-\ell)}y^\ell) = \deg(v^\ell) + \deg(y^\ell) + \ell = \deg(u^i)$, and by the associator formula
    \[(v^\ell_{(-1-\ell)}y^\ell)_{(-n_i)} = \sum_{j\ge 0}(-1)^j\binom{-1-\ell}{j}\left(v^\ell_{(-1-\ell-j)}y^\ell_{(-n_i+j)} - (-1)^{-1-\ell}y^\ell_{(-1-\ell-n_i-j)}v^\ell_{(j)}\right).\]
    On the right-hand side, we can see that each term contributes $\deg(v^\ell) + \deg(y^\ell) = \deg(u^i)-\ell$ to the filtration level. Modulo terms in $\cal F_{k-1}W$, we may then replace $u^i$ with its representative $x^i$ modulo $C_2(V)$.
\end{construction}

Now we establish Gaberdiel and Neitzke's spanning set for CFT-type vertex algebras over a ring $\kk$. The proof is exactly the same. 
Since the original proof is written with different notation, we give the same proof using our notation.
\begin{proposition}[\cite[Proposition 8]{GN03}]\label{prop:spanningC2}
    Let $V$ be a CFT-type vertex algebra over a ring $\kk$. Then $V$ is generated by elements of the form
    \[u = x^1_{(-n_1)}\cdots x^r_{(-n_r)}\vac,\]
    where $x^i\in X$ for all $i$ and $n_1 > n_2 > \cdots > n_r > 0$.
\end{proposition}
\begin{proof}
    Consider the filtration $\{\cal F_kV\}_{k\in \Z}$ that we defined above for $L=0$. Note that $\cal F_kV = 0$ unless $k\ge 0$ since $V$ is $\N$-graded. Given $ u = u^1_{(-n_1)}\cdots u^r_{(-n_r)}\vac\in V_0$, then $\sum_i\deg(u^i) = 0$ implies that $\deg(u^i) = 0$ for all $i$. Since $V$ is CFT-type, we have $u^i\in \kk\vac$. By the vacuum axioms, this shows that $\cal F_0V = \kk\vac$.
    
    It remains to show that we may strengthen the assumption to not allow $n_i = n_{i+1}$ for any $i$. We give the following induction hypothesis:
    \begin{center}
        $V_k$ is spanned by states of the form $u = x^1_{(-n_1)}\cdots x^r_{(-n_r)}\vac$ for $n_1 \ge \cdots\ge n_r > 0$, $x^i\in X$ for all $i$, and $n_i\neq n_{i+1}$ for $n_i \le N$.
    \end{center}
    We assign the lexicographic ordering $(k,N) > (k',N')$ if $k > k'$ or $k=k'$ and $N > N'$.
    
    The base case is for $(0,0)$, which we have already looked at. Given any string $u = u^1_{(-n_1)}\cdots u^r_{(-n_r)}\in \cal F_kV$, we consider it modulo $\cal F_{k-1}V$. We may replace $u^i$ with its homogeneous representative modulo $C_2(V)$, and we may rearrange the modes so that they are in the correct order. All repeats are allowed since $n_i > N=0$ for all $i$, so the statement for $(k,0)$ holds. 
    
    We assume that the induction hypothesis holds for all pairs $(k',N') < (k,N)$ for a fixed $(k,N)$. We already have the statement for $(k,N-1)$, so we just need to analyze any string of the form
    \[u = x^{1}_{(-M_1)}\cdots x^m_{(-M_m)}(x'_{(-N)})^sy^1_{(-L_1)}\cdots y^\ell_{(-L_\ell)}\vac,\]
    where $M_1\ge \cdots\ge M_m > N > L_1 > \cdots > L_\ell > 0$ and $s\ge 2$. If $m\neq 0$, then the string $(x'_{(-N)})^sy^1_{(-L_1)}\cdots y^\ell_{(-L_\ell)}\vac$ is in $\cal F_{k-1}V$. Therefore, we may apply the induction hypothesis for $(k-1,N)$ to replace the string with a sum of strings with no repeats for indices $\le N$.

    Now suppose that $m=0$.
    By the associator formula, we have
    \[a_{(-n)}b_{(-n)} = (a_{(-1)}b)_{(-2n+1)} - \sum_{\substack{k\ge 0\\ k\neq n-1}}a_{(-1-k)}b_{(-2n+1+k)} - \sum_{k\ge 0}b_{(-2n-k)}a_{(k)}.\]
    Then we have
    \begin{align*}
        u &= (x'_{(-N)})^sy^1_{(-L_1)}\cdots y^\ell_{(-L_\ell)}\vac\\
        &= (x'_{(-1)}x')_{(-2N+1)}(x'_{(-N)})^{s-2}y^1_{(-L_1)}\cdots y^\ell_{(-L_\ell)}\vac\\
        &\quad - \bigg(\sum_{\substack{k\ge 0\\ k\neq N-1}}x'_{(-1-k)}x'_{(-2N+1+k)} + \sum_{k\ge 0}x'_{(-2N-k)}x'_{(k)}\bigg) (x'_{(-N)})^{s-2}y^1_{(-L_1)}\cdots y^\ell_{(-L_\ell)}\vac.
    \end{align*}
    Now we take $u$ modulo $\cal F_{k-1}V$. We may replace $x'_{(-1)}x'$ with any homogeneous representative of the class $x'_{(-1)}x' + C_2(V)$, which we may then write in terms of the spanning set. The string $u' = (x'_{(-N)})^{s-2}y^1_{(-L_1)}\cdots y^\ell_{(-L_\ell)}\vac$ belongs to $\cal F_{k-1}V$, so by the induction hypothesis for $(k-1,N)$, we can again replace $u'$ with a sum of strings that has no repeats for indices $ \le N$. It remains to analyze the remaining terms. These are terms of the form
    \[x'_{(-N-K)}x'_{(-N+K)}(x'_{(-N)})^{s-2}y^1_{(-L_1)}\cdots y^\ell_{(-L_\ell)}\vac,\]
    where $K>0$. Modulo $\cal F_{k-1}V$, we can again rearrange the modes in the string. If $-N+K\ge 0$, then we get 0 modulo $\cal F_{k-1}V$ since $x'_{(-N+K)}\vac = 0$. Therefore we may assume that $N+K>0$ and $N-K > 0$. Once again, by the induction hypothesis for $(k-1,N)$ we may replace
    $x'_{(-N+K)}(x'_{(-N)})^{s-2}y^1_{(-L_1)}\cdots y^\ell_{(-L_\ell)}\vac$ with a sum of strings with no repeats for indices $\le N$. We have checked all of the terms, so $u$ can be expressed as a sum of strings with no repeats for indices $\le N$. This shows the induction hypothesis for $(k,N)$, so the statement holds for all $(k,N)$.

    To conclude the proof, consider an element $u = u^1_{(-n_1)}\cdots u^r_{(-n_r)}\vac \in V_k$ that is homogeneous of degree $\deg(u)\ge 0$. We have $\deg(u) = \sum_{i=1}^r(\deg(u^i) - 1 + n_i)$, so in particular $0\le n_i-1\le \deg(u)$. Setting $N = \deg(u)+1$, we conclude by the induction hypothesis for $(k,\deg(u)+1)$ that $u$ can be written as a sum of strings of the form $x^1_{(-m_1)}\cdots x^s_{(-m_s)}\vac$ for $m_1\ge \cdots \ge m_s > 0$, $\sum_{i=1}^s\deg(x^i)\le k$, and $m_i\neq m_{i+1}$ for $m_i\le \deg(u)+1$. The no-repeat condition always holds since each string has degree $\deg(u)$ as well, so this implies that $m_1 > \cdots > m_s > 0$. Since the filtration on $V$ is exhaustive, any element of the degree $s$ component $V_{s}$ belongs to $\cal F_kV$ for $k\gg 0$. Therefore this statement holds for every element of $V$, which concludes the proof.
\end{proof}
\begin{proposition}[\cite[Theorem 2.8]{BuhlSpanning}]
    Suppose $V$ is $C_2$-cofinite. Then $V$ is $C_n$-cofinite for all $n\ge 2$.
\end{proposition}
\begin{proof}
    We use the spanning set from \zcref{prop:spanningC2}. Consider a string of the form $u = x^1_{(-n_1)}\cdots x^r_{(-n_r)}\vac$ for $x^i\in X$, $n_1 > \cdots >n_r > 0$. We have that $u\in C_n(V)$ if $n_1\ge n$. Therefore $V/C_n(V)$ is spanned by strings of the form $x^1_{(-n_1)}\cdots x^r_{(-n_r)}\vac$ for $n > n_1 > \cdots > n_r > 0$. If $V$ is $C_2$-cofinite, then we can take $X$ to be finite, hence there are only finitely many strings of this form. This concludes the proof.
\end{proof}
\begin{definition}
    If $V$ is $C_2$-cofinite, then there exists a minimal $N\ge 0$ such that $V_{\ge N}\subset C_2(V)$. We define $Q = \max\{N,2B-1\}+1$.
\end{definition}
One may use a more general version of the above argument to prove the following result:
\begin{theorem}[\cite[Theorem 1]{BuhlSpanning}]
    Let $V$ be a $C_2$-cofinite CFT-type vertex algebra over a ring $\kk$, and let $W$ be a weak $V$-module generated by $w\in W$. Then $W$ is spanned by elements of the form
    \[x^1_{(-n_1)}\cdots x^k_{(-n_k)}w,\]
    where $n_1>n_2>\cdots>n_k>-L$ and $x^j\in \overline X$ for all $1\le j\le k$. In addition, if $n_j>0$, then $n_j>n_{j'}$ for $j<j'$. If $n_j\le 0$, then $n_j = n_{j'}$ for at most $Q-1$ values of $j'$.
\end{theorem}
The proof of this is essentially a refinement of \zcref{prop:spanningC2}. The original proof does not require any usage of dividing by integers, provided we replace the standard definition of $C_2(V)$ with the one in \zcref[noname]{Cn(V)}. However, the proof is very long, so we omit the details.

\subsection{Some corollaries on $C_n$-cofiniteness}
Let $U$ be a graded subspace of $V$ such that $V = U\oplus C_2(V)$. Recall that $V$ is linearly spanned by vectors of the form $\alpha^1_{(-n_1)}\cdots \alpha^r_{(-n_r)}\vac$, with $\alpha^i\in U$ and $n_1 > \cdots > n_r > 0$. For any $m,q\in \N_{\ge 1}$, we set
\[C_{m,q}(W) = \{(\alpha^1_{(-n_1)}\cdots \alpha^r_{(-n_r)}\vac)_{(-p)}w \mid m\ge n_1 > \cdots n_r > 0,\ \alpha^i\in U,\ p\ge q\}.\]
For $m\ge 2$, we also define
\[C_m(U,W) = \{\alpha_{(-k)}w \mid \alpha\in U,\ w\in W,\ k\ge m\}.\]

We also have the following results from \cite{AN2}. Aside from one small modification in \zcref{lem:ANlem2}, the proofs below are the same as originally presented.
\begin{lemma}[\cite[Lemma 4.3]{AN2}]\label{lem:ANlem1}
    Let $m,q>0$. Then $C_q(W)\subset C_{m,q}(W) + C_m(U,W)$.
\end{lemma}
\begin{proof}
    The proof is the same as the original proof over $\C$.
    
    It suffices to show that for any $\alpha^i\in U$ and $n_1 > \cdots > n_r > 0$, $p\ge q$, and $w\in W$, we have
    \[(\alpha^1_{(-n_1)}\cdots \alpha^r_{(-n_r)}\vac)_{(-p)}w \in C_{m,q}(W) + C_m(U,W).\]
    We may assume that $n_1 > m$, since otherwise the element is automatically in $C_{m,q}(W)$.

    First we show the above statement for $r=1$. For all $\alpha\in U$, we have
    \[(\alpha_{(-n)}\vac)_{(-p)}w = (-1)^{p-1}\binom{-n}{p-1}\alpha_{(-n-p+1)}w \in C_m(U,W).\]
    Now suppose that the above holds for any $r<r_0$ with $r_0\ge 2$. Define 
    \[\beta = \alpha^2_{(-n_2)}\cdots \alpha^{r_0}_{(-n_{r_0})}\vac;\] 
    by the associator formula, we have
    \[(\alpha^1_{(-n_1)}\beta)_{(-p)}w = \sum_{i\ge 0}\binom{-n_1}{i}(-1)^i\left(\alpha^1_{(-n_1-i)}\beta_{(-p+i)}w - (-1)^{n_1}\beta_{(-n_1-p-i)}\alpha^1_{(i)}w\right).\]
    The first term on the right-hand side belongs to $C_m(U,W)$ since $n_1 > m$. By the inductive hypothesis, the second term on the right-hand side belongs to $C_{m,q}(W) + C_m(U,W)$ since $n_1+p+i\ge p\ge q$.
\end{proof}
\begin{lemma}[\cite[Lemma 4.4]{AN2}]\label{lem:ANlem2}
    Let $m>0$. For any $a\in V_+$ and $w\in W$ we have $a_{(-q)}w \in C_m(U,W)$ for all $q\ge \deg(a)m$.
\end{lemma}
\begin{proof}
    The proof is the same as the original proof over $\C$.
    
    If $\deg(a) = 1$, then $a\in U$ because $V_{1}\cap C_2(V) = \{0\}$; therefore $a_{(-q)}w \in C_m(U,W)$ for any $q\ge m = \deg(a)m$.

    Now suppose that the desired statement holds for all $1 \le \deg(a) < r_0$ for some $r_0\ge 2$, and suppose that $\deg(a) = r_0$. If $a\in U$ then $a_{(-q)}w\in C_m(U,W)$ for all $q\ge \deg(a)m$ (since $\deg(a)m > m$). Now assume that $a\neq 0$ is of the form $a = b'_{(-1-k)}c \in C_2(V)$ for $b',c\in V$ and $k\ge 1$. Define $b = T^{(k)}b'$, so $a = b_{(-1)}c$. We have that $\deg(a) = \deg(b) + \deg(c)$. Since $V$ is $\N$-graded, we have that $\deg(a)\ge \deg(b)$. Note that $\deg(b) = \deg(b') + k\ge 1$, and since $a$ is nonzero we have $b\neq 0$, implying $\deg(b') > 0$ (otherwise $b' \in \kk \vac$ by the CFT-type assumption). Therefore $1\le \deg(b)\le \deg(a)$. This implies $\deg(c) < \deg(a)$.

    By the associator formula, we have for any $q\in \Z$ that
    \[a_{(-q)}w = \sum_{i\ge 0}\left(b_{(-1-i)}c_{(-q+i)} + c_{(-1-q-i)}b_{(i)}\right)w.\]
    Since $q\ge \deg(a)m > \deg(c)m$, by the induction hypothesis we have that
    \[c_{(-1-q-i)}b_{(i)}C_m(U,W)\] 
    for all $i\ge 0$ and $q\ge \deg(a)m$ using induction on the element $c$. Now we must show that the other term is in $C_m(U,W)$. If $i \ge \deg(b)m$, then by the induction hypothesis again we have
    \[b_{(-1-i)}c_{(-q+i)}w \in C_m(U,W)\]
    for all $q\in \Z$. If instead we have $i < \deg(b)m$, then the commutator formula says
    \[b_{(-1-i)}c_{(-q+i)}w = c_{(-q+i)}b_{(-1-i)}w + \sum_{j\ge 0}\binom{-1-i}{j}(b_{(j)}c)_{(-1-q-j)}w.\]
    Since $q-i\ge \deg(a)m-\deg(b)m + 1 = \deg(c)m + 1 > \deg(c)m$, by the induction hypothesis we have that
    \[c_{(-q+i)}b_{(-1-i)}w\in C_m(U,W).\]
    Since $\deg(a) > \deg(b_{(j)}c)$ for all $j\ge 0$, we have by the induction hypothesis that
    \[(b_{(j)}c)_{(-1-q-j)}w\in C_m(U,W)\]
    because $1+q+j \ge \deg(a)m > \deg(b_{(j)}c)m$. This shows the result.
\end{proof}
\begin{proposition}[\cite[Proposition 4.5]{AN2}]\label{AN Prop 4.5}
    Let $V$ be a $C_2$-cofinite, CFT-type vertex $\kk$-algebra. Let $U\subset V$ be a graded subspace such that $V = U\oplus C_2(V)$, and let $m>0$. Then there exists a positive integer $k$ such that $C_k(W)\subset C_m(U,W)$.
\end{proposition}
\begin{proof}
    The proof is the same as the original proof over $\C$.

    By \zcref{lem:ANlem1}, we have that $C_k(W) \subset C_{m,k}(W) + C_m(U,W)$ for any $k\ge 2$. It suffices to show that there exists $k$ such that $C_{m,k}(W)\subset C_m(U,W)$. Since $U$ is a finite-dimensional graded subspace, let $s_U$ be the maximum degree in $U$. For any $\alpha^i\in U$ and $m\ge n_1 > \cdots > n_r > 0$, we have
    \begin{align*}
        \deg(\alpha^1_{(-n_1)}\cdots \alpha^r_{(-n_r)}\vac) &= \sum_{i=1}^r(\deg(\alpha^i) - 1 + n_i)\\
        &\le r(s_U - 1) + \sum_{i=1}^r (m-i+1) = -\frac{1}{2}r^2 + r\left(s_U + m - \frac{1}{2}\right).
    \end{align*}
    This is bounded above by some positive integer $k_0 \ge \tfrac{1}{2}(s_U + m - \tfrac{1}{2})^2$; here $k_0$ only depends on $U$ and $m$. Setting $k = k_0m$ we get $C_{m,k}(W)\subset C_m(U,W)$ by \zcref{lem:ANlem2} since any element in $C_{m,k}(W)$ is a linear combination of elements of the form $a_{(-p)}w$ for $\deg(a)\le k_0$, $w\in W$, and $p\ge k$ (note that $k\ge \deg(a)m$).
\end{proof}
These results will be used in \zcref{prop:fin-dim-coinvars}.

\section{Change of variables for vertex algebras}\label{sec:Aut O}
In this section, we describe describe vertex algebras equipped with an action of the group $\Aut \mathcal{O}$ of continuous automorphisms of $\kk[\![t]\!]$. Geometrically, this allows us to think of the data of the vertex operator $Y(a,z)$ as being independent of the choice of $z$, up to some equivariance relation. This group action is necessary to define the vertex algebra sheaf.

\subsection{Integrating group actions}
Let $\kk$ be a commutative unital ring (in practice, an algebraically closed field), and let $\mathcal{O} = \kk[\![t]\!]$. Let $R$ be a $\kk$-algebra.
We define the group functor $\uAut\mathcal{O}$ as follows:
\[
\uAut\mathcal{O}(R)\eqdef\left\{t\mapsto\rho(t)=a_0+a_1t+a_2t^2+\cdots \;\middle|\; \begin{array}{l}  a_i \in R, \, a_1 \mbox{ a unit},\\ a_0 \mbox{ nilpotent}
  \end{array}\right\}.
\]
Similarly, we define the subfunctors $\Aut\mathcal{O}$ and $\Aut_+\mathcal{O}$:
\begin{align*}
    \Aut\mathcal{O}(R)&\eqdef\left\{t\mapsto\rho(t)=a_1t+a_2t^2+\cdots \;\middle|\; \begin{array}{l}  a_i \in R, \\  a_1 \mbox{ a unit}
  \end{array}\right\},\\
  \Aut_+\mathcal{O}(R)&\eqdef\left\{t\mapsto\rho(t)=t+a_2t^2+\cdots \;\middle|\; a_i \in R \right\}.
\end{align*}
Finally, we define the group functor $\Aut\cal K$:
\[
\Aut\cal K(R) \eqdef \left\{t\mapsto\rho(t)=\sum_{i\ge i_0} a_i t^i \;\middle|\; \begin{array}{l}  a_i \in R, \, a_1 \mbox{ a unit}, \\ a_i \mbox{ nilpotent for $i\le 0$}
  \end{array}\right\}.
\]
The groups $\Aut\mathcal{O}$ and $\Aut_+\mathcal{O}$ are related by a semidirect product:
\[\Aut\mathcal{O} \cong \Aut_+\mathcal{O}\rtimes \kk^\times.\]
Every element $x\in \kk^\times$ may be interpreted as a continuous automorphism of $\mathcal{O}$ via $x\cdot f(t) = x^{t\partial_t}f(t) = f(xt)$.
Explicitly, every continuous automorphism $\sigma\in \Aut\mathcal{O}$ may be expressed as a composition $\sigma = \rho\circ x^{t\partial_t}$, where $x\in \kk^\times$ is the $t$ coefficient of $\sigma$ and $\rho(t) = \sigma(x^{-1}t)$.

The goal of this section is to provide a natural way of defining group actions of $\Aut \mathcal{O}$ on a vertex algebra $V$ that are compatible with its vertex algebra structure.
\subsection{In characteristic 0}
In \cite{FBZ}, the authors give the following definition as a means of obtaining an $\Aut\mathcal{O}$-action on $V$:
\begin{definition}\label{FBZquasi}
    Let $V$ be a $\Z$-graded vertex algebra over a $\Q$-algebra $\kk$. We say $V$ is \textit{quasi-conformal} if it has a Lie algebra representation
    \[\pi\colon \Der\mathcal{O}\to \End_{\kk}(V),\]
    where we denote $L_{m} = \pi(-t^{m+1}\partial_t)$, such that the following hold:
    \begin{enumerate}[label = (\roman*)]
        \item We have $L_{-1}=T$.
        \item For all $a\in V$ homogeneous we have $L_0a = \deg(a)a$.
        \item For all $m\ge 0$, the operator $L_m$ is homogeneous of degree $-m$.
        \item For all $m\ge 0$, we have the following commutator formula:
        \begin{equation}
            [L_{m-1},Y(a,z)] = \sum_{k\ge 0}\binom{m}{k}z^{m-k}Y(L_{k-1}a,z).
        \end{equation}
    \end{enumerate}
\end{definition}
The authors only define the action over $\C$, but the same approach can be done over any $\Q$-algebra $\kk$.
One can show that axiom (iii) is redundant given the others. Many common examples of vertex algebras are quasi-conformal. For example, every vertex operator algebra is quasi-conformal via the action of the Virasoro algebra.
\begin{lemma}
    Let $R$ be a $\Q$-algebra. Then there is a bijection
    \begin{align*}
        \exp\colon \Der_+\mathcal{O} &\to \Aut_+\mathcal{O},\\
        v(t)\partial_{t} &\mapsto \exp(v(t)\partial_{t}) = \sum_{k=0}^\infty \frac{1}{k!}(v(t)\partial_{t})^k,
    \end{align*}
    whose inverse is given by the logarithm map
    \begin{align*}
        \log\colon \Aut_+\mathcal{O} &\to \Der_+\mathcal{O},\\
        \rho&\mapsto \log\rho = \sum_{k=1}^\infty \frac{(-1)^{k-1}}{k}(\rho-\id)^k.
    \end{align*}
\end{lemma}
\begin{lemma}[\cite[Lemma 6.3.2]{FBZ}]
    Let $V$ be a quasi-conformal vertex algebra over a $\Q$-algebra $\kk$. We may define a representation $\Pi\colon \Aut\mathcal{O}\to \GL_{\kk}(V)$ as follows: 
    \begin{itemize}
        \item Given $x\in \kk^\times$ and a homogeneous element $a\in V$, we define
        \[\Pi(x^{t\partial_t})a = x^{-L_0}a = x^{-\deg(a)}a.\]
        \item Given $\rho\in \Aut_+\mathcal{O}$, let $v(t)\partial_t = \sum_{m\ge 1}v_mt^{m+1}\partial_t\in \Der_+\mathcal{O}$ be the unique vector field such that $\rho = \exp(v(t)\partial_t)$. We define
        \[\Pi(\rho) = \sum_{k=0}^\infty\frac{1}{k!} \pi(v(t)\partial_t)^k,\]
        where we set $\pi(v(t)\partial_t) = -\sum_{m\ge 1}v_mL_m$.
    \end{itemize}
    Moreover, for all $\rho\in \Aut\mathcal{O}$ we have \textup{Huang's identity}:
    \begin{equation}\label{eq:Huang1}
        Y(a,z) = \Pi(\rho)Y(\Pi(\rho_z)^{-1}a,\rho(z))\Pi(\rho)^{-1},
    \end{equation}
    where $\Pi(\rho_z) = e^{-zT}\Pi(\rho)e^{\rho(z)T}$.
\end{lemma}
We refer to \cite[Remark 6.6.5]{FBZ} for the derivation of the identity for $\Pi(\rho_z)$. We take this as a definition for our own convenience.

Plugging in the definition of $\Pi(\rho_z)$ into Huang's identity, we obtain the following:
\begin{align*}
    \Pi(\rho)Y(\Pi(\rho_z)^{-1}a,\rho(z))\Pi(\rho)^{-1} &= \Pi(\rho)Y(e^{(\rho(z)-z)T}e^{-\rho(z)T}\Pi(\rho)^{-1}e^{zT}a,z)\Pi(\rho)^{-1}\\
    &= \Pi(\rho)Y(e^{-zT}\Pi(\rho)^{-1}e^{zT}a,z)\Pi(\rho)^{-1}.
\end{align*}
Rearranging, we obtain the following equivalent form of Huang's identity:
\begin{equation}\label{eq:Huang2}
\Pi(\rho)Y(a,z)\Pi(\rho)^{-1} = Y(e^{-zT}\Pi(\rho)e^{zT}a,z).
\end{equation}
There is a notion of a primary element with respect to the action of $\Der\mathcal{O}$. Specifically, we say that an element $a\in V$ is a \textit{primary of degree $\Delta$} if $L_0a = \Delta a$ and $L_ia = 0$ for $i>0$. It is shown in \cite[Proposition 6.4.4]{FBZ} that for such $a\in V$ we have
\[Y(a,z) = \Pi(\rho)Y(a,\rho(z))\Pi(\rho)^{-1}(\rho'(z))^\Delta.\]
Using Huang's identity, this amounts to saying that $\Pi(\rho_z)a = (\rho'(z))^{-\Delta}a$. By setting $z=0$, we obtain $\Pi(\rho)a = (\rho'(0))^{-\Delta}a$, implying that all of the degree $<0$ terms in $\Pi(\rho)$ act on $a$ by 0 by varying the coefficients of $\rho$. That is, $a\in V$ is a primary of degree $\Delta$ if and only if $\Pi(\rho_z)a = (\rho'(z))^{-\Delta}a$.

\subsection{General actions of $\Aut\mathcal{O}$}
If we are working in characteristic $p>0$, then the exponential map $\exp\colon \Der_+\mathcal{O} \to \Aut_+\mathcal{O}$ as described above may not make sense in positive characteristic since we have to divide by $k$ for all $k\ge 1$.

For the goals of this paper, it suffices to just give the correct notion of an $\Aut\mathcal{O}$-action on a vertex $\kk$-algebra $V$.
\begin{definition}\label{Aut O action def}
    Let $V = (V,\vac,Y)$ be a $\Z$-graded vertex algebra over a ring $\kk$. An \textit{action of $\Aut\mathcal{O}$ on $V$} is a group representation $\Pi\colon \Aut\mathcal{O}\to \GL_{\kk}(V)$ such that the following hold:
    \begin{enumerate}[label = (\roman*)]
        \item For all $\rho\in \Aut\mathcal{O}$ and $n\in \N$, the $z^n$ term in $\Pi(\rho_z) = e^{-zT}\Pi(\rho)e^{\rho(z)T}$, denoted $\Pi(\rho_z)[z^n]$, preserves the left filtration on $V$ induced by the $\Z$-grading. That is, $\Pi(\rho_z)[z^n]\cdot\mrm G_pV\subseteq \mrm G_pV$ for all $n\in \N$ and $p\in \Z$.
        \item For all $\rho\in \Aut\mathcal{O}$ and homogeneous $a\in V_{\Delta}$, we have
        \[\Pi(\rho_z)a - (\rho'(z))^{-\Delta}a \in \mrm G_{\Delta-1}V[\![z]\!].\]
        That is, in the associated graded $\gr^{\mrm G}V$ we have $\Pi(\rho_z)a = (\rho'(z))^{-\Delta}a$.
        \item For all $x\in \kk^\times\subset \Aut\mathcal{O}$, we have $\Pi(x^{t\partial_t})a = x^{-\deg(a)}a$ for every homogeneous $a\in V$.
        \item For all $\rho\in \Aut\mathcal{O}$, we have \textit{Huang's identity}:
        \[\Pi(\rho)Y(\Pi(\rho_z)^{-1}a,\rho(z))\Pi(\rho)^{-1} = Y(a,z).\]
        Equivalently, we have
        \[\Pi(\rho)Y(a,z)\Pi(\rho)^{-1} = Y(e^{-zT}\Pi(\rho)e^{zT}a,z).\]
    \end{enumerate}
\end{definition}
These conditions are all satisfied for $\Aut\mathcal{O}$-actions arising from quasi-conformal vertex algebras over a $\Q$-algebra.
\begin{remark}
    Note that for any $\Aut\mathcal{O}$-action on $V$, we have $\Pi(\rho)\vac = \vac$ for all $\rho\in \Aut\mathcal{O}$ using Huang's identity. Due to the axiom $a = a_{(-1)}\vac$ for vertex algebras, describing how $\rho$ acts on $V$ amounts to giving a formula for $\Pi(\rho)Y(a,z)\Pi(\rho)^{-1}$ for all $a\in V$. We will describe most $\Aut\mathcal{O}$-actions in this way.
\end{remark}
\begin{remark}\label{rmk:poschardescent}
    A general principle in algebraic geometry is that descent data for quasi-coherent sheaves is determined by the order 1 descent data (in this case, a flat connection) \textit{only in characteristic 0}. If one wishes to perform descent over more general bases, one usually has to analyze descent data of all orders.\footnote{Descent data of orders $p^n$ for all $n\ge 0$ are usually sufficient in characteristic $p>0$.}
    Our main reason for changing the underlying hypotheses in this way is that the structure of a VOA in positive characteristic is not enough to perform the desired descent to $\overline{\mathcal{M}}_{g,n}$.
\end{remark}

Given a vertex algebra with an $\Aut\mathcal{O}$-action, there is a natural version of the involution $\theta$ from \zcref{sec:involution}. Indeed, we simply define $\theta(J_n(a)) = J_{-n}(\gamma(a))$, and this induces an involution of the universal enveloping algebra $\UV$. We conjecture that an $\Aut\mathcal{O}$-action on $V$ naturally induces a M\"obius vertex algebra structure on $V$ such that $L_1^{(i)}(a) = (-1)^{\deg(a)}\Pi(\gamma)_{-i}(a)$, where $\Pi(\gamma)_{-i} = \pr_{-i}\circ\Pi(\gamma)$ is the degree $-i$ component of the filtration-preserving endomorphism $\Pi(\gamma)$.

There is also a natural notion of $\Aut\mathcal{O}$-equivariant $V$-modules. They satisfy similar identities to those of an $\Aut\mathcal{O}$-action on $V$, but we have to use $\Lambda$-graded $V$-modules again. For the following definition, we again assume that $\Lambda$ is a subset of $\Q$ that is closed under the action of $(\Z,+)$.
\begin{definition}\label{def:Aut O V-module}
    Let $V$ be a vertex algebra over an algebraically closed field $\kk$ with an $\Aut\mathcal{O}$-action. Let $W$ be a lower-truncated $\Lambda$-graded $(V,\mathcal{H})$-module. An \textit{$\Aut\mathcal{O}$-equivariant structure on $W$} is a group representation $\Pi^W\colon \Aut\mathcal{O} \to \GL_\kk(W)$ such that the following holds:
    \begin{enumerate}[label = (\roman*)]
        \item For all $\rho\in \Aut\mathcal{O}$ and $n\in \N$, the $z^n$ term in $\Pi^W(\rho_z) = e^{-zT}\Pi^W(\rho)e^{\rho(z)T}$, denoted $\Pi^W(\rho_z)[z^n]$, preserves the left filtration on $V$ induced by the $\Lambda$-grading. That is, $\Pi^W(\rho_z)[z^n]\cdot\mrm G_\ell W\subseteq \mrm G_\ell W$ for all $n\in \N$ and $\ell\in \Lambda$.
        \item For all $\rho\in \Aut\mathcal{O}$ and homogeneous $w\in W_{\ell}$, we have
        \[\Pi^W(\rho_z)w - (\rho'(z))^{-\ell}w \in \mrm G_{\ell-1}W[\![z]\!].\]
        That is, in the associated graded $\gr^{\mrm G}W$ we have $\Pi^W(\rho_z)w = (\rho'(z))^{-\ell}w$.
        \item For all $x\in \kk^\times\subset \Aut\mathcal{O}$ and $w\in W_\ell$, we have $\Pi(x^{t\partial_t})w = x^{-\ell}w$.
        \item For all $\rho\in \Aut\mathcal{O}$, we have \textit{Huang's identity}:
        \[\Pi^W(\rho)Y^W(\Pi(\rho_z)^{-1}a,\rho(z))\Pi^W(\rho)^{-1} = Y^W(a,z).\]
        Equivalently,
        \[\Pi^W(\rho)Y^W(a,z)\Pi^W(\rho)^{-1} = Y^W(e^{-zT}\Pi(\rho)e^{zT}a,z).\]
    \end{enumerate}
\end{definition}
We will not generally require that our admissible $V$-modules admit $\Aut\mathcal{O}$-equivariant structures, though we will ask for them later on when we discuss descent to $\overline{\mathcal{M}}_{g,n}$.

In this context, a \textit{primary element} of $V$ is a homogeneous element $a\in V_\Delta$ such that $\Pi(\rho_z)a = (\rho'(z))^{-\Delta}a$ for all $\rho\in \Aut\mathcal{O}$. There is a similar definition for primary elements of modules.

\subsubsection{Examples}
Let $V$ be a VOA over $\kk$. If $V$ is generated by a collection of primary fields $\{a^i(z)\}_{i\in I}$, then $V$ is naturally equipped with an action of $\Aut\mathcal{O}$ given by
\[Y(a^i,z) = \Pi(\rho)Y(a^i,\rho(z))\Pi(\rho)^{-1}(\rho'(z))^\Delta,\]
where $a^i(z) = Y(a^i,z)$. This applies to VOAs such as the affine, Heisenberg, and $\beta\gamma$ VOAs for instance. All of these are then equipped with an action of $\Aut\mathcal{O}$ in a natural sense, and the M\"obius vertex algebra structure is given by $L_1^{(i)} = \frac{1}{i!}L_1^i$ for all $i\ge 0$ in the na\"ive sense, checking that the action of these symbols is indeed well-defined. These structures are indeed compatible in the sense we have required above; we omit the details.

Another natural example is the universal Virasoro vertex algebra $\overline{V}_{\Vir}(c,0)_\kk$ over a field $\kk$ of characteristic $>2$.
Given $\rho\in \Aut\mathcal{O}$, We assign to $T(z) = Y(\omega,z)$ the action
\[\Pi(\rho)^{-1}Y(\omega,z)\Pi(\rho) = Y(\omega,\rho(z))(\rho'(z))^2 + \frac{c}{2}\widehat{S}(\rho,z),\]
where we define the \textit{normalized Schwarzian derivative} $\widehat{S}(f,z)$ to be
\[\widehat{S}(f,z) = \frac{\partial_z^{(3)}f(z)}{\partial_z^{(1)}f(z)} - \left(\frac{\partial_z^{(2)}f(z)}{\partial_z^{(1)}f(z)}\right)^2.\]
Recall that the Schwarzian derivative is
\[S(f,z) = \frac{f'''(z)}{f'(z)} - \frac{3}{2}\left(\frac{f''(z)}{f'(z)}\right)^2.\]
We defined $\widehat{S}(\rho,z)$ in the way we did so that over $\Q$ we have
\[\frac{c}{12}S(f,z) = \frac{c}{2}\widehat{S}(f,z).\]
If we replace $c/2$ with $\widetilde c$, then this is well-defined over any ring $\kk$.
The M\"obius vertex algebra structure is described in \cite{Li18} or \cite{Griffin26}, and it is given by setting $L_1^{(i)} = \frac{1}{i!}L_1^i$ in the na\"ive sense again.

\part{Sheaves of vertex algebras on curves}\label{part: sheaves}
In this part, we will carry out the main goal of this paper, which is to develop the theory of coinvariants and conformal blocks associated to an $\N$-graded vertex $\kk$-algebra with an action of $\Aut\mathcal{O}$ over families of stable curves on a fixed algebraically closed field $\kk$. We will follow the constructions in \cite{DGT21} and especially \cite{DGT23} closely.

\section{The vertex algebra sheaf on curves}\label{sec:vertex alg sheaf}
Given an $\N$-graded vertex $\kk$-algebra $V$ with an action of $\Aut\mathcal{O}$, we define the sheaf of vertex algebras $\mathcal{V}_C$ on any curve $C$ over $\kk$ with at worst simple nodal singularities, following \cite{DGT23}. We will not assume $V$ is of CFT-type.
After doing so, we generalize the construction to relative nodal curves.
In characteristic 0, it is known that $\mathcal{V}_C$ has a canonical flat logarithmic connection $\nabla\colon \mathcal{V}_C\to \mathcal{V}_C\otimes \omega_C$. In characteristic $p>0$, we still have this flat logarithmic connection, but we also have higher-order differential maps, which we arrange into the data of a so-called ``logarithmic $\mathcal{D}$-module structure'' on $\mathcal{V}_C$. We use this logarithmic stratification to define the sheaf of chiral Lie algebras, which is a geometric analogue of the ancillary Lie algebra $\mathfrak{L}(V)^{\mathsf{L}}$.



\subsection{Bundles of formal coordinates}
Here we review the principal $\left(\Aut\mathcal{O}\right)$-bundle $\mathscr{A}ut_{\mathcal{C}}$ associated to a relative nodal curve $\mathcal{C}\to S$ over $\kk$, which is described in \cite[\textsection 2.2]{DGT23}. The only difference here is that we are working over an arbitrary algebraically closed field $\kk$; most constructions carry over with minimal changes otherwise.

\subsubsection{On smooth curves}
Let $C$ be a smooth projective curve over $\kk$.
We define $\mathscr{A}ut_C$ to be the space of pairs $(P,t)$, where $P\in C$ is a point in $C$ and $t$ is a formal coordinate at $P$ (which is an element of the completed local ring $\widehat{\mathcal{O}}_{C,P}$ such that $t\in \frak m_P\setminus \frak m_P^2$, where $\frak m_P$ is the maximal ideal of $\widehat{\mathcal{O}}_{C,P}$). There is a forgetful map $\mathscr{A}ut_C\to C$ whose fiber at $P\in C$ is equal to the space $\mathscr{A}ut_P$ of formal coordinates at $P$.

The group $\Aut\mathcal{O}$ acts on fibers of $\mathscr{A}ut_C\to C$ via change of coordinates:
\[\mathscr{A}ut_C\times\Aut\mathcal{O}\to \mathscr{A}ut_C,\qquad ((P,t),\rho)\mapsto (P,t\cdot \rho \eqdef \rho(t)).\]
This action is simply transitive, which is to say $\mathscr{A}ut_C$ is a principal $\left(\Aut\mathcal{O}\right)$-bundle over $C$. The choice of a formal coordinate $t$ at $P$ induces a trivialization
\[\Aut\mathcal{O}\stackrel{\simeq_t}{\longrightarrow}\mathscr{A}ut_P,\qquad \rho\mapsto \rho(t).\]

\subsubsection{On nodal curves}
In the case where the curve is nodal, we can understand it in terms of normalizations as follows. Given a curve $C$ with a node $Q$, let $\eta\colon\widetilde C\to C$ be the partial normalization of $C$ at $Q$, and let $Q_\pm$ be the two preimages of $Q$ in $\widetilde C$. 
We construct the principal $\left(\Aut\mathcal{O}\right)$-bundle $\mathscr{A}ut_C$ on the curve $C$ by identifying the fibers of $\mathscr{A}ut_{\widetilde C}$ at the two preimages $Q_+$ and $Q_-$ of a node $Q$ via a gluing isomorphism
\begin{equation}\label{eq:nodal gluing}
    \mathscr{A}ut_{Q_+} \simeq_{s_+}\Aut\mathcal{O}\stackrel{\gamma}{\to}\Aut\mathcal{O} \simeq_{s_-}\mathscr{A}ut_{Q_-},\qquad \rho(s_+)\mapsto \rho\circ \gamma(s_-),
\end{equation} 
where $\gamma\in \Aut\mathcal{O}$ is defined as
\[\gamma(z) = \frac{1}{1+z}-1 = \sum_{k\ge 1}(-z)^k.\]
Notice that $\gamma\circ \gamma(z) = z$, hence $\gamma$ determines an involution of $\Aut\mathcal{O}$. The isomorphism in \zcref[noname]{eq:nodal gluing} then amounts to the identification $s_+ = \gamma(s_-)$ between formal coordinates. It is useful to point out that $\gamma$ may be realized as an exponential operator
\[\gamma(z) = e^{-z^2\partial_z}(-1)^{-z\partial_z}(z).\]
This fundamental relationship between $s_+$ and $s_-$ can be motivated using log geometry, which we will discuss in \zcref{sec:logstratification}.

\subsubsection{On families of curves}\label{sec families of curves Aut}
There is a natural generalization of $\mathscr{A}ut_{C}$ for families of curves, which can be understood in terms of moduli of curves.
We define $\widetriangle{\mathcal{M}}_{g,n}$ to be the moduli space of objects of type $(C,P_\bullet, t_\bullet)$, where $(C,P_\bullet=(P_1,\dots,P_n))$ is a stable, $n$-pointed genus $g$ curve, and $t_\bullet=(t_1,\dots,t_n)$ with $t_i$ a formal coordinate at $P_i$ for all $i$. The action of $(\Aut\mathcal{O})^{\times n}$ by change of coordinates endows $\widetriangle{\mathcal{M}}_{g,n}$ with the structure of an $(\Aut\mathcal{O})^{\times n}$-torsor over $\overline{\mathcal{M}}_{g,n}$. For further information about $\widetriangle{\mathcal{M}}_{g,n}$ over the locus of smooth curves, see \cite{ADCKP}, \cite[\S 6.5]{FBZ}, and \cite[\S 3.2]{DGT21}.


For example, an object of $\widetriangle{\mathcal{M}}_{g,1}$ over a scheme $S$ consists of a
semistable curve $\mathcal{C} \to S$ with a section $P\colon S \to \mathcal{C}$ mapping to the smooth locus of $\mathcal{C}$, together with a formally unramified thickening $S \times  {\rm{Spf}}(\kk[\![ t]\!]) \to \mathcal{C}$ of the section $P$ such that every genus 1 component has at least one special point and every rational component has at least three special points. 
Here, a special point is either a marked point or a node, counted with multiplicity. The object ${\rm{Spf}}(\kk[\![ t]\!])$ is the formal spectrum of the complete topological ring $\kk[\![ t]\!]$ (see \cite[\S A.1.1]{FBZ}).

The action of $\Aut\mathcal{O}$ via change of coordinates identifies $\widetriangle{\mathcal{M}}_{g,1}$ as a principal $(\Aut\mathcal{O})$-bundle over $\overline{\mathcal{M}}_{g,1}$. There is a diagram of Cartesian squares
\[
\begin{tikzcd}
\mathscr{A}ut_{P/S} \arrow[rightarrow]{r} \arrow[rightarrow]{d} &\mathscr{A}ut_{\mathcal{C}/S} \arrow[rightarrow]{r} \arrow[rightarrow]{d}& \widetriangle{\mathcal{M}}_{g,1} \arrow[rightarrow]{d}{\Aut\mathcal{O}} \\
S \arrow[rightarrow]{r}{P} &\mathcal{C} \arrow[rightarrow]{r} \arrow[rightarrow]{d}& \overline{\mathcal{M}}_{g,1} \arrow[rightarrow]{d}\\
&S \arrow[rightarrow]{r}{}& 	\overline{\mathcal{M}}_{g}.
\end{tikzcd}
\]
That is, the space $\mathscr{A}ut_{\mathcal{C}/S}$ is the pullback of the maps $\widetriangle{\mathcal{M}}_{g,1}\to \overline{\mathcal{M}}_{g}$ and $S \rightarrow \overline{\mathcal{M}}_{g}$. We recommend \cite[\textsection 2.2.2]{DGT23} for an in-depth description of the fibers of $\mathscr{A}ut_C$ over a node.

\subsection{The sheaf of vertex algebras}\label{sheaf of vertex algebras}
The sheaf $\mathcal{V}_C$ of vertex algebras on a relative nodal curve is constructed by faithfully flat descent of an $\Aut\mathcal{O}$-equivariant sheaf along an $\Aut\mathcal{O}$-torsor. This allows us to discuss coordinates for the sections of this sheaf provided that we can show that our constructions are equivariant with respect to change of coordinates.
\begin{definition}
    Let $\mathcal{C}\to S$ be a relative nodal curve over $\kk$. The \textit{vertex algebra bundle} $V_{\mathcal{C}}$ is
    \[V_{\mathcal{C}} \eqdef \mathscr{A}ut_{\mathcal{C}/S} \times_{\Aut\mathcal{O}}(V\otimes \mathcal{O}_S).\]
    The \textit{vertex algebra sheaf} $\mathcal{V}_{\mathcal{C}/S}$ is the associated sheaf of sections over $\mathcal{C}$. That is,
    \[\mathcal{V}_{\mathcal{C}/S} = (V\otimes \pi_*\mathcal{O}_{\mathscr{A}ut_{\mathcal{C}/S}})^{\Aut\mathcal{O}},\]
    where $\pi\colon \mathscr{A}ut_{\mathcal{C}/S} \to \mathcal{C}$ is the natural forgetful map.
\end{definition}
We will sometimes omit the base $S$ when referring to the vertex algebra sheaf, writing $\mathcal{V}_{\mathcal{C}}$ instead of $\mathcal{V}_{\mathcal{C}/S}$.

The above definition is sometimes not the most useful presentation.
We opt to use a more explicit presentation of the sheaf using formal smoothings.
We follow \cite[\textsection 2.5]{DGT23} for our discussion.

\subsubsection{On smooth curves}
Let $C$ be a smooth projective curve over $\kk$. Then the vertex algebra bundle $V_C$ is the space of pairs $(P,t,a) \in \mathscr{A}ut_C\times V$ modulo the equivalence 
\[(P,t,a) \equiv (P,\rho(t),\Pi(\rho^{-1}) a)\]
for all $\rho\in \Aut\mathcal{O}$. 
The fiber of $V_C$ at a point $P\in C$ is isomorphic to $\mathscr{A}ut_P\times_{\Aut\mathcal{O}}V$. Up to a formal coordinate $t$ at $P$, we get the identification
\[\mathscr{A}ut_P\times_{\Aut\mathcal{O}}V \simeq_t \bigoplus_{k\ge 0} V_k \otimes \kk t^k \simeq V.\]
This gives a description of the bundle $V_C$ locally at each fiber.

Now, we describe the vertex algebra sheaf. Given a point $P\in C$ and an open neighborhood $U\subset C$ of $P$ admitting an \'etale map $U\to \A^1_{\kk}$ determined by a local coordinate $t\in \mathcal{O}_U$, we may identify the vertex algebra sheaf $\mathcal{V}_C$ on $U$ via the trivialization
\[\mathcal{V}_C(U) \simeq_t V\otimes \mathcal{O}_U.\] 
Therefore, $\mathcal{V}_C$ is a locally free (hence quasi-coherent) $\mathcal{O}_C$-module. The rank of $\mathcal{V}_C$ on each connected component of $C$ is then equal to the dimension of $V$ over $\kk$, which is usually infinite.

\subsubsection{On nodal curves}\label{sheaf on nodal curves}
Now, suppose $C$ is a nodal curve over $\kk$ with a node $Q$. Let $\eta\colon \widetilde C\to C$ be the partial normalization of $C$ at $Q$, and let $Q_\pm$ be the two preimages of $Q$ in $\widetilde C$. Choose formal coordinates $s_\pm$ at $Q_\pm$, respectively. Denote by $\widetilde\pi\colon \mathscr{A}ut_{\widetilde C}\to \widetilde C$ the natural forgetful morphism. The action of $\Aut\mathcal{O}$ on $V\otimes \widetilde\pi_* \mathcal{O}_{\mathscr{A}ut_{\widetilde C}}$ restricts to an action of $\Aut\mathcal{O}$ on
\[\bigoplus_{k\ge 0}V_k\otimes \mathcal{O}_{\widetilde C}(-kQ_+ - kQ_-)\otimes \widetilde\pi_* \mathcal{O}_{\mathscr{A}ut_{\widetilde C}}.\]
We define the sheaf $\widetilde{\mathcal{V}}_{\widetilde{C}}$ to be
\begin{equation}\label{eq:tilde V}
    \widetilde {\mathcal{V}}_{\widetilde{C}}\eqdef \left(\bigoplus_{k\ge 0}V_k\otimes \mathcal{O}_{\widetilde C}(-kQ_+-kQ_-)\otimes \widetilde\pi_* \mathcal{O}_{\mathscr{A}ut_{\widetilde C}}\right)^{\Aut\mathcal{O}}.
\end{equation}
The sheaves $\widetilde {\mathcal{V}}_{\widetilde{C}}$ and $\mathcal{V}_{\widetilde{C}}$ agree away from the points $Q_\pm$.
Looking at the formal neighborhoods $D_{Q_{\pm}} \subset \widetilde{C}$, we have the trivializations
\[\mathcal{V}_{\widetilde{C}}(D_{Q_\pm}) \simeq_{s_\pm} V \otimes_\kk \kk[\![s_\pm]\!],\qquad \widetilde{\mathcal{V}}_{\widetilde{C}}(D_{Q_\pm}) \simeq_{s_\pm} \bigoplus_{k\ge 0}V_k\otimes_\kk s_\pm ^k\kk[\![s_\pm]\!].\]
These sections naturally agree as $\kk[\![s_\pm]\!]$-modules, so we have an isomorphism $\widetilde{\mathcal{V}}_{\widetilde{C}} \cong \mathcal{V}_{\widetilde{C}}$ of $\mathcal{O}_{\widetilde{C}}$-modules. That said, we prefer to think of them separately.

Consider the exact sequence of $\cal{O}_{\widetilde{C}}$-modules
\[ 
\begin{tikzcd}
0 \ar[r] & \cal {O}_{\widetilde{C}}(-Q_+-Q_-) \otimes \widetilde{\mathcal{V}}_{\widetilde{C}} \ar[r] & \widetilde{\mathcal{V}}_{\widetilde{C}} \ar{r}{} & V_{Q_+} \oplus V_{Q_-} \ar[r] & 0,
\end{tikzcd}
\] 
where $V_{Q_\pm}$ is the skyscraper sheaf on $C$ supported at $Q_\pm$ with space of sections isomorphic to $V$ via the choice of the coordinate $s_\pm$. Pushing forward by $\eta$, we obtain the diagram
\begin{equation}
\label{eq:diagramVC} 
\begin{tikzcd}
0 \ar[r] & \eta_* \left(\cal{O}_{\widetilde{C}}(-Q_+-Q_-) \otimes \widetilde{\mathcal{V}}_{\widetilde{C}}\right) \ar[r] & \eta_* \widetilde{\mathcal{V}}_{\widetilde{C}} \ar{r}{q} & V_{Q}^{\oplus 2} \ar[r] \ar[->>]{d}{\gamma\,\circ\, \pi_1 -\pi_2} & 0 \\
&&& V_Q.
	\end{tikzcd}
\end{equation}
Here the top row is exact, and $\pi_i\colon V_{Q} ^{\oplus 2} \rightarrow V_{Q}$ are the natural projections for $i=1,2$. In this picture, we may define the sheaf of vertex algebras $\mathcal{V}_C$ on $C$ to be
\[\mathcal{V}_C = \ker\left( (\gamma\circ \pi_1 -\pi_2) \circ q\right).\] 
That is, $\mathcal{V}_C$ is the equalizer of the morphisms $\gamma\circ \pi_1\circ q$ and $\pi_2\circ q$.


We can give an explicit local description of the sections of $\mathcal{V}_C$ at a node $Q$ as follows. The formal neighborhood of the node $Q\in C$ is given in terms of a pair of formal variables $s_\pm$ as
\[D_Q = \Spf \widehat{\mathcal{O}}_Q \simeq \Spf(\kk[\![s_+,s_-]\!]/(s_+s_-)).\]
Based on the above description of $\mathcal{V}_C$, the space of sections $\mathcal{V}_C(D_Q)$ is given by the kernel of the map
\[\eta_*\widetilde{\mathcal{V}}_{\widetilde{C}}(D_Q) = \bigoplus_{k\ge 0} V_k\otimes_\kk \left(s_+^k\kk[\![s_+]\!]\oplus s_-^k\kk[\![s_-]\!]\right)\longrightarrow V\]
induced by
\[\left(a\otimes s_+^{\deg(a)}f(s_+), b\otimes s_-^{\deg(b)}g(s_-)\right) \longmapsto f(0)\gamma(a) - g(0)b\]
for homogeneous $a,b\in V$, and for all $f(s_+)\in \kk[\![s_+]\!]$ and $g(s_-)\in \kk[\![s_-]\!]$.
Hence, $\mathcal{V}_C(D_Q)$ is spanned by elements of the form
\[\Big(a\otimes s_+^{\deg(a)},\, (-1)^{\deg(a)}\sum_{i\ge 0}L_1^{(i)}(a)\otimes s_-^{\deg(a)-i} \Big)\]
and
\[\Big(b\otimes s_+^{\deg(b)+1}f(s_+),\, c\otimes s_-^{\deg(c)+1}g(s_-)\Big)\]
for homogeneous elements $a,b,c\in V$ and power series $f \in \kk[\![s_+]\!]$ and $g\in \kk[\![s_-]\!]$. We may denote elements of the first form in shorter notation as
\[\Big(a\otimes s_+^{\deg(a)},\, (-1)^{\deg(a)}\sum_{i\ge 0}L_1^{(i)}(a)\otimes s_-^{\deg(a)-i} \Big) = \Big( s_+^{\deg(a)}a, (-1)^{\deg(a)}e^{s_-^{-1}L_1}[s_-^{\deg(a)}a]\Big).\]
Given the isomorphism $\widetilde{\mathcal{V}}_{\widetilde{C}} \simeq \mathcal{V}_{\widetilde{C}}$, we may also identify $\mathcal{V}_{C}$ as the coequalizer of the morphisms $\gamma\circ \pi_1\circ q$ and $\pi_2\circ q$ from $\eta_*\mathcal{V}_{\widetilde{C}}$ to $V_Q$.

\subsubsection{On formal smoothings}\label{smoothings}
In order to describe the vertex algebra sheaf $\mathcal{V}_{\mathcal{C}/S}$ associated to a relative nodal curve $\mathcal{C}\to S$, it suffices to describe this for certain special cases. In particular, it suffices to consider so-called smoothing families, which we describe here. Once again, we follow \cite{DGT23} and \cite{DGK23}.
\begin{construction}\label{constr:smoothing family}
    Let $R$ be a $\kk$-algebra such that $S_0 \eqdef \Spec(R)\to \Spec(\kk)$ is smooth. Let $\mathscr{C}_0\to S_0$ be a semistable curve with a section $Q$ and $n$ distinct smooth points given by sections $P_\bullet$. Assume that $\mathscr{C}_0\setminus P_\bullet(S_0)$ is affine over $S_0$. After an \'etale base change of $S_0$ of degree two, we can normalize $\mathscr{C}_0$ and obtain a smooth family of $(n+2)$-pointed curves $\widetilde{\mathscr{C}}_0\to S_0$ with sections $P_\bullet\sqcup(Q_+,Q_-)$, where $Q_\pm(S_0)\subset \widetilde{\mathscr{C}}_0$ are the preimages of the node in $\mathscr{C}_0/S_0$. Fix formal coordinates $s_\pm$ at $Q_\pm(S_0)$, respectively. Such coordinates determine what we call a \textit{smoothing} of $(\mathscr{C}_0,P_\bullet)$ over $S \eqdef \Spec(R[\![q]\!])$. That is, we have a flat family $\mathscr{C}\to S = \Spec(R[\![q]\!])$ (called a \textit{smoothing family}) with sections $P_\bullet = (P_1,\ldots,P_n)$ such that the general fiber is smooth and the special fiber is identified with $\mathscr{C}_0\to S_0$. 
    The formal coordinates at $Q_\pm (S_0)$ extends to formal coordinates at $Q_\pm(S)$, which we also denote by $s_\pm$. These formal coordinates have the property that locally around the node, the curve $\mathscr{C}$ is defined by the equation $s_+s_- = q$. The existence of smoothing families holds over the formal base $S = \Spec(R[\![q]\!])$, but it fails over a more general base. One has that $\widetilde {\mathscr{C}}\setminus P_\bullet(S)$ and $\mathscr{C}\setminus P_\bullet(S)$ are affine over $S$.
\end{construction}


Here we define the sheaf of vertex algebras $\mathcal{V}_{\mathscr{C}}$ over a smoothing family $\mathscr{C}\to S$. Similar to before, it is sufficient to describe $\mathcal{V}_{\mathscr{C}}$ on the formal neighborhood of the node $Q$, which we denote by
\[\frak{D}_Q \simeq_{s_\pm} \Spf R[\![s_+,s_-,q]\!]/(s_+s_--q).\]
The completed local ring $\widehat{\mathcal{O}}_Q$ consists of elements of the form $\sum_{i,j\ge 0}\alpha_{i,j}s_+^is_-^j$ for $\alpha_{i,j}\in R$. After identifying $s_+ = (s_+,\frac{q}{s_-})$ and $s_- = (\frac{q}{s_+},s_-)$, the ring $\widehat{\mathcal{O}}_Q$ is realized as the subring of $R(\!(s_+)\!)[\![q]\!]\oplus R(\!(s_-)\!)[\![q]\!]$ consisting of elements of the form
\[\sum_{i,j\ge 0}\alpha_{i,j}\left(s_+^{i-j}q^j,s_-^{j-i}q^i\right),\]
where $\alpha_{i,j}\in R$. The space of sections of $\mathcal{V}_{\mathscr{C}}$ over $\frak{D}_Q$ is generated by elements of the form
\begin{equation}\label{eq:typeA}
    \sum_{i,j\ge 0}\alpha_{i,j}\left(a\otimes s_+^{\deg(a)+i-j}q^j,(-1)^{\deg(a)}\sum_{k\ge 0}L_1^{(k)}(a)\otimes s_-^{\deg(a)-k+j-i}q^i\right)
\end{equation}
and
\begin{equation}\label{eq:typeB}
    \sum_{i,j\ge 0}\alpha_{i,j}\left(b\otimes s_+^{\deg(b)+i-j+1}q^j,c\otimes s_-^{\deg(c)+j-i}q^i\right)
\end{equation}
for homogeneous $a,b,c\in V$ and $\alpha_{i,j}\in R$. In the quotient $q=0$, we recover the definition of $\mathcal{V}_C$ for a nodal curve as in \zcref{sheaf on nodal curves}.

\begin{remark}\label{rmk:vertex algebra sheaf general case}
    Now that we have defined the vertex algebra sheaf explicitly on formal smoothings, we may glue using \cite{BL} to obtain the vertex algebra sheaf $\mathcal{V}_{\mathcal{C}}$ associated to a general relative nodal curve $\mathcal{C}\to S$. Indeed, the sheaf $\mathcal{V}_{\mathscr{C}}$ is constructed as the kernel of a map from $\eta_*\widetilde{\mathcal{V}}_{\widetilde{\mathscr{C}}}$ to $V_Q$. Both of these are $q$-torsion-free, so the kernel $\mathcal{V}_{\mathscr{C}}$ is also $q$-torsion-free. This is the same procedure as used in \cite[\textsection 8.6]{DGT23}, and we will use this procedure for other sheaves later on.
\end{remark}

\subsection{The structure of the vertex algebra sheaf}
For a smooth curve $C$, the sheaf $\mathcal{V}_C$ is filtered by the sheaves $\mrm{G}_k\mathcal{V}_C$, which are defined as the sheaves on $C$ of sections of the finite rank vector bundles $\mathscr{A}ut_C \times_{\Aut\mathcal{O}}\mrm{G}_kV$, where $\mrm{G}_kV = \bigoplus_{s\leq k}V_{s}$. While the action of $\Aut\mathcal{O}$ on $\mrm{G}_kV$ is well-defined, the action of $\Aut\mathcal{O}$ on $V_{k}$ is only well-defined modulo $\mrm{G}_{k-1}V$.

\begin{lemma}
    Let $V$ be an $\N$-graded vertex $\kk$-algebra with an $\Aut\mathcal{O}$-action. Then
    \[\gr_\bullet \mathcal{V}_C \cong \bigoplus_{k\ge 0}\left(\omega_C^{\otimes -k}\right)^{\oplus \dim V_{k}}.\]
\end{lemma}
\begin{proof}
    Consider $V_k$ as the quotient $\mrm G_kV/\mrm G_{k-1}V$. Let $a\in V_k$, and assume that $C$ is smooth.
    Then $a$ is a primary element in the associated graded since by definition we have 
    \[\Pi(\rho_z)a - (\rho'(z))^{-k}a \in \mrm G_{k-1}V[\![z]\!]\] 
    for all $\rho\in \Aut\mathcal{O}$. Therefore, the subsheaf
    \[\mathscr{A} \eqdef (\kk a\otimes \pi_*\mathcal{O}_{\mathscr{A}ut_C})^{\Aut\mathcal{O}}\]
    is isomorphic to $\omega_C^{\otimes -k}$. Therefore, the claimed isomorphism holds when $C$ is smooth; we may extend to the nodal case using the same argument as in \cite[Lemma 2.6.1]{DGT23}.
\end{proof}
Using the same argument as in \cite[Lemma 2.6.3]{DGT23}, we obtain the following:
\begin{lemma}\label{lem: sections gr}
    Let $(C,P_\bullet)$ be a smooth, $n$-pointed curve. Then
    \[H^0(C\setminus P_\bullet,\mathcal{V}_C) \cong H^0(C\setminus P_\bullet,\gr_\bullet\mathcal{V}_C).\]
\end{lemma}

\begin{remark}
    These isomorphisms are the reason we imposed that an $\Aut\mathcal{O}$-action on $V$ must have the property that $\Pi(\rho_z)a \equiv (\rho'(z))^{-\Delta}a$ modulo lower degree terms.
\end{remark}

\subsection{The logarithmic $\mathcal{D}$-module structure}\label{sec:logstratification}
Over a field of characteristic 0, the sheaf $\mathcal{V}_C$ on a nodal curve $C$ admits a natural flat logarithmic connection $\nabla\colon \mathcal{V}_C\to \mathcal{V}_C\otimes \omega_C$. This same flat logarithmic connection still exists over any algebraically closed field $\kk$, but it is no longer the full picture. 
The appropriate replacement for our setting is a so-called \textit{logarithmic $\mathcal{D}$-module structure} using principal parts in the sense of \cite{BerthelotOgus}. More specifically, we will present this structure in terms of what we call a \textit{logarithmic PD-connection}, where PD is taken from the notion of a PD (or divided power) structure for rings. See \zcref{stratification discussion} for a discussion on the different possible structures used to refer to a $\mathcal{D}$-module.

\subsubsection{On smooth curves}
First, we recall the theory of principal parts. 
\begin{deflem}
Let $f\colon X\to S$ be a morphism of schemes. We define the \textit{sheaf of principal parts of order $n$}, denoted $\mathcal{P}^n_{X/S}$, to be
\[\mathcal{P}^n_{X/S} = (\mathcal{O}_X\otimes_{f^{-1}(\mathcal{O}_S)}\mathcal{O}_X)/\cal J^{n+1},\]
where $\cal J = \ker(\mathcal{O}_X\otimes_{f^{-1}(\mathcal{O}_S)}\mathcal{O}_X \to \mathcal{O}_X)$ is the kernel of the multiplication map.
This sheaf is isomorphic to the sheaf of rings of the $n$-th infinitesimal thickening $X_{(n)}$ of $X$ over $S$. Specifically,
\[\mathcal{P}^n_{X/S} \cong \Delta_X^{-1}\mathcal{O}_{X\times_S X} / \mathcal{I}^{n+1},\]
where $\Delta_X\colon X\to X\times_S X$ is the diagonal map, and $\mathcal{I}$ is the kernel of the (surjective) map $\Delta_X^{\#}\colon\Delta_{X}^{-1}\mathcal{O}_{X\times_S X} \to \mathcal{O}_X$.
\end{deflem}
From the definition, we can see that $\mathcal{P}^n_{X/S}$ is a quasi-coherent $\mathcal{O}_X$-module. There are several maps associated to $\mathcal{P}^n_{X/S}$:
\begin{itemize}
    \item For all $n\ge 0$, there are two maps
    \[d^n_0,d^n_1\colon \mathcal{O}_X\to \mathcal{P}^n_{X/S}\]
    given by mapping to the first and second component of the tensor respectively. We use $d^n_0$ to endow $\mathcal{P}^n_{X/S}$ with a left $\mathcal{O}_X$-module structure, and we define $\delta = d^n_1 - d^n_0$. It is straightforward to see that the multiplication kernel $\cal J$ is generated over $\mathcal{O}_X$ by the image of $\delta$. On local sections $a,b$ of $\mathcal{O}_X$, we can see that
    \[\delta(ab) = a\,\delta(b)+b\,\delta(a) + \delta(a)\,\delta(b).\]
    This is a generalized Leibniz rule, where we allow for a higher-order term $\delta(a)\,\delta(b)$.
    
    In fact, we can characterize $\mathcal{P}^n_{X/S}$ as the unique left $\mathcal{O}_X$-algebra generated by local sections of the form $\delta(a)$ for all local sections $a$ of $\mathcal{O}_X$ such that the following identities hold for all local sections $a,b$ of $\mathcal{O}_X$:
    \begin{enumerate}
        \item $\delta(ab) = a\,\delta(b)+b\,\delta(a) + \delta(a)\,\delta(b)$;
        \item $\delta(a+b) = \delta(a)+\delta(b)$;
        \item $\delta(a)=0$ for all $a\in \im (f^{-1}(\mathcal{O}_S)\to \mathcal{O}_X)$.
    \end{enumerate}
    In terms of this characterization, the map $d^n_1$ is given locally by $a\mapsto a+\delta(a)$.
    \item For all $n\ge m$, there is a canonical projection
    \[\varphi_{m,n}\colon \mathcal{P}^n_{X/S}\twoheadrightarrow \mathcal{P}^m_{X/S}\]
    whose kernel is $\cal J^{m+1}/\cal J^{n+1}$.
    \item For $n,m\ge 0$, there are comultiplication maps
    \[\delta^{n,m}\colon \mathcal{P}^{n+m}_{X/S} \to \mathcal{P}^{n}_{X/S}\otimes \mathcal{P}^{m}_{X/S}\]
    given on local sections $x,y$ of $\mathcal{O}_X$ by $\overline{x\otimes y} \mapsto \overline{x\otimes 1}\otimes \overline{1\otimes y}$. In terms of the generators $\delta(a)$, we map
    \[\delta(a)\mapsto \delta(a)\otimes 1 + 1\otimes \delta(a).\]
\end{itemize}
\begin{warning}
    It is common in algebraic geometry to work with stratifications of spaces, which are not necessarily related to stratifications on $\mathcal{O}_X$-modules.
\end{warning}
Now, we describe $\mathcal{D}$-module structures using principal parts. 
\begin{deflem}
    We define
    \begin{equation}
        \mathcal{D}_{X/S} = \varinjlim_{n\ge 0}\Hom_{\mathcal{O}_X}\left(\mathcal{P}^n_{X/S},\mathcal{O}_X\right) = \Hom_{\mathcal{O}_X}\left(\varprojlim_{n\ge 0}\mathcal{P}^n_{X/S},\mathcal{O}_X\right),
    \end{equation}
    where the morphism $\Hom_{\mathcal{O}_X}\left(\mathcal{P}^m_{X/S},\mathcal{O}_X\right) \to \Hom_{\mathcal{O}_X}\left(\mathcal{P}^n_{X/S},\mathcal{O}_X\right)$
    for $m\le n$ is induced from $\varphi_{m,n}\colon \mathcal{P}^n_{X/S} \to \mathcal{P}^m_{X/S}$.
    We say that an $\mathcal{O}_X$-module $\mathcal{E}$ is a $\mathcal{D}_{X/S}$-module if it admits an $\mathcal{O}_X$-module map $\mathcal{D}_{X/S}\otimes \mathcal{E} \to \mathcal{E}$.
\end{deflem}
The second presentation of $\mathcal{D}_{X/S}$ follows from the fact that the category of quasi-coherent $\mathcal{O}_X$-modules admits colimits.
\begin{example}
    Consider the $A$-algebra $A[t]$ for a ring $A$. Then the space of global sections of the sheaf of differential operators $\mathcal{D}_{\Spec(A[t])/\Spec(A)}$ is isomorphic to the $A[t]$-algebra of differential operators in one formal variable $t$. This is the associative $A[t]$-algebra generated by the symbols $\{\partial_t^{(m)}\}_{m\ge 0}$ such that the following relations hold for all $m,k\ge 0$:
    \[\partial_t^{(0)} = 1,\qquad \partial_t^{(m)}\partial_t^{(k)} = \binom{m+k}{k}\partial_t^{(m+k)},\qquad [\partial_t^{(m)},t] = \partial_t^{(m-1)}.\]
\end{example}
The primary issue with the sheaf $\mathcal{D}_{X/S}$ is that it is defined in terms of duals.
To avoid issues with taking double duals over $\mathcal{O}_X$, we work with the following alternative to a $\mathcal{D}_{X/S}$-module:
\begin{definition}\label{PD-connection def}
    Let $f\colon X\to S$ be a morphism of schemes, and let $\mathcal{E}$ be an $\mathcal{O}_X$-module. A \textit{PD-connection} is a collection of right $\mathcal{O}_X$-linear morphisms
    \[\Theta_n\colon \mathcal{E}\to \mathcal{E}\otimes \mathcal{P}^n_{X/S}\]
    for all $n\ge 0$ such that:
    \begin{enumerate}
        \item $\Theta_0 = \id_{\mathcal{E}}$.
        \item The morphisms $\{\Theta_n\}_{n\ge 0}$ are compatible with the restriction morphisms. That is, for all $n\ge m$ we have
        \[(\id_{\mathcal{E}}\otimes \,\varphi_{m,n})\circ \Theta_n = \Theta_m.\]
        \item The morphisms $\{\Theta_n\}_{n\ge 0}$ are compatible with the morphisms $\delta^{m,n}$. That is, for all $n,m\ge 0$ we have a commutative diagram
        \begin{equation*}
            \begin{tikzcd}
                \mathcal{E} \ar{r}{\Theta_n} \ar{d}{\Theta_{m+n}}& \mathcal{E}\otimes \mathcal{P}^{n}_{X/S} \ar{d}{\Theta_m\otimes \id_n}\\
                \mathcal{E} \otimes \mathcal{P}^{m+n}_{X/S}\ar{r}{\id_{\mathcal{E}}\otimes \delta^{n,m}} & \mathcal{E}\otimes \mathcal{P}^m_{X/S}\otimes\mathcal{P}^n_{X/S}
            \end{tikzcd}
        \end{equation*}
    \end{enumerate}
\end{definition}
It is straightforward to see that $d^n_1$ has the structure of a PD-connection on $\mathcal{O}_X$.
\begin{remark}\label{stratification discussion}
    We take inspiration from the notion of a stratification on an $\mathcal{O}_X$-module associated to a morphism of schemes $X\to S$ from \cite{BerthelotOgus}. 
    A stratification on a quasi-coherent $\mathcal{O}_X$-module $\mathcal{E}$ is equivalent to a quasi-coherent $\mathcal{O}_X$-module on the de Rham stack $(X/S)^{\mathrm{dR}} = \operatorname{coeq}\left(\widehat{X\times_S X}\rightrightarrows X\right)$. If $X\to S$ is a smooth morphism over $\Q$, then a stratification on $\mathcal{E}$ is equivalent to a flat connection. If the morphism $X\to S$ is smooth and $S$ is a $\Q$-scheme, then furthermore a stratification is equivalent to a PD-connection or an action of $\mathcal{D}_{X/S}$ (see \cite[Proposition 2.11]{BerthelotOgus}).
    If $X\to S$ is smooth and $S$ is an $\F_p$-scheme, then the data of a stratification is equivalent to infinite Frobenius descent data \cite{Gieseker}. This is a very strong condition in positive characteristic as first-order Frobenius descent is more common to impose in practice.
\end{remark}

First, we will show that the sheaf of vertex algebras on a smooth projective curve $C$ over $\kk$ admits a PD-connection. It is straightforward to see that $\mathcal{P}^{n}_{C/\kk}$ is a free $\mathcal{O}_C$-module of rank $n+1$ for all $n\ge 0$. Indeed, if we choose an open subset $U\subset C$ with an \'etale map $U\to \A^1_{\kk}$ defining a coordinate, then
\[\mathcal{P}^n_{C/\kk}(U) \simeq_t \bigoplus_{0\le k\le n}\mathcal{O}_U\,(\delta t)^k = \mathcal{O}_U^{\oplus (n+1)}.\]
One of the challenges with defining the PD-connection on $\mathcal{V}_C$ is that we must show our construction is equivariant with respect to $\Aut\mathcal{O}$.
To show this, we need a formal calculus result for the $\Aut\mathcal{O}$-action on $V$.
\begin{lemma}\label{lem:formalcalcequiv}
    We have the following conjugation formula in $\End(V)[\![z,w]\!]$:
    \begin{equation}\label{eq:change of variables}
        e^{wT}e^{w\partial_z} = \Pi(\rho_z)e^{(\rho(z+w)-\rho(z))T}e^{w\partial_z}\Pi(\rho_z)^{-1}.
    \end{equation}
\end{lemma}
\begin{proof}
    We have
    \begin{align*}
        e^{w\nabla}\Pi(\rho_z) &= e^{w\partial_z}e^{wT}e^{-zT}\Pi(\rho) e^{\rho(z)T}\\
        &= e^{wT}e^{w\partial_z}\left[e^{-zT}\right]\Pi(\rho) e^{w\partial_z}\left[e^{\rho(z)T}\right]e^{w\partial_z}\\
        &= e^{-zT}\Pi(\rho)e^{w\partial_z}\left[e^{\rho(z)T}\right]e^{w\partial_z}\\
        &= \Pi(\rho_z) e^{-\rho(z)T}e^{w\partial_z}\left[e^{\rho(z)T}\right]e^{w\partial_z}\\
        &= \Pi(\rho_z) e^{(\rho(z+w)-\rho(z))T}e^{w\partial_z}.
    \end{align*}
\end{proof}
This identity follows from the usual change-of-variables formula in \cite[Eq. 6.6.1]{FBZ} in characteristic 0. Indeed, given $w = \rho(z)$, the connection $\widetilde{\nabla}_{\partial_w}$ for the coordinate $w$ and the connection $\nabla_{\partial_z}$ are related via
\[\widetilde\nabla_{\partial_z} = \partial_z + T = \Pi(\rho_z)\circ\partial_z\circ\Pi(\rho_z)^{-1} + \rho'(z)\Pi(\rho_z)\circ T\circ\Pi(\rho_z)^{-1}.\]
This is precisely the $w^1$ term of \zcref[noname]{eq:change of variables}.

Now, we give the generalization of the flat connection on $\mathcal{V}_C$ for a smooth, projective curve $C$ over $\kk$.
\begin{lemma}\label{PD-connection smooth curve}
    Let $C$ be a smooth, projective curve over $\kk$. Let $U\subset C$ be an open subset with a global coordinate $t$ (say coming from an \'etale map $U\to \A^1_{\kk}$). Using the trivialization $\mathcal{V}_C|_U \simeq_t V\otimes \mathcal{O}_U$, we define a PD-connection $\Theta=\{\Theta_n\}_{n\ge 0}$ on $\mathcal{V}_C$ by setting
    \[\Theta_n(a\otimes f(t)) = \sum_{0\le i\le j\le n}T^{(i)}(a)\otimes \partial_t^{(j-i)}[f(t)]\otimes (\delta t)^{j}.\]
    This assignment is canonical; that is, the definition is independent of the choice of the coordinate $t$.
\end{lemma}
\begin{proof}
    Let $U\subset C$ be an open subset as described in the statement above. To make the following discussion simpler, we write
    \[\Theta_n(a\otimes f) = e^{(\delta t)T}(a)e^{(\delta t)\partial_t}(f) \mod (\delta t)^{n+1}.\]
    First, for any $f,g\in \mathcal{O}_U$ and $a\in V$ we have
    \[\Theta_n(a\otimes fg) = e^{(\delta t)T}(a)e^{(\delta t)\partial_t}(fg) = e^{(\delta t)T}(a)e^{(\delta t)\partial_t}(f)e^{(\delta t)\partial_t}(g) \mod (\delta t)^{n+1},\]
    hence the map $\Theta_n\colon \mathcal{V}_C\to \mathcal{V}_C\otimes \mathcal{P}^n_{C/\kk}$ is right $\mathcal{O}_U$-linear. We also clearly have $\Theta_0=\id$ by definition.
    
    Now we check the rest of the properties. Denote by $\Theta_n(a\otimes f)[(\delta t)^j]$ the coefficient of $(\delta t)^j$ in $\Theta_n(a\otimes f)$ for $n\ge j$. That is,
    \[\Theta_n(a\otimes f)[(\delta t)^j] = \sum_{0\le i\le j}T^{(i)}(a)\otimes \partial_t^{(j-i)}(f).\]
    We give two observations:
    \begin{enumerate}
        \item It is straightforward to see that $\Theta_n(a\otimes f)[(\delta t)^j]$ is independent of $n\ge j$, hence we write
        \[\nabla^{(j)}(a\otimes f) = \Theta_j(a\otimes f)[(\delta t)^j],\]
        which extends to a well-defined endomorphism $\nabla^{(j)} \in \End(\mathcal{V}_C)$ for $j\ge 0$. Also, this independence of $n\ge j$ shows that the maps $\{\Theta_n\}_{n\ge 0}$ are compatible with the projections $\varphi_{n,j}$ for $n\ge j$.
        \item For $n,m\ge 0$, we have
        \[(\nabla^{(n)}\circ\nabla^{(m)})(a\otimes f) = \binom{n+m}{m}\nabla^{(n+m)}(a\otimes f).\]
        This shows compatibility with the comultiplication map.
    \end{enumerate}
    It remains to check that this definition is coordinate independent. To see this, we check at the level of formal neighborhoods. Given a point $P\in C$, a choice of formal coordinate induces a trivialization $D_P = \Spf \kk[\![t]\!]$.
    Suppose we pick another coordinate $\rho = \rho(t)$ for some $\rho\in \Aut\mathcal{O}$. We define the map $\widetilde \Theta_{n}$ given by
    \[\widetilde\Theta_{n}(a\otimes f) = e^{(\delta \rho)T}(a)e^{(\delta (\rho))\partial_\rho}(f) \mod (\delta\rho)^{n+1}.\]
    Note that in terms of $\delta t$, we have
    \[\delta\rho = \delta(\rho(t)) = \sum_{i\ge 1}\partial_t^{(i)}[\rho(t)](\delta t)^i \mod (\delta t)^{n+1}.\]
    Since $\rho'(0)\in \kk^\times$ and $t = \rho^{-1}(\rho(t))$, we have that $(\delta\rho)^{n+1}\equiv 0$ if and only if $(\delta t)^{n+1}\equiv 0$.
    From \zcref{lem:formalcalcequiv}, we have
    \begin{align*}
        e^{(\delta t) T}e^{(\delta t)\partial_t}\Pi(\rho_t) &= \Pi(\rho_t)e^{(\rho(t+\delta t)-\rho(t))T}e^{(\delta t)\partial_t} \mod (\delta t)^{n+1}.
    \end{align*}
    But we have $\delta\rho = \rho(t+\delta t) - \rho(t)$, and by change of variables we have 
    \[e^{(\delta t) \partial_t} = e^{(\rho(t+\delta t)-\rho(t))\partial_\rho} = e^{\delta\rho\partial_\rho}.\] 
    Therefore, we have
    \[\Pi(\rho_t)e^{\delta\rho\, T}e^{\delta\rho\partial_\rho}\Pi(\rho_t)^{-1} = e^{(\delta t)T}e^{(\delta t)\partial_t} \mod (\delta t)^{n+1}.\]
    This applies on the formal neighborhood of any point $P\in C$, so the PD-connection is naturally coordinate independent. Therefore, we obtain a well-defined PD-connection for $\mathcal{V}_C$ on all of $C$.
\end{proof}


\subsubsection{On nodal curves}
We now describe a generalization of the PD-connection on $\mathcal{V}_C$ as described in \zcref{PD-connection smooth curve} for a nodal curve $C$. This will make use of the log structure on $C$.
To make this notion precise, we introduce the following generalization of sheaves of principal parts in terms of local sections.
\begin{definition}
    Let $f\colon X\to S$ be a morphism of pre-log schemes, with $\alpha\colon \cal M_X\to (\mathcal{O}_X,\cdot)$ the log structure on $X$. We define $\mathcal{P}_{X/S,\log}$ to be the left $\mathcal{O}_X$-algebra generated over $\mathcal{O}_X$ by local sections of the form $\delta(a)$ and $\delta_{\log}(m)$ with $a\in \mathcal{O}_X$ and $m\in \cal M_X$ such that for all $a,b\in \mathcal{O}_X$ and $m_1,m_2\in \cal M_X$ we have:
    \begin{enumerate}
        \item $\delta(a+b) = \delta(a)+\delta(b)$;
        \item $\delta(ab) = b\,\delta(a) + a\,\delta(b) + \delta(a)\, \delta(b)$;
        \item $\delta_{\log}(m_1m_2) = \delta_{\log}(m_1) + \delta_{\log}(m_2) + \delta_{\log}(m_1)\,\delta_{\log}(m_2)$;
        \item $\delta(r) = 0$ for all $r\in \im\left(f^{-1}(\mathcal{O}_S)\to \mathcal{O}_X\right)$;
        \item $\delta_{\log}(m) = 0$ for all $m\in \im(f^{-1}(\cal M_S)\to \cal M_X)$;
        \item $\delta(\alpha(m)) = \alpha(m)\,\delta_{\log}(m)$.
    \end{enumerate}
    We assign a right filtration to $\mathcal{P}_{X/S,\log}$ by declaring elements of the form $\delta(a)$ and $\delta_{\log}(m)$ to have degree 1 for $a\in \mathcal{O}_X$ and $m\in \cal M_X$, denoted $(\mathcal{P}_{X/S,\log})_{\ge \bullet}$. We define the \textit{sheaf of logarithmic principal parts of order $\le n$} on $X$, denoted $\mathcal{P}^n_{X/S,\log}$, to be
    \[\mathcal{P}^n_{X/S,\log} \cong \mathcal{P}_{X/S,\log}/(\mathcal{P}_{X/S,\log})_{\ge n+1}.\]
\end{definition}
This definition of log principal parts was inspired by the log differentials of Kato.
\begin{enumerate}
    \item As with usual principal parts, for all $n\ge 0$ we have a pair of maps
    \[d^n_0,d^n_1\colon \mathcal{O}_X \to \mathcal{P}^n_{X/S,\log}.\]
    The map $d^n_0$ corresponds to the left $\mathcal{O}_X$-algebra structure on $\mathcal{P}^n_{X/S,\log}$, and $d^n_1$ is defined on a local section $a\in \mathcal{O}_X$ by
    \[d^n_1(a) = \delta(a) + a.\]
    There is also a map
    \[d^n_{\log}\colon \mathcal{M}_X \to \mathcal{P}^n_{X/S,\log}\]
    given on a local section $m\in \mathcal{M}_X$ by $d^n_{\log}(m) = \delta_{\log}(m)$.
    \item There are natural projections 
    \[\varphi_{m,n}\colon \mathcal{P}^n_{X/S,\log}\twoheadrightarrow \mathcal{P}^m_{X/S,\log}\] 
    for $n\ge m$ given by taking the quotient by $(\mathcal{P}_{X/S,\log})_{\ge m+1} / (\mathcal{P}_{X/S,\log})_{\ge n+1}$.
    \item There are natural comultiplication morphisms 
    \[\delta^{n,m}\colon \mathcal{P}^{n+m}_{X/S,\log} \to \mathcal{P}^{n}_{X/S,\log}\otimes_{\mathcal{O}_X} \mathcal{P}^{m}_{X/S,\log}\] 
    induced from the comultiplication map
$\delta\colon \mathcal{P}_{X/S,\log} \to \mathcal{P}_{X/S,\log}\otimes \mathcal{P}_{X/S,\log}$ given by
\begin{align}
    \delta(a)&\mapsto \delta(a)\otimes 1 + 1\otimes \delta(a)\\
    \delta_{\log}(m)&\mapsto \delta_{\log}(m)\otimes 1 + 1\otimes \delta_{\log}(m) + \delta_{\log}(m)\otimes \delta_{\log}(m)
\end{align}
\end{enumerate}
Using the maps associated to $\mathcal{P}^n_{X/S,\log}$, we may define a \textit{logarithmic PD-connection} on an $\mathcal{O}_X$-module $\mathcal{E}$ to be a collection of maps $\Theta_n\colon \mathcal{E}\to \mathcal{E}\otimes \mathcal{P}^n_{X/S,\log} $ satisfying similar axioms to \zcref{PD-connection def}, except we use the morphisms associated to $\mathcal{P}^n_{X/S,\log}$ written above.
\begin{remark}
    As mentioned before, a PD-connection on $\mathcal{E}$ over a smooth morphism of schemes $X\to S$ is equivalent to a stratification on $\mathcal{E}$ \cite[Proposition 2.11]{BerthelotOgus}. Given that we have defined a logarithmic PD-connection, one may also naturally define a logarithmic $\mathcal{D}$-module structure. We believe that these two notions are equivalent if $X\to S$ is log smooth, but we leave it open to determine if this is the case. We are also curious if there is a log de Rham stack $(X/S)^{\mathrm{dR}}_{\log}$ such that the quasi-coherent modules on this stack are equivalent to quasi-coherent $\mathcal{O}_X$-modules with a log stratification. A version of the log de Rham stack for flat connections with nilpotent $p$-curvature was constructed in \cite{Barz25}.
\end{remark}



Let $C$ be a nodal curve over $\kk$. We refer the reader to \cite{ACGHOSS} for a discussion on the log structure on nodal curves.
We now give a description of the sheaf of log principal parts $\mathcal{P}^n_{C/\kk,\log}$ on the formal neighborhood of a node $Q$.
\begin{lemma}\label{log principal parts at node}
    Let $C$ be a nodal curve over $\kk$ with a single node $Q$, and let $n\ge 0$. Then the sheaf $\mathcal{P}^n_{C/\kk,\log}$ is the unique $\mathcal{O}_C$-algebra such that the following hold:
    \begin{enumerate}
        \item Over a smooth point $P$ with formal neighborhood $D_P = \Spf(\kk[\![t]\!])$, the sheaf of log principal parts $\mathcal{P}^n_{C/\kk,\log}(D_P)$ is generated over $\kk[\![t]\!]$ by the form $\delta t$, subject to the relation $(\delta t)^{n+1}=0$. That is, $\mathcal{P}^n_{C/\kk,\log}$ agrees with $\mathcal{P}^n_{C/\kk}$ on the complement $C\setminus Q$ of the node.
        \item Over a node $Q$ with formal neighborhood $D_Q = \Spf \kk[\![s_+,s_-]\!]/(s_+s_-)$, the sheaf of log principal parts $\mathcal{P}^n_{C/\kk,\log}(D_Q)$ is generated over ${\kk[\![s_+,s_-]\!]}/{(s_+s_-)}$ by the forms $\delta_{\log} s_+ , \delta_{\log} s_-$, subject to the relations
        \[0 = \delta_{\log}(s_+s_-) = \delta_{\log}s_+ + \delta_{\log}s_- + \delta_{\log}s_+\delta_{\log}s_-,\qquad  (\delta_{\log} s_\pm)^{n+1}=0.\]
        Moreover, we have $\delta s_\pm = s_\pm\delta_{\log}s_\pm$, so we may regard $\delta_{\log}s_\pm$ as $s_\pm^{-1}\delta s_\pm$.
    \end{enumerate}
\end{lemma}
\begin{proof}
    The first statement is clear since the log structure is trivial away from the node $Q$. For the second statement, choose formal coordinates $s_\pm$ at $Q$, which define a trivialization
    \[D_Q = \Spf\kk[\![s_+,s_-]\!]/(s_+s_-) = \varinjlim_{p\ge 0} \Spec\left(\kk[s_+,s_-]/(s_+s_-,s_+^{p+1},s_-^{p+1})\right).\]
    For convenience, let $A_p = \kk[s_+,s_-]/(s_+s_-,s_+^{p+1},s_-^{p+1})$. The $\kk$-algebra $A_p$ has a natural log structure given by a morphism of monoids
    \[\N\longrightarrow (A_p,\cdot),\qquad \ell\longmapsto (s_+s_-)^\ell.\]
    
    By definition, $\mathcal{P}^n_{\Spec(A_p)/\kk,\log}(\Spec(A_p))$ is the $A_p$-algebra generated by the symbols $\delta(f)$ and $\delta_{\log}(m)$ subject to the relations (a)-(f) given above. We have
    \[\mathcal{P}^n_{C/\kk,\log}(D_Q) \cong \varprojlim_{p\ge 0} \mathcal{P}^n_{\Spec(A_p)/\kk,\log},\]
    which shows (b).
\end{proof}
Based on the above identities for $\delta_{\log}s_\pm$, we can see that 
\begin{equation}\label{eq:delta log gamma}
    \delta_{\log}s_- \equiv \gamma(\delta_{\log}s_+) \mod (\delta_{\log} s_+)^{n+1}
\end{equation} 
where again $\gamma(z) = -z/(1+z) = \sum_{k\ge 1}(-z)^k$. That is, the generator $\delta_{\log}s_-$ is redundant. The above results may be easily generalized when $C$ has more than one node. This is the origin of the series $\gamma$ used earlier on in the section.
\begin{proposition}\label{PD-connection nodal curve}
    Denote the normalization of $C$ by $\eta\colon \widetilde C\to C$. The PD-connection on $\mathcal{V}_{\widetilde C}$ induces a logarithmic PD-connection $\Theta_n\colon \mathcal{V}_C\to \mathcal{V}_C \otimes \mathcal{P}^n_{C/\kk,\log}$ on $C$ via the functor $\eta_*$.
\end{proposition}
\begin{proof}
    Denote the PD-connection on $\mathcal{V}_{\widetilde C}$ as $\{\widetilde\Theta_n\}_{n\ge 0}$. If we restrict these maps to the subsheaf $\widetilde {\mathcal{V}}_{\widetilde{C}} \subset \mathcal{V}_{\widetilde C}$, then we obtain a map
    \[\widetilde\Theta_n\colon \widetilde {\mathcal{V}}_{\widetilde{C}} \to \widetilde {\mathcal{V}}_{\widetilde{C}} \otimes \mathcal{P}^n_{\widetilde C/\kk,\log}.\]
    This just comes from the fact for an open $U\subset \widetilde C$ with coordinate $t$, the endomorphism $T^{(j)}\otimes \partial_t^{(k-j)}$ for $0\le j\le k$ maps
    \begin{align*}
        T^{(j)}\otimes \partial_t^{(k-j)}\colon &V_{s}\otimes \mathcal{O}_{U}(-sQ_+ - sQ_-)\\
        &\to V_{s+j}\otimes \mathcal{O}_{U}(-(s-k+j)Q_+ - (s-k+j)Q_-)\\
        &= \left(V_{s+j}\otimes \mathcal{O}_{U}(-(s+j)Q_+ - (s+j)Q_-)\right) \otimes \mathcal{O}_{U}(kQ_++kQ_-).
    \end{align*}
    Applying the functor $\eta_*$, we obtain maps
    \[\Theta_n\colon \eta_*\widetilde {\mathcal{V}}_{\widetilde{C}} \to \eta_*\widetilde {\mathcal{V}}_{\widetilde{C}} \otimes \eta_*\mathcal{P}^n_{\widetilde C/\kk,\log}.\]
    It remains to show that if we restrict $\Theta_n$ to $\mathcal{V}_{C}\subset \eta_*\widetilde{\mathcal{V}}_{\widetilde{C}}$, then we obtain a log PD-connection on $\mathcal{V}_{C}$.

    For simplicity, assume that $C$ has one node $Q$. Since we have already described the log PD-connection away from the node, it suffices to consider the maps locally at the node $Q$.
    Recall the identity in \zcref[noname]{eq:delta log gamma} relating $\delta_{\log}s_+$ and $\delta_{\log}s_-$. The induced PD-connection is given on $(a\otimes f,b\otimes g)\in \mathcal{V}_C(D_Q)$ by
    \begin{align*}
        &\Theta_n(a\otimes f,b\otimes g) \\
        &\equiv (e^{(\delta s_+)T}(a)e^{(\delta s_+)\partial_{s_+}}(f), e^{(\delta s_-)T}(b)e^{(\delta s_-)\partial_{s_-}}(g))\\
        &\equiv (e^{(\delta_{\log} s_+)s_+T}(a)e^{(\delta_{\log} s_+)s_+\partial_{s_+}}(f), e^{(\delta_{\log} s_-)s_-T}(b)e^{(\delta_{\log} s_-)s_-\partial_{s_-}}(g)) \\
        &\equiv (e^{(\delta_{\log} s_+)s_+T}(a)e^{(\delta_{\log} s_+)s_+\partial_{s_+}}(f), e^{\gamma(\delta_{\log} s_+)s_-T}(b)e^{\gamma(\delta_{\log} s_+)s_-\partial_{s_-}}(g)) \mod (\delta_{\log}s_+)^{n+1}.
    \end{align*}
    Given a formal power series $f(z) = \sum_{j\ge 0}a_jz^j$, we have
    \[f(\gamma(w)) = f(-w/(1+w)) = a_0 + \sum_{j\ge 1}w^j(-1)^j\sum_{1\le i\le j}\binom{j-1}{i-1}a_i.\]
    Recalling that $(s_+)^j\partial_{s_+}^{(j)} = \binom{s_+\partial_{s_+}}{j}$ in characteristic 0, we may write as shorthand
    \begin{align*}
        &\Theta_n(a\otimes f,b\otimes g) \\
        &\equiv \left(e^{(\delta_{\log}s_+)s_+T}(a)(1+\delta_{\log}s_+)^{s_+\partial_{s_+}}(f), e^{\gamma(\delta_{\log}s_+)s_-T}(b)(1+\delta_{\log}s_+)^{-s_-\partial_{s_-}}(g)\right),
    \end{align*}
    where we define the formal power series
    \begin{align*}
        (1+\delta_{\log}s_+)^{s_+\partial_{s_+}} &= \sum_{k\ge 0}(\delta_{\log}s_+)^k\binom{s_+\partial_{s_+}}{k}\\
        &= \sum_{k\ge 0}(\delta_{\log}s_+)^k(s_+)^k\partial_{s_+}^{(k)},\\
        (1+\delta_{\log}s_+)^{-s_-\partial_{s_-}} &= (1+\gamma(\delta_{\log}s_+))^{s_-\partial_{s_-}}.
    \end{align*}

    It remains to show that the logarithmic PD-connection is equivariant with respect to $\Aut\mathcal{O}$. Similar to before, we define $\nabla_{+,\log}^{(k)}(a\otimes f)$ (resp. $\nabla_{-,\log}^{(k)}(b\otimes g)$) to be the coefficient of $(\delta_{\log}s_+)^k$ in the first (resp. second) component of $\Theta_n(a\otimes f,b\otimes g)$ for $n\ge k$. We set $\nabla^{(k)}_{\log} = \left(\nabla^{(k)}_{+,\log},\nabla^{(k)}_{-,\log}\right)$, and we write
    \begin{align*}
        (1+w)^{\nabla_{+,\log}} &= e^{wsT}(1+w)^{s\partial_s}\\
        (1+w)^{\nabla_{-,\log}} &= e^{-w(1+w)^{-1}sT}(1+w)^{-s\partial_s}
    \end{align*}
    for $(1+w)^{\nabla_{\pm,\log}} = \sum_{k\ge 0}w^k\nabla_{\pm,\log}^{(k)}$.
    Recall one of the identities given in \zcref{sec:H-mod}:
    \[e^{xL_1}e^{zL_{-1}} = e^{(1-xz)^{-1}zL_{-1}}(1-xz)^{-2L_0}e^{(1-xz)^{-1}xL_1}.\]
    Replacing $x$ with $\lambda$ and using the fact that $L_1^{(i)}$ is degree $-i$ for $i\ge 0$, we have
    \[e^{\lambda L_1}e^{zT} = e^{z(1-\lambda z)^{-1}L_{-1}}e^{\lambda(1-\lambda z)L_1}(1-\lambda z)^{-2L_0}.\]
    Then we have as a relationship between formal series
    \begin{align*}
        &(e^{s^{-1}L_1}(-1)^{L_0})(1+w)^{\nabla_{+,\log}}[s^{\deg(a)}a]\\
        &=(-1)^{\deg(a)}(1+w)^{\deg(a)}s^{\deg(a)}e^{s^{-1}L_1}e^{-wsT}(a)\\
        &=(1+w)^{-\deg(a)}(-1)^{\deg(a)}s^{\deg(a)}e^{-w(1+w)^{-1}sT}e^{(1+w)s^{-1}L_1}(a)\\
        &=e^{-w(1+w)^{-1}sT}(1+w)^{-s\partial_{s}}(-1)^{\deg(a)}s^{\deg(a)}e^{s^{-1}L_1}(a)\\
        &=(1+w)^{\nabla_{-,\log}}(-1)^{\deg(a)}s^{\deg(a)}e^{s^{-1}L_1}(a)\\
        &=(1+w)^{\nabla_{-,\log}}(e^{s^{-1}L_1}(-1)^{L_0})[s^{\deg(a)}a].
    \end{align*}
    This shows that the logarithmic PD-connection is well-defined since the maps $\nabla_{\log}^{(k)} = \left(\nabla_{+,\log}^{(k)} , \nabla_{-,\log}^{(k)}\right)$ are well-defined endomorphisms of $\mathcal{V}_C(D_Q)$ for all $k\ge 0$. It remains to check the properties of these maps. The fact that the maps $\{\Theta_n\}_{n\ge 0}$ on $\mathcal{V}_C$ are right $\mathcal{O}_C$-linear and compatible with restriction and comultiplication follows from the corresponding properties of the PD-connection on $\mathcal{V}_{\widetilde C}$.
\end{proof}
\begin{remark}
    Given a smooth projective curve $C$, checking that the canonical PD-connection on $\mathcal{V}_C$ was compatible with comultiplication amounted to proving that 
    \[e^{z\nabla}e^{w\nabla} = e^{(z+w)\nabla},\] 
    where again $\nabla^{(n)} = \sum_{0\le k\le n}T^{(k)}\otimes \partial_t^{(n-k)}$. 
    Note that $F(z,w) = z+w$ is the formal group law for addition.
    
    Similarly, showing that the logarithmic PD-connection on $\mathcal{V}_C$ is compatible with comultiplication on $D_Q$ amounts to the following calculation:
    \begin{align*}
        (1+z)^{\nabla_{+,\log}}(1+w)^{\nabla_{+,\log}} 
        &= e^{zsT}(1+z)^{s\partial_{s}}e^{wsT}(1+w)^{s\partial_{s}}\\
        &= e^{zsT}e^{w(1+z)sT}(1+z)^{s\partial_{s}}(1+w)^{s\partial_{s}}\\
        &= e^{zsT}e^{w(1+z)sT}(1+z+w+zw)^{s\partial_{s}}\\
        &= e^{(z+w+zw)sT}(1+z+w+zw)^{s\partial_{s}}\\
        &= (1+z+w+zw)^{\nabla_{+,\log}}.
    \end{align*}
    Note that $F(z,w) = z+w+zw$ is the formal group law for multiplication.
\end{remark}

\begin{remark}\label{remark characteristic 0 comparison PD-conn}
    If $\kk$ has characteristic 0, then the log PD-connection on the vertex algebra sheaf $\mathcal{V}_C$ associated to a nodal curve $C$ is equivalent to the log connection given in \cite[\textsection 2.7]{DGT23}. This is because for all $k\ge 0$ we have the following identities:
    \[\nabla^{(k)} = \frac{1}{k!}\nabla^k \eqdef \frac{1}{k!}(T\otimes \id + \id \otimes \partial_t)^k,\qquad \nabla_{+,\log}^{(k)} \eqdef \binom{\nabla_{+,\log}}{k} = \binom{T\otimes s_+ + \id \otimes s_+\partial_{s_+}}{k}.\]
    This is an instance of the general principle that when looking at descent problems, one only needs to consider order 1 data when working in characteristic 0.
\end{remark}

\subsubsection{On formal smoothings and families of curves}\label{PD-connection formal smoothings}
Similar to the case of the sheaf of vertex algebras on a nodal curve, we may define a sheaf of logarithmic principal parts $\mathcal{P}^n_{\mathscr{C}/S,\log}$ associated to a smoothing family $\mathscr{C}\to S$.
Moreover, for all $n\ge 0$ there is a logarithmic PD-connection on $\mathcal{V}_{\mathscr{C}}$ given by a collection of maps
\[\Theta_n\colon \mathcal{V}_{\mathscr{C}}\to \mathcal{V}_{\mathscr{C}}\otimes \mathcal{P}^n_{\mathscr{C}/S,\log}.\]
It suffices to describe these maps on the formal neighborhood $\frak{D}_Q$. We define $\Theta_n$ on a pair $(a\otimes f,b\otimes g)$ for $a,b\in V$ and $f\in R(\!(s_+)\!)[\![q]\!]$ and $g\in R(\!(s_-)\!)[\![q]\!]$ by expanding linearly (and continuously) over $q$ and applying the definition of $\Theta_n$ as in \zcref{PD-connection nodal curve}.

\begin{remark}\label{rmk:vertex algebra sheaf PD-conn}
    This is a continuation of the discussion in \zcref{rmk:vertex algebra sheaf general case}. Given that we have described the log PD-connection on the vertex algebra sheaf for formal smoothings, we can glue using \cite{BL} again to construct the log PD-connection for the vertex algebra sheaf $\mathcal{V}_{\mathcal{C}}$ associated to any relative semistable curve $f\colon\mathcal{C}\to S$. We will describe the log PD-connection explicitly here. Given any open subset $U\subset \mathcal{C}$ in the smooth locus of $f$ admitting a local coordinate $t$ (say from an \'etale map $U\to \A^1_S$), we obtain a trivialization $\mathcal{V}_{\mathcal{C}}|_U \simeq_t V\otimes \mathcal{O}_U$. We may describe the PD-connection on $\mathcal{V}_{\mathcal{C}}|_U$ using the same description as in \zcref{PD-connection smooth curve}. In order to describe the log PD-connection on the whole curve $\mathcal{C}$, we just have to describe it on the formal neighborhood of a node $Q$ of the closed fiber of $f$. Up to a choice of formal coordinates $s_\pm$, the formal neighborhood is given by
    \[D_Q = \Spf \mathcal{O}_S[\![s_+,s_-]\!]/(s_-s_+).\]
    Once again, we may describe the log PD-connection using \zcref{PD-connection nodal curve}.
\end{remark}

\section{The chiral Lie algebra}\label{sec:chiralLieAlg}
Given a stable relative curve $\mathcal{C}\to S$ with $n$ marked points $P_1,\ldots,P_n\colon S\to \mathcal{C}$, there is a natural sheaf of Lie algebras $\mathcal{L}_{\mathcal{C}\setminus P_\bullet}(V)$ on $S$ associated to the sheaf of vertex algebras $\mathcal{V}_{\mathcal{C}/S}$. We define this sheaf in general and study it in the case of a nodal curve or a smoothing family. Our definition will differ from that of \cite{DGK23} in the sense that we are replacing their flat logarithmic connection with our logarithmic PD-connection, but otherwise all arguments carry over with few modifications.

\subsection{The geometric picture for the ancillary Lie algebra}\label{ancillary Lie algebra geom}
Given a nodal curve $C$ and an $\N$-graded vertex $\kk$-algebra $V$ with an $\Aut\mathcal{O}$-action, we showed that the sheaf $\mathcal{V}_C$ is naturally endowed with a logarithmic PD-connection. 
The formula for $\Theta_n$ in \zcref{PD-connection nodal curve} may also be applied to sections of $\mathcal{V}_C\otimes \omega_C$, naturally endowing it with the structure of a logarithmic PD-connection.
Now we apply the above observations to define the chiral Lie algebra. 
\begin{definition}
    Let $X\to S$ be a morphism of pre-log schemes, and let $\mathcal{E}$ be an $\mathcal{O}_X$-module with a log PD-connection $\{\Theta_n\}_{n\ge 1}$. We define $\Im \nabla$ to be the minimal subsheaf such that $\Theta_n$ factors as
    \[\Theta_n\colon \mathcal{E}\to \Im \nabla\otimes \mathcal{P}^n_{X/S,\log}\to \mathcal{E}\otimes \mathcal{P}^n_{X/S,\log}.\]
\end{definition}
The condition that $\Theta_n$ factors through an induced morphism by an inclusion of sheaves $\mathcal{F}\subset \mathcal{E}$ is linear, so the minimal object indeed exists. Note that $\mathcal{E}/\Im \nabla$ has an induced log PD-connection, which is trivial.\footnote{Here, the \textit{trivial (log) PD-connection} on an $\mathcal{O}_X$-module $\mathcal{E}$ is given by the composition of morphisms $\mathcal{E}\stackrel{\sim}{\to}\mathcal{E}\otimes\mathcal{O}_X \stackrel{d^n_0}{\to}\mathcal{E}\otimes \mathcal{P}^n_{X/S,\log}$ for $n\ge 0$.}

The primary example we will consider is $\mathcal{E} = \mathcal{V}_{C} \otimes \omega_{C}$ for a nodal curve $C$. We may realize $\Im \nabla$ as the subsheaf of $\mathcal{V}_C\otimes\omega_C$ such that the sections on the formal neighborhood of a smooth point $P$ are given by the colimit $(\Im\nabla^{(k)})_{k>0}$, and the sections on the formal neighborhood of a node $Q$ are given by $(\Im\nabla^{(k)}_{\log})_{k>0}$.

Using \zcref{rmk:vertex algebra sheaf general case} and \zcref{rmk:vertex algebra sheaf PD-conn}, the vertex algebra sheaf $\mathcal{V}_{\mathcal{C}}$ associated to a relative nodal curve $\mathcal{C}\to S$ is naturally equipped with a logarithmic PD-connection, as is the case with the relative dualizing sheaf $\omega_{\mathcal{C}/S}$. Similar to the above discussions, $\mathcal{V}_{\mathcal{C}}\otimes \omega_{\mathcal{C}/S}$ is equipped with a log PD-connection.

\begin{remark}
    In characteristic 0 the data of the log PD-connection on $\mathcal{V}_C\otimes\omega_C$ we have defined is equivalent to a flat log connection $\nabla\colon \mathcal{V}_C\to \mathcal{V}_C \otimes \omega_C$, which was described in \cite[\textsection 2.7]{DGT23} (cf. \zcref{remark characteristic 0 comparison PD-conn}). Therefore, the subsheaf $\Im \nabla$ is precisely the image of the morphism $\nabla$. This is where the notation for $\Im\nabla$ comes from, and it is also the primary reason we are working with the quotient $(\mathcal{V}_C\otimes\omega_C)/\Im \nabla$ instead of $\mathcal{V}_C/\Im \nabla$, despite $\mathcal{V}_C\otimes\omega_C$ being naturally isomorphic to $\mathcal{V}_C$. The other reason we are working with $(\mathcal{V}_C\otimes\omega_C)/\Im \nabla$ instead of $\mathcal{V}_C/\Im \nabla$ is for degree reasons.
\end{remark}

\begin{lemma}\label{lem:formal neighborhood ancillary Lie alg}
    At a smooth point $P\in C$, we have a natural identification
    \[H^0\left(D_P^\times,\frac{\mathcal{V}_C\otimes\omega_C}{\Im\nabla}\right) \simeq_t\mathfrak{L}_t(V) = \frac{V\otimes_\kk \kk(\!(t)\!)}{(\Im \nabla^{(k)})_{k>0}}\]
    up to a choice of coordinate $t$ on $D_P^\times = \Spf \kk(\!(t)\!)$.
\end{lemma}
\begin{proof}
    Given a formal coordinate $t$ at $P$, the punctured formal neighborhood $D_P^\times$ admits a trivialization $D_P^\times \simeq_t \Spf \kk(\!(t)\!)$. We define a map
    \[H^0\left(D_P^\times,\frac{\mathcal{V}_C\otimes\omega_C}{\Im\nabla}\right) \to \mathfrak{L}_t(V)\]
    as follows.
    Given a section of $\mathcal{V}_C\otimes\omega_C$ on $D_P^\times$ with respect to the $t$-trivialization, written as
    \[a\otimes \sum_{j\ge j_0}f_jt^j\,dt \in V\otimes_{\kk}\kk(\!(t)\!)\otimes_{\kk(\!(t)\!)}\kk(\!(t)\!)\,dt\simeq_t H^0(D_P^\times,\mathcal{V}_C\otimes \omega_C),\]
    we map it to the residue
    \[\Res_{t=0}Y(a,t)\sum_{j\ge j_0}f_jt^j\,dt = \sum_{j\ge j_0}f_ja_{[j]}\in \mathfrak{L}_t(V).\]
    The subsheaf $\Im \nabla$ is generated by elements of the form
    \[\nabla^{(k)}\left(a\otimes \sum_{j\ge j_0}f_jt^j\,dt\right) = \sum_{i=0}^k\left(T^{(i)}(a)\otimes \sum_{j\ge j_0}f_j\partial_t^{(k-i)}[t^j]\,dt\right)\]
    for all $a\in V$, $f\in \kk(\!(t)\!)$, and $k > 0$. These elements map to the residues
    \begin{align*}
        \sum_{i=0}^k\Res_{t=0}\left(Y(T^{(i)}a,t)\sum_{j\ge j_0}f_j\partial_t^{(k-i)}[t^j]\,dt\right) &=\sum_{i=0}^k\Res_{t=0}\left(\partial_t^{(i)}Y(a,t)\sum_{j\ge j_0}f_j\partial_t^{(k-i)}[t^j]\,dt\right)\\
        &= \Res_{t=0}\left(\partial_t^{(k)}\left(Y(a,t)\sum_{j\ge j_0}f_jt^j\right)\,dt\right)\\
        &=0.
    \end{align*}
Therefore, sections in $\Im\nabla^{(k)}\subset \mathcal{V}_C\otimes \omega_C$ for any $k>0$ map to 0, and this defines an isomorphism from sections of $(\mathcal{V}_C\otimes\omega_C)/\Im\nabla$ on $D_P^\times$ to $\mathfrak{L}_t(V)$.
\end{proof}

\subsubsection{Chiral algebras}
In characteristic 0, it is known that for a smooth curve $C$ the vertex algebra sheaf $\mathcal{V}_C$ has the structure of a chiral algebra. We expect that the same is true in positive characteristic provided that we work with PD-connections or regular differential operators instead of just ordinary connections.

In characteristic 0, the vector space $H^0(D_P^\times,(\mathcal{V}_C\otimes \omega_C)/\Im\nabla)$ has the structure of a (topological) Lie algebra arising from the chiral algebra structure on $\mathcal{V}_C$ such that the isomorphism in \zcref{lem:formal neighborhood ancillary Lie alg} extends to an isomorphism of Lie algebras \cite[\textsection 19.4.14, \textsection 6.6.9]{FBZ}. Chiral algebras have not been developed in positive characteristic, but one could certainly do so and replicate the proof.

A rough sketch of the proof is as follows:
If we replace every instance of $\frac{1}{n!}\partial_z^{n}$ with $\partial_z^{(n)}$ in \cite[\textsection 19, \textsection 6.6.9]{FBZ}, we will find that all of their arguments extend naturally with the convention that a $\cal D$-module is a sheaf equipped with the sheaf of regular differential operators, not just a sheaf with a flat connection (these are not the same in characteristic $p>0$). Thus, the desired claim follows from replicating their arguments. Alternatively, one could define the Lie algebra structure locally using the bracket on $\mathfrak{L}^{\mathsf{L}}(V)$. Either way, for a smooth curve $C$ the vertex algebra sheaf $\mathcal{V}_C$ is a chiral algebra in an appropriate sense such that the sheaf $(\mathcal{V}_C\otimes \omega_C)/\Im \nabla$ is naturally a sheaf of Lie algebras over $\kk$.
\begin{warning}
    Technically, we have not proven that the chiral Lie algebra $\mathcal{L}_{C\setminus P_\bullet}(V)$ is naturally a sheaf of Lie algebras. However, we believe that this is the case, so we will refer to it as such anyways. Either way, this convention is not used for any of our results.
\end{warning}

As mentioned before, it would be interesting to consider weaker notions of $\cal D$-modules in positive characteristic, such as connections with nilpotent or vanishing $p$-curvature, and study the corresponding notion of a chiral algebra.

\begin{remark}
    The arguments of \cite{HuangD-mod} could be used here to conclude that a vertex algebra over a field $\kk$ of any characteristic is equivalent to a weakly equivariant, unital chiral algebra on $\A^1_{\kk}$. Once again, we have to be careful with what we mean by a chiral algebra since we have to use the right version of a $\mathcal{D}$-module structure. We omit the details since it involves some tedious calculations.
\end{remark}

\subsection{The main definition}
Given the above observations, we may define the chiral Lie algebra for stable curves:
\begin{definition}
    Given a stable curve $f\colon\mathcal{C}\to S$ with marked sections $P_1,\ldots,P_n\colon S\to \mathcal{C}$, denote by $f_{\mathcal{C}\setminus P_\bullet}\colon \mathcal{C}\setminus P_\bullet(S) \to S$ the induced map from $f$. We define the $\mathcal{O}_S$-module $\mathcal{L}_{C\setminus P_\bullet}(V)$ to be 
    \[\mathcal{L}_{\mathcal{C}\setminus P_\bullet}(V) = (f_{\mathcal{C}\setminus P_\bullet})_*\left(\left((\mathcal{V}_{\mathcal{C}/S}\otimes \omega_{\mathcal{C}/S})/\Im \nabla\right)\big|_{\mathcal{C}\setminus P_\bullet(S)}\right).\]
\end{definition}
Given a stable, $n$-pointed curve $(C,P_\bullet)$ and an $\N$-graded vertex algebra $V$ with an $\Aut\mathcal{O}$-action, the chiral Lie algebra is given by
\[\mathcal{L}_{C\setminus P_\bullet}(V) \eqdef H^0\left(C\setminus P_\bullet,\frac{\mathcal{V}_C\otimes\omega_C}{\Im\nabla}\right).\]
From the above discussions, the chiral Lie algebra naturally has the structure of a Lie algebra.
\begin{notation}
    Occasionally, we will refer to the subsheaf $\Im \nabla \subset \mathcal{V}_C\otimes \omega_C$ as $\Im \nabla_C$ instead. This is so that it is clear over what curve $\Im \nabla$ is a sheaf.
\end{notation}

\subsection{Properties of the chiral Lie algebra}
Let $f\colon \mathcal{C}\to S$ be a stable, $n$-pointed curve with marked sections $P_1,\ldots,P_n\colon S\to \mathcal{C}$. In this relative setting, the formal neighborhood $D_P$ of a section $P\colon S\to \mathcal{C}$ is noncanonically given by $\Spf\mathcal{O}_S[\![t]\!]$, where $t$ is a chosen formal coordinate on $P(S)$. The punctured formal neighborhood is similarly $D_P^\times \simeq \Spf \mathcal{O}_S(\!(t)\!)$. Denote by $f_{D_P^\times}\colon D_P^\times \to S$ for the restriction of $f$ to $D_P^\times$. Using a similar argument to \zcref{lem:formal neighborhood ancillary Lie alg}, we have an identification
\[(f_{D_P^\times})_*\left(\left(\left(\mathcal{V}_{\mathcal{C}/S} \otimes \omega_{\mathcal{C}/S}\right)/\Im \nabla\right)\big|_{D_P^\times}\right) \simeq_t \frac{V \otimes_\kk \mathcal{O}_S(\!(t)\!)}{(\Im \nabla^{(k)})_{k>0}} = \mathfrak{L}_t(V)\:\widehat{\otimes}_\kk\: \mathcal{O}_S.\]
As shorthand, we write
\begin{equation}
\mathfrak{L}(V;\mathcal{O}_S) = \mathfrak{L}_t(V)\:\widehat{\otimes}_\kk\: \mathcal{O}_S.
\end{equation}
By restriction to the punctured formal neighborhoods of the marked sections, we have a natural morphism of Lie algebras
\begin{equation}\label{eq: varphiL}
    \varphi\colon \mathcal{L}_{\mathcal{C}\setminus P_\bullet}(V) \longrightarrow \bigoplus_{i=1}^n (f_{D_{P_i}^\times})_*\left(\left(\left(\mathcal{V}_{\mathcal{C}/S} \otimes \omega_{\mathcal{C}/S}\right)/\Im \nabla\right)\big|_{D_{P_i}^\times}\right) \stackrel{\simeq}{\longrightarrow}\bigoplus_{i=1}^n \mathfrak{L}_{t_i}(V;\mathcal{O}_S).
\end{equation}
We denote by $\sigma_{P_i}$ the image of a section $\sigma\in \mathcal{L}_{\mathcal{C}\setminus P_\bullet}(V)$ under the map $\mathcal{L}_{\mathcal{C}\setminus P_\bullet}(V)\to \mathfrak{L}_{t_i}(V;\mathcal{O}_S)$ for all $i=1,\ldots,n$.
With this, we obtain an action of $\mathcal{L}_{\mathcal{C}\setminus P_\bullet}(V)$ on $\mathfrak{L}(V;\mathcal{O}_S)^{\oplus n}$-modules. If $C$ is nodal, then there is an analogous version of this statement in \zcref{lem:lie algebra action}.

Let $(C,P_\bullet)$ be a stable, $n$-pointed curve such that $C\setminus P_\bullet$ is affine. Assume, for simplicity, that $C$ has exactly one simple node, which we denote by $Q$. Let $\eta\colon \widetilde C\to C$ be the normalization of $C$, with preimages $Q_+$ and $Q_-$, written as $Q_\pm=(Q_+,Q_-)$. Let $s_\pm$ be formal coordinates at $Q_\pm$ such that locally around $Q$, the curve is given by the equation $s_+s_-=0$. 
The chiral Lie algebra for $(\widetilde C,P_\bullet\sqcup Q_\pm)$ is
\[\mathcal{L}_{\widetilde C\setminus (P_\bullet\sqcup Q_\pm)}(V) = H^0\left(\widetilde C\setminus (P_\bullet\sqcup Q_\pm),\frac{\mathcal{V}_{\widetilde C}\otimes \omega_{\widetilde C}}{\Im\nabla_{\widetilde C}}\right).\]
Consider the linear maps given by restriction:
\begin{equation}
    \mathcal{L}_{\widetilde C\setminus (P_\bullet\sqcup Q_\pm)}(V) \longrightarrow H^0\left(D^\times_{Q_\pm},\frac{\mathcal{V}_{\widetilde C}\otimes \omega_{\widetilde C}}{\Im\nabla_{\widetilde C}}\right)\stackrel{\sim_{s_\pm}}{\longrightarrow}\mathfrak{L}_{Q_\pm}(V).
\end{equation}
Denote by $\sigma_{Q_\pm}$ the image of $\sigma\in \mathcal{L}_{\widetilde C\setminus (P_\bullet\sqcup Q_\pm)}(V)$ in $\mathfrak{L}_{Q_\pm}(V)$, and denote by $[\sigma_{Q_\pm}]_0$ the degree 0 part of $\sigma_{Q_\pm}$ with respect to the canonical split-filtration on $\mathfrak{L}(V)^{\mathsf{L}}\simeq \mathfrak{L}_{Q_\pm}(V)$.

Recall the involution $\theta$ of $\LVf$, which is given for a homogeneous element $a$ by
\begin{equation}\label{eq:iota}
    \theta(a_{[\deg(a)-1]}) = (-1)^{\deg(a)}\sum_{i\ge 0}(L_1^{(i)}a)_{[\deg(a)-i-1]}.
\end{equation}
For the following result, we fix a stable, $n$-pointed curve $C$ over $\kk$ with marked points $P_1,\ldots,P_n$.
\begin{proposition}\label{prop:chiral Lie alg nodal}
    Assume that $C\setminus P_\bullet$ is affine. Then we have an identification
    \[\eta^*(\mathcal{L}_{C\setminus P_\bullet}(V)) = \left\{\sigma\in \mathcal{L}_{\widetilde C\setminus (P_\bullet\sqcup Q_\pm)}(V) \,\big|\, \sigma_{Q_\pm}\in \LVf_{\le 0},\ [\sigma_{Q_-}]_0 = -\theta ([\sigma_{Q_+}]_0)\right\}.\]
\end{proposition}
\begin{proof}
    The proof more or less carries out in the same way as in \cite[Proposition 3.3.1]{DGT23}, but we have to change some notation.

    Since $C\setminus P_\bullet$ is affine, we have
    \[\mathcal{L}_{C\setminus P_\bullet}(V) = \frac{H^0(C\setminus P_\bullet, \mathcal{V}_C\otimes \omega_C)}{H^0(C\setminus P_\bullet,\Im\nabla_C)}.\]
    Moreover, $\widetilde C\setminus (P_\bullet\sqcup Q_\pm)$ is also affine, so we have
    \[\mathcal{L}_{\widetilde C\setminus (P_\bullet\sqcup Q_\pm)}(V) = \frac{H^0\left(\widetilde C\setminus (P_\bullet\sqcup Q_\pm) , \mathcal{V}_{\widetilde C} \otimes \omega_{\widetilde C}\right)}{H^0(\widetilde C\setminus (P_\bullet\sqcup Q_\pm) , \Im\nabla_{\widetilde C})}.\]
    Therefore, elements in $\mathcal{L}_{C\setminus P_\bullet}(V)$ are represented by elements in $H^0(C\setminus P_\bullet, \mathcal{V}_C\otimes \omega_C)$.

    We have inclusions of sheaves $\eta^*\omega_C\subset \omega_{\widetilde{C}}(Q_++Q_-)$ and $\eta^*\mathcal{V}_C\subset \widetilde{\mathcal{V}}_{\widetilde{C}}$. We then have an inclusion
    \begin{equation}\label{pullback by eta H^0}
        \eta^*H^0(C\setminus P_\bullet, \mathcal{V}_C\otimes\omega_C) \subseteq H^0\left(\widetilde C\setminus P_\bullet,\widetilde{\mathcal{V}}_{\widetilde{C}}\otimes \omega_{\widetilde C}(Q_++Q_-)\right).
    \end{equation}
    Now, consider the restriction map induced by the inclusion $D_{Q_\pm} \hookrightarrow \widetilde C\setminus P_\bullet$. We may identify
    \[H^0\left(\widetilde C\setminus P_\bullet,\widetilde{\mathcal{V}}_{\widetilde{C}}\otimes \omega_{\widetilde C}(Q_++Q_-)\right) \simeq_{s_\pm} \bigoplus_{k\ge 0}V_k \otimes_\kk s^{k-1}_\pm \kk[\![s_\pm]\!]ds_\pm.\]
    It follows that the residue map
    \begin{equation}\label{res map le 0}
        \bigoplus_{k\ge 0}V_k\otimes s_\pm^{k-1}\kk[\![s_\pm]\!]ds_\pm \to \frakL_{Q_\pm}(V),\qquad a\otimes f(s_\pm)\,ds_\pm \mapsto \Res_{s_\pm = 0}\left(Y(a,s_\pm)\,ds_\pm\right)
    \end{equation}
    lies in $\frakL_{Q_\pm}(V)_{\le 0}$. Composing the inclusion in \zcref[noname]{pullback by eta H^0} with the residue map in \zcref[noname]{res map le 0}, we deduce that for $\sigma\in \eta^*H^0(C\setminus P_\bullet,\mathcal{V}_C\otimes \omega_C)$ the image $\sigma_{Q_\pm}\in\frakL_{Q_\pm}(V)$ lies in $\frakL_{Q_\pm}(V)_{\le 0}\cong \LVf_{\le 0}$.


    Now we will describe the gluing isomorphism for elements $\sigma$ of \zcref[noname]{pullback by eta H^0}. For homogeneous $a\in V$, there are gluing isomorphism on fibers of $\eta_*\widetilde{\mathcal{V}}_{\widetilde{C}}$ and $\eta_*\omega_{\widetilde{C}}$ at $Q_\pm$ defining $\mathcal{V}_C$ and $\omega_C$ respectively:
    \[a\otimes s_+^{\deg(a)} \equiv (-1)^{\deg(a)}\sum_{i\ge 0}L_1^{(i)}(a)\otimes s_-^{\deg(a)-i},\qquad s_+^{-1}ds_+ \equiv -s_-^{-1}ds_-.\]
    Combining these two identifications, the induced gluing isomorphisms on fibers of $\widetilde{\mathcal{V}}_{\widetilde{C}}\otimes \omega_{\widetilde C}(Q_++Q_-)$ at $Q_\pm$ is given by
    \begin{equation}\label{gluing V otimes omega}
        a\otimes s_+^{\deg(a)-1}ds_+ \equiv (-1)^{\deg(a)-1}\sum_{i\ge 0}L_1^{(i)}(a)\otimes s_-^{\deg(a)-i-1}ds_-.
    \end{equation}
    The above relation is equivalent to requiring that $[\sigma_{Q_-}]_0 = -\theta[\sigma_{Q_+}]_0$.
    
    To show that the two conditions are equivariant with respect to the PD-connection $\Theta$, we just have to check on a formal neighborhood of the node $Q$. Looking at the definition of $\Theta$ in \zcref{PD-connection nodal curve}, we have $\eta^*\Im \nabla_C|_{D_Q}\subseteq \eta^*\Im \nabla_{\widetilde C}|_{D_{Q_+}^\times\sqcup D_{Q_-}^\times}$, so we are done.
\end{proof}
This proposition is the reason we choose to work with $\mathcal{V}_C\otimes \omega_C$ in the definition of $\mathcal{L}_{C\setminus P_\bullet}(V)$. If we chose to replace $\mathcal{V}_C\otimes \omega_C$ with $\mathcal{V}_C$, then we would have a degree mismatch, and we would need to redefine the action of $\mathcal{L}_{C\setminus P_\bullet}(V)$ via the residue map. In principle, though, one may work with $\mathcal{V}_C$ instead of $\mathcal{V}_C\otimes \omega_C$.

\begin{corollary}\label{lem:lie algebra action}
    The normalization map $\eta\colon \widetilde C\to C$ identifies $\mathcal{L}_{C\setminus P_\bullet}(V)$ with a Lie subalgebra of $\mathcal{L}_{\widetilde C\setminus P_\bullet \sqcup Q_\bullet}(V)$ via the pullback $\eta^*$.
\end{corollary}
The proof is the same as in \cite[Proposition 3.3.2]{DGT23}, so we omit it.

\subsection{An approximation result for chiral Lie algebras}
Let $\mathcal{C}\to S$ be a stable, $n$-pointed curve with marked sections $P_1,\ldots,P_n\colon S\to \mathcal{C}$. We will need an important technical result about the sheaf of chiral Lie algebras for our main theorems.

We give the setting for our result here.
Let $\mathcal{C}\to S$ be a stable, $(n+m)$-pointed curve with marked sections $P_1,\ldots,P_n,Q_1,\ldots,Q_m\colon S\to \mathcal{C}$ such that the images of the sections are contained in the smooth locus and are pairwise disjoint. Assume that $\mathcal{C}\setminus P_\bullet(S)$ is affine over $S$. For each $i=1,\ldots,m$, let $s_i$ be a formal coordinate at $Q_i(S)$, which generates the maximal ideal sheaf $\widehat{\mathfrak{m}}_{\mathcal{C},Q_i}$ of the completion $\widehat{\mathcal{O}}_{\mathcal{C},Q_i} \simeq_{s_i} \mathcal{O}_S[\![s_i]\!]$.
\begin{proposition}\label{prop: exists section chiral Lie alg}
    In the above setting, for all $\ell\ge 0$ the natural map
    \[\left(\mathcal{V}_{\mathcal{C}/S} \otimes \omega_{\mathcal{C}/S}\right)\left(\mathcal{C}\setminus P_\bullet(S)\right) \to \bigoplus_i \left(\mathcal{V}_{\mathcal{C}/S} \otimes \omega_{\mathcal{C}/S}\right)\left(\Spec\widehat{\mathcal{O}}_{\mathcal{C},Q_i} / (\widehat{\mathfrak{m}}_{\mathcal{C},Q_i})^\ell\right)\]
    is surjective.
\end{proposition}
\begin{proof}
    First, we have
    \[\widehat{\mathcal{O}}_{\mathcal{C},Q_i} / (\widehat{\mathfrak{m}}_{\mathcal{C},Q_i})^\ell \cong \mathcal{O}_{\mathcal{C},Q_i} / (\mathfrak{m}_{\mathcal{C},Q_i})^\ell\]
    since the topology on $\mathcal{O}_{\mathcal{C},Q_i}$ is induced from $\mathfrak{m}_{\mathcal{C},Q_i}$.
    Note that the natural map
    \[\mathcal{O}_{\mathcal{C}}\left(\mathcal{C}\setminus P_\bullet(S)\right) \to \bigoplus_i \mathcal{O}_{\mathcal{C}}\left(\Spec\mathcal{O}_{\mathcal{C},Q_i} / (\mathfrak{m}_{\mathcal{C},Q_i})^\ell\right)\]
    is surjective by the Chinese remainder theorem. The $\mathcal{O}_{\mathcal{C}}$-module $\mathcal{V}_{\mathcal{C}/S} \otimes \omega_{\mathcal{C}/S}$ is flat since it is locally free, so tensoring with it is right exact. This shows the desired statement.
\end{proof}
Here is the primary consequence of the statement that we will make use of.
\begin{corollary}\label{cor:Riemann Roch calculation}
    Fix $j\in \{1,\ldots,m\}$ and integers $d<N$. Then for any $a\in V$ there exists $\sigma\in \mathcal{L}_{\mathcal{C}\setminus (P_\bullet\sqcup Q_\bullet)}(V)$ such that the expansion $\sigma_{Q_i} \in \mathfrak{L}_{s_i}(V;\mathcal{O}_S)^{\mathsf{L}}$ satisfies
    \[\sigma_{Q_i} \in \delta_{i,j}a_{[d]} + s_i^N(V\otimes \mathcal{O}_S)[\![s_i]\!]\]
    for all $i\in \{i,\ldots,m\}$.
\end{corollary}

Let $C$ be a smooth curve, possibly disconnected, and suppose there are two non-empty sets of distinct marked points $P_\bullet = (P_1,\ldots,P_n)$ and $Q_\bullet = (Q_1,\ldots,Q_m)$. For $i\in \{1,\ldots,m\}$, let $s_i$ be a formal coordinate at $Q_i$. For $\sigma\in \mathcal{L}_{C\setminus P_\bullet\sqcup Q_\bullet}(V)$, let $\sigma_{Q_i}$ be the image of $\sigma$ under the map given by restriction:
\[\mathcal{L}_{C\setminus P_\bullet\sqcup Q_\bullet}(V) \to H^0\left(D^\times_{Q_i},\frac{\mathcal{V}_C\otimes \omega_C}{\Im\nabla}\right) \stackrel{\sim_{s_i}}{\longrightarrow}\mathfrak{L}_{Q_i}(V).\]
For an integer $N$, consider
\[\mathfrak{L}(V,NQ_i) = \frac{V\otimes_{\kk} s_i^N\kk[\![s_i]\!]}{\Im\nabla}.\]
This is a Lie subalgebra of $\mathfrak{L}_{Q_i}(V)$.

The last general result we need is not actually about chiral Lie algebras, but it is fitting to put it here since it is related to the relative dualizing sheaf $\omega_{X/S}$.
\begin{lemma}[Strong residue theorem]\label{lem:strong residue}
    Let $f\colon X\to S$ be a proper flat family of connected nodal curves. Let $D = P_1+\cdots+P_n$ be a relative effective Cartier divisor such that the sections $P_1,\ldots,P_n\colon S\to X$ have mutually distinct images that all lie in the smooth locus of $f$. Then the image of the residue map
    \[\Res\colon f_*(\omega_{X/S}(D))\to \mathcal{O}_S^{\oplus n}\]
    is given by the kernel of the summation map
    \[\ker\left(\Sigma\colon \mathcal{O}_S^{\oplus n} \to \mathcal{O}_S\right).\]
\end{lemma}
\begin{proof}
    We have a short exact sequence
    \begin{equation}
            0 \to \omega_{X/S} \hookrightarrow \omega_{X/S}(D) \to \omega_{X/S}(D)/\omega_{X/S} \to 0.
    \end{equation}
    We have $\omega_{X/S}(D)/\omega_{X/S} = \bigoplus_{i=1}^n\mathcal{O}_{P_i}$. Applying the pushforward $f_*$, we have a long exact sequence
    \begin{equation}
        \begin{tikzcd}
            0 \to f_*\omega_{X/S} \to f_*\left(\omega_{X/S}(D)\right) \to f_*\left(\omega_{X/S}(D)/\omega_{X/S}\right) \stackrel{\delta}{\to} R^1f_*\omega_{X/S} \to \cdots.
        \end{tikzcd}
    \end{equation}
    By construction, we have $f_*\left(\omega_{X/S}(D)/\omega_{X/S}\right) \cong \mathcal{O}_S^{\oplus n}$. From \cite[0E6N, 0E5W]{stacks-project}, the relative dualizing sheaf $\omega_{X/S}$ admits an isomorphism 
    \[R^1f_*\omega_{X/S} \stackrel{\cong}{\longrightarrow} \mathcal{O}_S.\]
    Moreover, under these isomorphisms, the trace map $\delta$ is identified with the summation map $\Sigma\colon \mathcal{O}_S^{\oplus n}\to \mathcal{O}_S$.
\end{proof}


\subsection{On formal smoothings}
Let $\mathscr{C}\to S$ be a smoothing family as in \zcref{smoothings}. On the formal neighborhood $\frak{D}_{P_i} = \Spec R[\![t_i,q]\!]$, we have the trivialization $\mathcal{V}_{\mathscr{C}}|_{\frak{D}_{P_i}} \simeq_{t_i}(V\otimes R[\![t_i]\!])[\![q]\!]$. Once again, there is an induced log PD-connection on the tensor product $\mathcal{V}_{\mathscr{C}}\otimes \omega_{\mathscr{C}/S}$ given by the log PD-connection on $\mathcal{V}_{\mathscr{C}}$ as described in \zcref{PD-connection formal smoothings}. As such, we may define the chiral Lie algebra associated to $\mathcal{V}_{\mathscr{C}}$ to be
\[\mathcal{L}_{\mathscr{C}\setminus P_\bullet}(V) \eqdef H^0\left(\mathscr{C}\setminus P_\bullet(S),\frac{\mathcal{V}_{\mathscr{C}}\otimes \omega_{\mathscr{C}/S}}{\Im \nabla}\right),\]
where $\Im \nabla$ is defined to be the smallest subsheaf of $\mathcal{V}_{\mathscr{C}}\otimes \omega_{\mathscr{C}/S}$ such that the induced log PD-connection on $(\mathcal{V}_{\mathscr{C}}\otimes \omega_{\mathscr{C}/S})/\Im \nabla$ is trivial.

We have a map of Lie algebras
\[\mathcal{L}_{\mathscr{C}\setminus P_\bullet}(V) \to \bigoplus_{i=1}^n H^0\left(\frak{D}_{P_i}^\times,\frac{\mathcal{V}_{\mathscr{C}}\otimes \omega_{\mathscr{C}/S}}{\Im\nabla}\right) \stackrel{\cong}{\longrightarrow} \bigoplus_{i=1}^n\frakL_{t_i}(V)[\![q]\!]\]
induced by restriction of sections, similar to \zcref[noname]{eq: varphiL}.
\begin{remark}
    The Lie algebra $\mathcal{L}_{\mathscr{C}\setminus P_\bullet}(V)$ is equivalently the subspace of $\mathcal{L}_{\widetilde{\mathscr{C}}\setminus P_\bullet\sqcup Q_\bullet}(V)$ generated by elements $\sigma$ whose restrictions $\sigma_{Q_\pm}$ to $\frakL_{Q_\pm}(V)[\![q]\!]$ satisfy
    \begin{align*} 
        \sigma_{Q_+} &= \sum_{i,j\ge 0} \alpha_{i,j} \, a_{[\deg(a)+i-j-1]} \, q^j, \\
        \sigma_{Q_-} &= (-1)^{\deg(a)-1}\sum_{i,j\ge 0}  \alpha_{i,j} \sum_{k \ge 0} \left(L_1^{(k)}a \right)_{[\deg(a)-k+j-i-1]} q^i \\
        &= -\sum_{i,j\ge 0}  \alpha_{i,j} \, \theta\left(a_{[\deg(a)+i-j-1]}\right) q^i
    \end{align*} 
    for homogeneous $a \in V$ and integers $i, j \in \Z_{\ge 0}$. This is an extension of \zcref{prop:chiral Lie alg nodal}. This uses the facts that sections of $\mathcal{V}_{\mathscr{C}}$ over $\mathfrak{D}_Q$ are generated by elements of the form \zcref[noname]{eq:typeA} and \zcref[noname]{eq:typeB}; that sections of $\omega_{\mathscr{C}/S}$ over $\mathfrak{D}_Q$ are generated by $\left(\frac{ds_+}{s_+}, - \frac{ds_-}{s_-} \right)$ over $\widehat{\cal{O}}_Q$; and the definition of $\theta$ from \zcref[noname]{eq:iota}.
\end{remark}


\section{Sheaves of coinvariants}\label{sec:coinvariants}
Now that we have defined the chiral Lie algebra, we will study its associated coinvariants. These will be the main focus of study throughout the rest of our work.

\subsection{Representations of the chiral Lie algebra}\label{sec:rep of chiral}
We fix the following setting for this section.
Let $(\mathcal{C}\to S,P_\bullet,t_\bullet)$ be a stable, $n$-pointed relative curve over a smooth base $S$, with marked sections $P_i\colon S\to \mathcal{C}$ and formal coordinates $t_i$ at $P_i(S)$ for $i=1,\ldots,n$. 
Let $V$ be an $\N$-graded vertex $\kk$-algebra with an $\Aut\mathcal{O}$-action, and let $M^\bullet = (M^1,\ldots,M^n)$ be admissible $V$-modules. As an abuse of notation we also write $M^\bullet = M^1\otimes \cdots\otimes M^n$ when the context is appropriate. Here, the tensor product $M^\bullet\otimes \mathcal{O}_S$ is a module for $\mathcal{L}_{\mathcal{C}\setminus P_\bullet}(V)$ via the map $\varphi\colon \mathcal{L}_{\mathcal{C}\setminus P_\bullet}(V) \to \bigoplus_{i=1}^n\mathfrak{L}_{t_i}(V;\mathcal{O}_S)$ from \zcref[noname]{eq: varphiL}. Explicitly, for local sections $u = u^1\otimes \cdots\otimes u^n \in M^\bullet$ and $\sigma \in \mathcal{L}_{\mathcal{C}\setminus P_\bullet}(V)$ we set
\[\sigma(u) = \sum_{i=1}^n u^1\otimes\cdots\otimes \sigma_{P_i}(u^i)\otimes\cdots\otimes u^n.\]
Now we can define coinvariants of this action.
\begin{definition}
We define the \textit{sheaf of coinvariants} $\mathbb{V}(V;M^\bullet)_{(\mathcal{C}\to S,P_\bullet,t_\bullet)}$ to be the cokernel of the action map of $\mathcal{L}_{\mathcal{C}\setminus P_\bullet}(V)$ on $M^\bullet\otimes \mathcal{O}_S$:
\[\mathcal{L}_{\mathcal{C}\setminus P_\bullet}(V) \otimes (M^\bullet \otimes \mathcal{O}_S) \to M^\bullet \otimes \mathcal{O}_S \to \mathbb{V}(V;M^\bullet)_{(\mathcal{C}\to S,P_\bullet,t_\bullet)}\to 0.\]
\end{definition}
If $\mathcal{C}\setminus P_\bullet(S)$ is affine, then we have an obvious isomorphism
\[\mathbb{V}(V;M^\bullet)_{(\mathcal{C}\to S,P_\bullet,t_\bullet)} \cong (M^\bullet\otimes \mathcal{O}_S)_{\mathcal{L}_{\mathcal{C}\setminus P_\bullet}(V)} = \frac{M^\bullet\otimes \mathcal{O}_S}{\mathcal{L}_{\mathcal{C}\setminus P_\bullet}(V)\cdot \left(M^\bullet \otimes \mathcal{O}_S\right)}.\]
If $\mathcal{C}\setminus P_\bullet$ is not affine, then the sheaf of coinvariants is more complicated to describe.



\subsection{Propagation of vacua}
To describe the sheaves of coinvariants when $\mathcal{C}\setminus P_\bullet(S)$ is not affine, we will use \textit{propagation of vacua}. This allows us to describe the sheaf of coinvariants as the colimit over a specific family of sheaves. This is the first major result of our work.
\begin{theorem}[Propagation of vacua]\label{thm:propagation of vacua}
    Assume that $\mathcal{C}\setminus P_\bullet(S)$ is affine over $S$. The linear map
    \[\iota\colon M^\bullet\to M^\bullet \otimes_\kk V,\quad u\mapsto u\otimes \vac\]
    induces a canonical $\mathcal{O}_S$-module isomorphism
    \[\varphi\colon \mathbb{V}(V;M^\bullet)_{(\mathcal{C}/S,P_\bullet,t_\bullet)} \stackrel{\cong}{\to}\mathbb{V}(V;M^\bullet \otimes_\kk V)_{(\mathcal{C}/S,\, P_\bullet\sqcup Q,\, t_\bullet\sqcup r)}.\]
    The induced isomorphisms are compatible as we vary $(Q,r)$. Moreover, as $\vac$ is fixed by the action of $\Aut\mathcal{O}$, the isomorphism is equivariant with respect to a change of coordinates.
\end{theorem}
\begin{proof}
    First, we show that the induced map $\varphi$ is well-defined. Let $\sigma\in \mathcal{L}_{\mathcal{C}\setminus P_\bullet}(V)$. Since $\sigma$ is regular at $Q$, each term in its Laurent series expansion $\sigma_Q \in \mathfrak{L}(V)^{\mathsf{L}}$ is of the form $a_{[i]}$ for some $a\in V$ and $i\ge 0$. From the axioms on the vacuum vector $\vac$, we have $a_{(i)}\vac = 0$ for all $i\ge 0$, which implies that $\sigma_Q(\vac) = 0$. Therefore, for all $\sigma\in \mathcal{L}_{\mathcal{C}\setminus P_\bullet}(V)$, $u\in M^\bullet$, and $f\in \mathcal{O}_S$, we have
    \[\varphi(\sigma(u\otimes f)) = \varphi(\sigma(u)\otimes f) = \sigma(u)\otimes \vac\otimes f + u\otimes \sigma_Q(\vac)\otimes f = \sigma(u\otimes \vac \otimes f).\]
    We have $\mathcal{L}_{\mathcal{C}\setminus P_\bullet}(V)\subset \mathcal{L}_{\mathcal{C}\setminus (P_\bullet\sqcup Q)}(V)$, so $\varphi$ maps 0 to 0. This implies that the map $\varphi$ is well-defined.

    Now we show that $\varphi$ is surjective. Let $u\in M^\bullet$, $a\in V$, and $f\in \mathcal{O}_S$.
    By \zcref{cor:Riemann Roch calculation}, there exists a section $\sigma$ of $\mathcal{L}_{\mathcal{C}\setminus (P_\bullet\sqcup Q)}(V)$ such that 
    \[\sigma_Q \in a_{[-1]} + (V\otimes \mathcal{O}_S)[\![r]\!].\]
    Due to the axioms on $\vac$, we have
    \[\sigma_Q\vac = a,\]
    so we have
    \begin{align*}
        u\otimes a\otimes f &= u\otimes (a\otimes \mu_Q)(\vac)\otimes f\\
        &\equiv -\sigma(u)\otimes \vac\otimes f \bmod \mathcal{L}_{\mathcal{C}\setminus (P_\bullet\sqcup Q)}(V)\\
        &= -\varphi\left(\sigma(u)\otimes f\right),
    \end{align*}
    where $\sigma(u) = \sum_{i=1}^nu_1\otimes\cdots\otimes \sigma_{P_i}(u_i)\otimes\cdots\otimes u_n$.
    This shows that $\varphi$ is surjective.

    Now we show that $\varphi$ is injective. By definition, we have
    \[\ker\varphi = \frac{\iota^{-1}\left(\mathcal{L}_{\mathcal{C}\setminus (P_\bullet\sqcup Q)}(V)\cdot ((M^\bullet\otimes V)\otimes \mathcal{O}_S)\right)}{\mathcal{L}_{\mathcal{C}\setminus P_\bullet}(V)\cdot \left(M^\bullet \otimes \mathcal{O}_S\right)}\]
    The morphism $M^\bullet\to M^\bullet\otimes V$ factors through $M^\bullet\otimes \kk\vac \cong M^\bullet$, so we may identify 
    \[\mathcal{L}_{\mathcal{C}\setminus P_\bullet}(V) \cdot \left(M^\bullet\otimes \mathcal{O}_S\right) \cong \mathcal{L}_{\mathcal{C}\setminus P_\bullet}(V) \cdot \left(\left(M^\bullet\otimes \kk\vac\right)\otimes \mathcal{O}_S\right),\]
    where the action of $\sigma\in\mathcal{L}_{\mathcal{C}\setminus P_\bullet}(V)\subset \mathcal{L}_{\mathcal{C}\setminus (P_\bullet\sqcup Q)}(V)$ on $\left(M^\bullet\otimes \kk\vac\right)\otimes \mathcal{O}_S$ satisfies
    \[\sigma\cdot(u\otimes \vac\otimes f) = \sigma(u)\otimes \vac\otimes f\]
    for $u\in M^\bullet$ and $f\in \mathcal{O}_S$. Identifying $\iota$ with the inclusion map $M^\bullet\otimes \kk\vac \to M^\bullet\otimes V$, we have
    \begin{align*}
        &\iota^{-1}\left(\mathcal{L}_{\mathcal{C}\setminus(P_\bullet\sqcup Q)}(V)\cdot \left(\left(M^\bullet\otimes V\right) \otimes \mathcal{O}_S\right)\right)\\
        &= \left(\left(M^\bullet\otimes \kk\vac\right) \otimes \mathcal{O}_S\right)\cap \left(\mathcal{L}_{\mathcal{C}\setminus(P_\bullet\sqcup Q)}(V)\cdot \left(\left(M^\bullet\otimes \kk\vac\right) \otimes \mathcal{O}_S\right)\right).
    \end{align*}
    Locally, the sheaf $\mathcal{L}_{\mathcal{C}\setminus(P_\bullet\sqcup Q)}(V)\cdot \left(\left(M^\bullet\otimes \kk\vac\right) \otimes \mathcal{O}_S\right)$ is generated by sections of the form 
    \[\sigma(u\otimes \vac\otimes f)\]
    for $\sigma\in \mathcal{L}_{\mathcal{C}\setminus (P_\bullet\sqcup Q)}(V)$, $u\in M^\bullet$ and $f\in \mathcal{O}_S$. 
    
    Again, the goal here is to show that $\ker\varphi=0$. For this, we will show that 
    \[\iota^{-1}\left(\mathcal{L}_{\mathcal{C}\setminus(P_\bullet\sqcup Q)}(V)\cdot \left(\left(M^\bullet\otimes V\right) \otimes \mathcal{O}_S\right)\right)\subseteq\mathcal{L}_{\mathcal{C}\setminus P_\bullet}(V) \cdot \left(\left(M^\bullet\otimes \kk\vac\right)\otimes \mathcal{O}_S\right).\]
    Since the reverse inclusion is obvious, this would imply that the inclusion is an equality, hence $\ker\varphi=0$.
    
    Now, suppose that 
    \[\sigma(u\otimes \vac\otimes f)\in \left(M^\bullet\otimes \kk\vac\right) \otimes \mathcal{O}_S\] 
    for some $\sigma\in \mathcal{L}_{\mathcal{C}\setminus (P_\bullet\sqcup Q)}(V)$, $u\in M^\bullet$ and $f\in \mathcal{O}_S$.
    Since $\sigma(u)\otimes \vac \otimes f$ is already a section of $\left(M^\bullet\otimes \kk\vac\right) \otimes \mathcal{O}_S$, this is equivalent to requiring that 
    \[u\otimes \sigma_Q(\vac)\otimes f \in \left(M^\bullet\otimes \kk\vac\right) \otimes \mathcal{O}_S.\]
    Once again, since we are tensoring over the field $\kk$, this is equivalent to 
    \[\sigma_Q(\vac) \in \mathcal{O}_S\vac.\]
    Say that $\sigma_Q(\vac) = \vac\otimes g$ for some $g\in \mathcal{O}_S$, so
    \[\sigma(u\otimes \vac\otimes f) = \sigma(u)\otimes \vac\otimes f + u\otimes \vac\otimes gf.\]
    It suffices to show that there is a section $\sigma'\in \mathcal{L}_{\mathcal{C}\setminus P_\bullet}(V)$ such that
    \[\sigma'(u\otimes \vac\otimes f) = \sigma(u\otimes \vac\otimes f).\]
    By definition, we have $\sigma'_Q(\vac) = 0$. It suffices to choose $\sigma'$ such that
    \[\sigma'(u\otimes f) = \sigma(u)\otimes f + u\otimes gf.\]
    There is an injective map $\omega_{\mathcal{C}/S} \hookrightarrow \mathcal{V}_{\mathcal{C}}$ given by $\tau\mapsto \vac\otimes \tau$.
    Choose any $j\in [1,n]$. By the strong residue theorem (\zcref{lem:strong residue}), there exists a section $\psi \in \omega_{\mathcal{C}/S}$ such that $\Res_{P_i}\psi = \delta_{i,j}g$ and $\Res_Q\psi = -g$. Identifying this as a section of $\mathcal{V}_{\mathcal{C}}$, we conclude that there exists a section $\tau\in \mathcal{L}_{\mathcal{C}\setminus (P_\bullet\sqcup Q)}(V)$ such that $\tau_Q = -\vac_{[-1]} \otimes g$ and $\tau_{P_i} = \delta_{i,j}\vac_{[-1]} \otimes g$. Letting $\sigma' = \sigma+\tau$, we then have
    \[\sigma'_{P_i} = \sigma_{P_i} + \tau_{P_i} = \sigma_{P_i} + \delta_{i,j}\vac_{[-1]}\otimes g,\qquad \sigma'_{Q} = \sigma_{Q} + \tau_{Q} = -\vac_{[-1]}\otimes g.\]
    Therefore, $\sigma'(u\otimes \vac\otimes f)  = \sigma(u\otimes \vac\otimes f)$. This shows that $\ker\varphi=0$, so we are done.
\end{proof}
\begin{remark}
    The original proof outlined in \cite{DGT21} for propagation of vacua assumes that the $\N$-graded vertex algebra $V$ is a CFT-type VOA. Our argument shows that the assumption $V_0\cong \kk \vac$ is unnecessary. Their proof also used a different method of showing injectivity; namely, they examined the dual morphism
    \[\varphi^\lor\colon \Hom_{\mathcal{O}_S}\left(\mathbb{V}(V;M^\bullet \otimes_\kk V)_{(\mathcal{C}/S,\, P_\bullet\sqcup Q,\, t_\bullet\sqcup r)},\mathcal{F}\right) \to \Hom_{\mathcal{O}_S}\left(\mathbb{V}(V;M^\bullet)_{(\mathcal{C}/S,P_\bullet,t_\bullet)} ,\mathcal{F}\right).\]
    They specifically let $\mathcal{F} = \mathcal{O}_S$, but the argument is the same.
    By the Yoneda embedding, $\varphi$ is injective if and only if $\varphi^\lor$ is surjective. Given any $\Psi$ in the codomain, it is easy to define the candidate $\widetilde{\Psi}$ such that $\varphi^\lor(\widetilde{\Psi}) = \Psi$, however showing that $\widetilde{\Psi}$ is well-defined is quite difficult. A detailed analysis reveals that you essentially have to do the same argument as what we did in our proof.

    Another method of showing that $\varphi$ is an isomorphism is to construct the inverse map $\varphi^{-1}$. Showing that this inverse map is well-defined is also quite difficult.
\end{remark}
Suppose we have a family of stable pointed curves $(\mathcal{C}\to S,P_\bullet)$ as above, except that $\mathcal{C}\setminus P_\bullet(S)$ is not affine.
After an \'etale base change, we may assume that the family $(\mathcal{C}\to S,P_\bullet)$ admits $m$ additional sections $Q_i\colon S\to \mathcal{C}$ such that $\mathcal{C}\setminus (P_\bullet(S)\sqcup Q_\bullet(S))$ is affine over $S$ and $(\mathcal{C}\to S,P_\bullet\sqcup Q_\bullet)$ is stable. Let $t_\bullet$ and $r_\bullet$ be the formal coordinates at $P_\bullet(S)$ and $Q_\bullet(S)$ respectively. We may calculate
\[\mathbb{V}(V;M^\bullet\sqcup (V,\ldots,V))_{(\mathcal{C}\to S,P_\bullet\sqcup Q_\bullet,t_\bullet\sqcup r_\bullet)}\]
as the usual quotient, and due to \zcref{thm:propagation of vacua} this definition is independent of the choice of $(Q_\bullet,r_\bullet)$. Consider the sheaf associated to the presheaf that assigns to an open $U\subset S$ the filtered colimit
\begin{equation}\label{coinvariants defn}
\varinjlim_{(Q_\bullet,r_\bullet)}\mathbb{V}(V;M^\bullet\sqcup (V,\ldots,V))_{(\mathcal{C}\to U,P_\bullet\sqcup Q_\bullet, t_\bullet\sqcup r_\bullet)},
\end{equation}
which is taken over all pairs $(Q_\bullet,r_\bullet) = \{(Q_j,r_j)\}_{j=1}^m$ consisting of sections $Q_\bullet$ of $\mathcal{C}_U\to U$ such that:
\begin{itemize}
    \item $(\mathcal{C}\to S,P_\bullet\sqcup Q_\bullet)$ is stable;
    \item The image $Q_j(U)$ has trivial intersection with $P_i(U)$ and $Q_k(U)$ for all $i=1,\ldots,n$ and $k\neq j$;
    \item $\mathcal{C}_U\setminus (P_\bullet(U)\sqcup Q_\bullet(U))$ is affine over $U$.
\end{itemize}
One may check that this colimit is isomorphic to $\mathbb{V}(V;M^\bullet)_{(\mathcal{C}\to S,P_\bullet, t_\bullet)}$, as is the case over $\C$.

Another convenient aspect of propagation of vacua is that in order to prove general properties about sheaves of coinvariants associated to a stable curve $\mathcal{C}\to S$, it usually suffices to assume that $\mathcal{C}\setminus P_\bullet(S)$ is affine.

\subsection{The factorization theorem}
In this section, we will set up and prove a generalization of the factorization theorems from \cite[Theorem 7.0.1]{DGT23} and \cite[Lemma 4.4.4]{DGK23} more generally. This is the second major result of our paper. 
We fix a stable, $n$-pointed curve $\mathscr{C}_0\to S_0= \Spec(R)$ for a smooth $\kk$-algebra $R$ as in \zcref{constr:smoothing family}. We denote the marked sections by $P_1,\ldots,P_n\colon S_0\to \mathscr{C}_0$, and we suppose $\mathscr{C}_0$ has a single node $Q$. The normalization may be regarded as a smooth, $(n+2)$-pointed curve $\widetilde{\mathscr{C}_0}\to S_0$ with marked sections $P_\bullet\sqcup Q_\pm$. Let $t_\bullet=(t_1,\ldots,t_n)$ be a collection of formal coordinates on $P_1(S_0),\ldots,P_n(S_0)$ respectively, and let $s_\pm(S_0) = (s_+(S_0),s_-(S_0))$ be a pair of formal coordinates on $Q_+(S_0),Q_-(S_0)$ respectively. Let $M^\bullet=(M^1,\ldots,M^n)$ be a collection of admissible $V$-modules.

\begin{theorem}[Factorization]\label{thm:factorization theorem}
    In the above setting, the morphism
    \[\iota\colon M^\bullet\longrightarrow M^\bullet\otimes \mathfrak{A},\qquad u\mapsto u\otimes 1\otimes 1\]
    induces a canonical isomorphism
    \[\alpha\colon \mathbb{V}(V;M^\bullet)_{(\mathscr{C}_0,P_\bullet,t_\bullet)} \stackrel{\cong}{\longrightarrow} \mathbb{V}(V;M^\bullet\otimes \mathfrak{A})_{(\widetilde{\mathscr{C}_0},P_\bullet\sqcup Q_\pm, t_\bullet\sqcup s_\pm)}. \]
\end{theorem}
\begin{proof}
    Due to \zcref{thm:propagation of vacua}, without loss of generality we may assume that $\mathscr{C}_0\setminus P_\bullet(S_0)$ is affine.
    
    First, we show that the map $\alpha$ is well-defined. Given a section $\sigma\in \mathcal{L}_{\mathscr{C}_0\setminus P_\bullet}(V)$, by \zcref{prop:chiral Lie alg nodal} we may identify it with a section $\sigma\in \mathcal{L}_{\widetilde{\mathscr{C}_0}\setminus (P_\bullet\sqcup Q_\pm)}(V)$ such that
    \[\sigma_{Q_\pm} \in \mathfrak{L}(V)^{\mathsf{L}}_{\le 0}\quad \text{and}\quad [\sigma_{Q_-}]_0 = -\theta[\sigma_{Q_+}]_0.\]
    Recall that 
    \[\mathfrak{A} = \PhiL_0(\Aa_0) \otimes_{\Aa_0}\PhiR_0(\Aa_0) = \PhiL_0(\Aa_0) \otimes_{\UV_0}\PhiR_0(\Aa_0).\]
    Now, $\PhiR_0(\Aa_0)$ is naturally a right $\UV$-module with a natural $\Z_{\le 0}$-grading with respect to the right $\UV$-action. After using the involution $\theta$ to regard $\PhiR_0(\Aa_0) = \PhiL_0({}^\theta\Aa_0)$ as a left $\UV$-module, the $\Z_{\le 0}$-grading becomes a $\Z_{\ge 0}$-grading with respect to the left $\UV$-action. 
    
    Now, choose $u\in M^\bullet$ and $f\in \mathcal{O}_{S_0}$. Since $1\in \Aa_0\subset \PhiL(\Aa_0),\PhiR(\Aa_0)$ is a degree 0 element, we conclude that the action of $\sigma$ is given by
    \begin{align*}
        \alpha(\sigma(u\otimes f)) &= \alpha(\sigma(u)\otimes f)\\ 
        &= \sigma(u)\otimes 1\otimes 1\otimes f\\
        &= \sigma(\alpha(u))\otimes f - u\otimes \sigma_{Q_+}(1)\otimes 1 \otimes f- u\otimes 1\otimes \theta(\sigma_{Q_-})(1)\otimes f.
    \end{align*}
    Recall that $[\sigma_{Q_-}]_0 = -\theta[\sigma_{Q_+}]_0$. Since $\mathfrak{A}$ is a tensor product over $\UV_0$, the degree 0 components of the last two terms of the right-hand side cancel. After canceling, the remaining components of the last two terms must strictly lower the degree. However, such elements must be 0, so we conclude that $\alpha$ is well-defined.

    Now we show that $\alpha$ is surjective. We write a pure tensor element of $\Ac$ as
    \[a\otimes b \eqdef a^1_{[n_1]}\cdots a^r_{[n_r]} \otimes b^1_{[m_1]}\cdots b^s_{[m_s]}.\]
    Let $a' = a^2_{[n_2]}\cdots a^r_{[n_r]}$ and $b' = b^2_{[m_2]}\cdots b^s_{[m_s]}$. Define
    \[d_+ = \deg(a),\qquad d_- = \deg(b),\qquad d_+' = \deg(a'),\qquad d_-' = \deg(b').\]
	Fix $N > n_1$. By \zcref{cor:Riemann Roch calculation}, we can find a section $\mu$ of $\omega_{\mathcal{C}/S}$ such that 
	\[\mu_{Q_+}\in s_+^{n_1}ds_+ + s_+^N\mathcal{O}_{S_0}[\![s_+]\!]\,ds_+,\qquad \mu_{Q_-}\in s_-^N\mathcal{O}_{S_0}[\![s_-]\!]\,ds_-,\]
    and we can find a section $\mathcal{B}$ of $\mathcal{V}_{\widetilde{\mathcal{C}}/S}$ such that $\mathcal{B}_{Q_+} = a^1$. Let $\sigma = \mathcal{B}\otimes \mu$.
    By choosing $N\gg 0$, we have
    \begin{equation}
    \begin{aligned}
        &\sigma(u)\otimes a'\otimes b\otimes f + u\otimes a\otimes b\otimes f\\
        &= \sigma(u)\otimes a'\otimes b\otimes f + u\otimes \sigma_{Q_+}(a')\otimes b\otimes f + u\otimes a\otimes \theta(\sigma_{Q_-})(b)\otimes f\\
        &= \sigma(u\otimes a'\otimes b\otimes f)\in \mathcal{L}_{\widetilde{\mathscr{C}_0}\setminus (P_\bullet\sqcup Q_\pm)}(V)\cdot \left(\left(M^\bullet\otimes \mathfrak{A}\right)\otimes \mathcal{O}_{S_0}\right).
    \end{aligned}
    \end{equation}
    In a similar fashion, we may choose a section $\sigma = \mathcal{B}\otimes \mu$ of $\mathcal{V}_{\mathcal{C}/S}\otimes \omega_{\mathcal{C}/S}$ such that for $N\gg 0$ we have
    \begin{equation}
    \begin{aligned}
        &\sigma(u)\otimes a\otimes b'\otimes f + u\otimes a\otimes b\otimes f\\
        &= \sigma(u)\otimes a\otimes b'\otimes f + u\otimes \sigma_{Q_+}(a)\otimes b'\otimes f + u\otimes a\otimes \theta(\sigma_{Q_-})(b')\otimes f\\
        &= \sigma(u\otimes a\otimes b'\otimes f)\in \mathcal{L}_{\widetilde{\mathscr{C}_0}\setminus (P_\bullet\sqcup Q_\pm)}(V)\cdot \left(\left(M^\bullet\otimes \mathfrak{A}\right)\otimes \mathcal{O}_{S_0}\right).
    \end{aligned}
    \end{equation}
    By double induction on the string lengths of $a$ and $b$, we conclude that every section of $\mathbb{V}(V;M^\bullet\otimes \mathfrak{A})_{(\widetilde{\mathcal{C}}/S,P_\bullet\sqcup Q_\pm, t_\bullet\sqcup s_\pm)}$ is equivalent to a linear combination of sections of the form $u\otimes 1\otimes 1\otimes f$ for $u\in M^\bullet$ and $f\in \mathcal{O}_{S_0}$. We conclude that $\alpha$ is surjective.

    Now we show that $\varphi$ is injective. Similar to before, we have
    \[\ker\alpha = \frac{\iota^{-1}\left(\mathcal{L}_{\widetilde{\mathscr{C}_0}\setminus (P_\bullet\sqcup Q_\pm)}(V)\cdot \left(\left(M^\bullet\otimes \mathfrak{A}\right)\otimes \mathcal{O}_{S_0}\right)\right)}{\mathcal{L}_{\mathscr{C}_0\setminus P_\bullet}(V) \cdot \left(M^\bullet\otimes \mathcal{O}_{S_0}\right)}.\]
    After applying the identification $M^\bullet \cong M^\bullet \otimes \kk 1\otimes \kk 1$, we have the identification
    \[\mathcal{L}_{\mathscr{C}_0\setminus P_\bullet}(V) \cdot \left(M^\bullet\otimes \mathcal{O}_{S_0}\right)\cong \mathcal{L}_{\mathscr{C}_0\setminus P_\bullet}(V) \cdot \left(\left(M^\bullet\otimes \kk 1\otimes \kk 1\right)\otimes \mathcal{O}_{S_0}\right),\]
    where the action of $\sigma\in \mathcal{L}_{\mathscr{C}_0\setminus P_\bullet}(V)\subset \mathcal{L}_{\widetilde{\mathscr{C}_0}\setminus (P_\bullet\sqcup Q_\pm)}(V)$ on $\left(M^\bullet\otimes \kk 1\otimes \kk 1\right)\otimes \mathcal{O}_{S_0}$ is given by
    \[\sigma\cdot(u\otimes 1\otimes 1\otimes f) = \sigma(u)\otimes 1\otimes 1\otimes f\]
    for $\sigma\in \mathcal{L}_{\mathscr{C}_0\setminus P_\bullet}(V)$, $u\in M^\bullet$ and $f\in \mathcal{O}_{S_0}$.
    Since $\iota$ is an inclusion, we have
    \begin{align*}
        &\iota^{-1}\left(\mathcal{L}_{\widetilde{\mathscr{C}_0}\setminus (P_\bullet\sqcup Q_\pm)}(V)\cdot \left(\left(M^\bullet\otimes \mathfrak{A}\right)\otimes \mathcal{O}_{S_0}\right)\right)\\
        &\cong \left(\left(M^\bullet\otimes \kk 1\otimes \kk 1\right) \otimes \mathcal{O}_{S_0}\right)\cap \left(\mathcal{L}_{\widetilde{\mathscr{C}_0}\setminus (P_\bullet\sqcup Q_\pm)}(V)\cdot \left(\left(M^\bullet\otimes \kk 1\otimes \kk 1\right) \otimes \mathcal{O}_{S_0}\right)\right).
    \end{align*}
    Locally, the sheaf $\mathcal{L}_{\widetilde{\mathscr{C}_0}\setminus (P_\bullet\sqcup Q_\pm)}(V)\cdot \left(\left(M^\bullet\otimes \kk 1\otimes \kk 1\right) \otimes \mathcal{O}_{S_0}\right)$ is generated by sections of the form 
    \[\sigma(u\otimes 1\otimes 1\otimes f)\]
    for $\sigma\in \mathcal{L}_{\widetilde{\mathscr{C}_0}\setminus (P_\bullet\sqcup Q_\pm)}(V)$, $u\in M^\bullet$ and $f\in \mathcal{O}_{S_0}$. 

    We carry out a similar argument to the proof of propagation of vacua. Suppose we have a section
	\[\sigma(u\otimes 1\otimes 1\otimes f) \in \left(\left(M^\bullet\otimes \kk 1\otimes \kk 1\right) \otimes \mathcal{O}_{S_0}\right)\cap \left(\mathcal{L}_{\widetilde{\mathscr{C}_0}\setminus (P_\bullet\sqcup Q_\pm)}(V)\cdot \left(\left(M^\bullet\otimes \kk 1\otimes \kk 1\right) \otimes \mathcal{O}_{S_0}\right)\right).\]
	Since $\sigma(u)\otimes \kk 1\otimes \kk 1\otimes f$ is already a section in the right-hand side, this is equivalent to $u\otimes \sigma(1\otimes 1)\otimes f$ being a section in the right-hand side.
    Since we are tensoring over the field $\kk$, this is equivalent to
	\[\sigma(1\otimes 1) = \sigma_{Q_+}(1)\otimes 1 + 1\otimes \theta(\sigma_{Q_-})(1) \in (\kk 1\otimes \kk 1)\otimes \mathcal{O}_{S_0}.\]
	This condition forces $\sigma_{Q_+}(1)$ and $\theta(\sigma_{Q_-})(1)$ to have degree 0 due to the bigrading on $\mathfrak{A}$. Since the tensor product is over $\UV_0$, the above condition is equivalent to
	\[\sigma_{Q_+}(1)\otimes 1 + 1\otimes \theta(\sigma_{Q_-})(1) = ([\sigma_{Q_+}]_0 + \theta([\sigma_{Q_-}]_0))(1)\otimes 1\in (\kk 1\otimes \kk 1)\otimes \mathcal{O}_{S_0}.\]
	Equivalently, $[\sigma_{Q_+}]_0 + \theta([\sigma_{Q_-}]_0)$ acts on $1\otimes 1$ by scaling by some regular function $g\in \mathcal{O}_{S_0}$. Therefore, we have
	\[\sigma(u\otimes 1\otimes 1\otimes f) = \sigma(u)\otimes 1\otimes 1 \otimes f + u\otimes 1\otimes 1\otimes gf.\]
	As with the proof of propagation of vacua, it suffices to show that there is some $\sigma'\in \mathcal{L}_{\mathscr{C}_0\setminus P_\bullet}(V) \subset \mathcal{L}_{\mathscr{C}_0\setminus (P_\bullet\sqcup Q_\pm)}(V)$ such that 
    \[\sigma(u\otimes 1\otimes 1\otimes f) = \sigma'(u\otimes 1\otimes 1\otimes f).\] 
    Once again, such a section exists by the strong residue theorem (\zcref{lem:strong residue}).
\end{proof}

\begin{remark}
    This statement was first proven in \cite[Lemma 4.4.4]{DGK23} as a generalization of \cite[Proposition 6.2.1]{DGT23}. They assume that $V$ is $C_1$-cofinite and CFT-type. Our proof shows that these assumptions are unnecessary. However, one is essentially forced to work with a nodal curve $C$ over $\Spec(\kk)$ or some other sufficiently local base. For instance, this result generally does not hold if we try to na\"ively replace a nodal curve with a smoothing family $\mathscr{C}\to S$. We will discuss how to overcome this issue in \zcref{sec:finiteness-smoothing}.
\end{remark}
Here are some further comments about \zcref{thm:factorization theorem}.
Our proof of the factorization theorem uses induction on the string length of elements of $\PhiL_0(\Aa_0)$ and $\PhiR_0(\Aa_0)$. If one wishes, they can safely assume that all strings are already of length 1 due to the fact that the map $V \to \Aa_{n,m}(V)$ given by $a\mapsto J_{m-n}(a)$ is surjective \cite[Lemma 3.4]{Xu2026}.

In the case where $V$ is rational, we obtain the factorization theorem from \cite{DGT23}:
\begin{lemma}
    Assume that $V$ is rational. Then
    \[\mathfrak{A} \cong \bigoplus_{[W]\in \mathscr{S}(V)} W\otimes_\kk W',\]
    where $\mathscr{S}(V)$ is the set of isomorphism classes of simple $V$-modules.
\end{lemma}
\begin{proof}
    This is a consequence of the arguments in \cite[Proposition 4.28]{Griffin26}. The same arguments may be found in \cite[Remark 3.4.6]{DGK23} or \cite[Proposition 4.0.9]{DGK24} for the case over $\C$.
\end{proof}
If $V$ is rational, then there are only finitely many simple $V$-modules up to isomorphism. This is because there is a bijection between simple $V$-modules and simple $\Aa_0$-modules, and $V$ being rational implies that $\Aa_0$ is semisimple (hence there are finitely many simple $\Aa_0$-modules up to isomorphism).
\begin{corollary}
    If $V$ is rational, then we have an isomorphism
    \[\alpha\colon \mathbb{V}(V;M^\bullet)_{(\mathscr{C}_0,P_\bullet,t_\bullet)} \stackrel{\cong}{\longrightarrow} \bigoplus_{[W]\in \mathscr{S}(V)}\mathbb{V}(V;M^\bullet\otimes W\otimes W')_{(\widetilde{\mathscr{C}_0}/S_0,P_\bullet\sqcup Q_\pm, t_\bullet\sqcup s_\pm)}.\]
\end{corollary}
This is a slightly stronger version of the factorization theorem from \cite[Theorem 7.0.1]{DGT23}. We do not need $C_2$-cofiniteness for this isomorphism, as was pointed out in \cite[Remark 3.4.6]{DGK23}.

\subsection{Relationship with intertwining operators}
While we are primarily interested in sheaves of coinvariants, we can also study the dual sheaf of conformal blocks:
\begin{definition}
    In the setting of \zcref{sec:rep of chiral}, we define the \textit{sheaf of conformal blocks} to be
    the $\mathcal{O}_S$-dual sheaf of $\mathbb{V}(V;M^\bullet)_{(\mathcal{C}\to S,P_\bullet,t_\bullet)}$.
\end{definition}
Here, we show how intertwining operators of a certain kind are equivalent to a sheaf of conformal blocks on $\P^1_S\to S = \Spec\kk[z,z^{-1}]$, where we fix the points $(0,z,z^{-1})$.
Let $M^1,M^2,M^3$ be $\Aut\mathcal{O}$-equivariant admissible $V$-modules, and assume the graded components of $M^3$ are all finite-dimensional over $\kk$. Let $S = \Spec\kk[z,z^{-1}]$. Over $\P^1_S\to S$, pick points $P_\bullet=(0,z,z^{-1})$ and local coordinates $t_\bullet = (t-z,t,t^{-1})$.\footnote{The local coordinate $t$ is not the same as a parameter $x$ such that $\P^1_\kk = \operatorname{Proj} \kk[x]$. We are just labeling the chosen local coordinate at $z$ by $t$. Geometrically, setting $t=0$ lands us at the point $z$, so $t-z$ defines a formal coordinate shifted to the point $0$.}
\begin{theorem}\label{thm:intertwining}
    Let $V$ be a $\Z$-graded vertex $\kk$-algebra with an $\Aut\mathcal{O}$-action.   Then every intertwining operator $\cal Y(\cdot , z) \in \operatorname{IO}\binom{M^3}{M^1\ M^2}$ given as a map $M^1\otimes M^2 \to M^3[\![z,z^{-1}]\!]$ is equivalent to a conformal block
    \[\Phi\in\Hom(\mathbb{V}(V;M^1,M^2,(M^3)')_{(\P^1_S\to S,P_\bullet,t_\bullet)},\mathcal{O}_S)\]
    such that $\Phi$ is $\kk^\times$-equivariant, by which we mean 
    \[\Phi((\Pi^{M^1}(k^{t\partial_t})m^1\otimes \Pi^{M^2}(k^{t\partial_t})m^2\otimes \Pi^{M^3}(k^{t\partial_t})\varphi) \otimes f(z)) = \Phi((m^1\otimes m^2\otimes \varphi) \otimes f(kz))\] 
    for all $m^1 \in M^1$, $m^2 \in M^2$, $\varphi\in (M^3)'$, and $k\in \kk^\times$.
\end{theorem}
\begin{proof}
    Note the equivalence of spaces 
    \begin{align*}
        &\Hom((M^1\otimes M^2\otimes (M^3)')\otimes \mathcal{O}_S,\mathcal{O}_S) \\
        &\cong \Hom_{\kk[z,z^{-1}]}((M^1\otimes M^2\otimes (M^3)')\otimes_{\kk} \kk[z,z^{-1}],\kk[z,z^{-1}])\\
        &= \Hom_{\kk}(M^1\otimes M^2\otimes (M^3)',\End_{\kk}(\kk[z,z^{-1}]))\\
        &= \Hom_{\kk}(M^1\otimes M^2\otimes (M^3)',\kk[z,z^{-1}]),
    \end{align*}
    where in the first step we used the fact that $S$ is affine, and in the second we used the tensor-hom adjunction. Using this adjunction again, we have
    \begin{align*}
        \Hom_{\kk}\left(M^1\otimes M^2\otimes (M^3)', \kk[z,z^{-1}]\right) &\cong \Hom_{\kk}\left(M^1\otimes M^2, \Hom_{\kk}((M^3)',\kk[z,z^{-1}])\right)\\
        &= \Hom_{\kk}\left(M^1\otimes M^2, (M^3)''[z,z^{-1}]\right).
    \end{align*}
    Note that $(M^3)'' \cong M^3$ since the graded components of $M^3$ are finite-dimensional.
    Given a conformal block $\Phi\colon M^1\otimes M^2\otimes (M^3)' \to \kk[z,z^{-1}]$, we may use the pairing $(M^3)'\otimes M^3\to \kk$ to regard it as a map
    \[m^1\otimes m^2\otimes \varphi \mapsto \langle\varphi,\cal Y(m^1,z)m^2\rangle,\]
    where $\cal Y(\cdot,z)\colon M^1\otimes M^2\mapsto M^3[\![z,z^{-1}]\!]$ is our claimed intertwining operator, which we write as
    \[\cal Y(m^1,z) = \sum_{n\in \Z}m^1_{(n)}z^{-1-n}.\]
    First, we show that the series is lower-truncated; that is, for all $m^2\in M^2$ we have $m^1_{(n)}m^2 = 0$ for $n\gg 0$.
    Now we examine the $\kk^\times$-equivariance requirement. We have \[\Phi(m^1\otimes m^2\otimes \varphi)(kz) = \Phi(k^{L_0}m^1\otimes k^{L_0}m^2\otimes k^{L_0}\varphi)(z).\] 
    Equivalently,
    \begin{align*}
        \langle\varphi, k^{-L_0} \cal Y(k^{L_0}m^1,z) k^{L_0}m^2\rangle &= \langle k^{L_0}\varphi, \cal Y(k^{L_0}m^1,z) k^{L_0}m^2\rangle\\
        &= \langle\varphi,\cal Y(m^1,kz)m^2\rangle.
    \end{align*}
    That is, we have $k^{\deg(m^1_{(n)}m^2)-\deg(m^1)-\deg(m^2)} = k^{-1-n}$ for all $k\in \kk^\times$. The field $\kk$ is infinite because it is algebraically closed, so this implies that $\deg(m^1_{(n)}m^2) = \deg(m^1)+\deg(m^2)-1-n$ for all $m^1,m^2$ and $n\in \Z$. Since the $\Z$-grading is lower-truncated, we have $m^1_{(n)}m^2 = 0$ for $n\gg 0$.

    It remains to check that the Jacobi identity for intertwining operators holds. This follows from the fact that conformal blocks factor through coinvariants.
    Note that $\P^1_{S}$ minus $\{0,z,z^{-1}\}$ is isomorphic to $\Spec \kk[t,z,z^{-1},t^{-1},(t-z)^{-1}]$. We choose local coordinates $t_\bullet = (t-z,t,t^{-1})$ on the points $(0,z,z^{-1})$. Then the chiral Lie algebra may be identified with
    \[\mathcal{L}_{\P^1_S\to S}(V) = \frac{ V[t,z,t^{-1},z^{-1},(t-z)^{-1}] }{( \partial_t^{(n)} - (-1)^nT^{(n)})_{n>0}}.\]
    Fix a rational function $f(t,z)\in \kk[t,z,t^{-1},z^{-1},(t-z)^{-1}]$ and $a\in V$. Explicitly, the quotient by the action of the chiral Lie algebra guarantees that the diagonal action is 0:
    \begin{align*}
        &\Phi((a\otimes f(t,z))\cdot m^1 \otimes m^2\otimes \varphi) + \Phi(m^1 \otimes (a\otimes f(t,z))\cdot m^2\otimes \varphi)\\
        &\equiv -\Phi(m^1\otimes m^2 \otimes (a\otimes f(t,z))\cdot \varphi)
    \end{align*}
    Using the local coordinate $t-z$ at 0, it is easy to see the first term on the left-hand side is
    \begin{align*}
        &\Phi((a\otimes f(t,z))\cdot m^1 \otimes m^2\otimes \varphi) \\
        &= \left\langle \varphi,\Res_{t-z}\left(\cal Y(Y(a,t-z)m^1,z)m^2 \, \iota_{z,t-z}f(t,z)\,d(t-z)\right)\right\rangle.
    \end{align*}
    Similarly, using the local coordinate $t$ at the point $z$, the second term on the left-hand side is
    \begin{align*}
        &\Phi( m^1 \otimes (a\otimes f(t,z))\cdot m^2\otimes \varphi) \\
        &= \left\langle \varphi,\cal Y(m^1,z)\Res_{t}\left(Y(a,t)m^2 \, \iota_{z,t}f(t,z)\, dt\right)\right\rangle.
    \end{align*}
    It remains to check what the last term is.
    Using the change from a right $V$-module action to a left $V$-module action given by the involution $\theta$ and a change of variables, we have
    \begin{align*}
        &-\Phi(m^1\otimes m^2 \otimes (a\otimes f(t,z))\cdot \varphi)\\
        &=-\left\langle \Res_{t^{-1}}\left(Y(e^{t^{-1}L_1}(-t^{2})^{L_0}a,t^{-1})\varphi\, \iota_{z,t^{-1}}f(t,z)\,dt^{-1}\right),\cal Y(m^1,z)m^2\right\rangle \\
        &=-\Res_{t^{-1}}\left(\left\langle Y(e^{t^{-1}L_1}(-t^{2})^{L_0}a,t^{-1})\varphi,\cal Y(m^1,z)m^2\right\rangle \iota_{z,t^{-1}}f(t,z)\,dt^{-1}\right)\\
        &=-\Res_{t}\left(t^2\left\langle Y(e^{t^{-1}L_1}(-t^{2})^{L_0}a,t^{-1})\varphi,\cal Y(m^1,z)m^2\right\rangle \iota_{z,t^{-1}}f(t,z)\,dt^{-1}\right)\\
        &=\Res_{t}\left(\left\langle Y(e^{t^{-1}L_1}(-t^{2})^{L_0}a,t^{-1})\varphi,\cal Y(m^1,z)m^2\right\rangle \iota_{t,z}f(t,z)\,dt\right)\\
        &=\Res_{t}\left(\left\langle \varphi,Y(a,t)\cal Y(m^1,z)m^2\right\rangle \iota_{t,z}f(t,z)\,dt\right)\\
        &=\left\langle \varphi,\Res_{t}\left(Y(a,t)\cal Y(m^1,z)m^2\, \iota_{t,z}f(t,z)\,dt\right)\right\rangle.
    \end{align*}
    Since the action holds for all $\varphi$ and we are over a field, we may remove the $\langle\varphi,-\rangle$ and conclude that for all $m^1\in M^1$, $m^2\in M^2$, $\varphi\in (M^3)'$, and $f(t,z)\in \kk[t,z,t^{-1},z^{-1},(t-z)^{-1}]$
    \begin{align*}
        &\Res_{t-z}\left( \cal Y(Y(a,t-z)m^1,z)m^2 \, \iota_{z,t-z}f(t,z)\,d(t-z)\right)\\
        &=\Res_{t}\left(\left(Y(a,t)\cal Y(m^1,z)m^2\, \iota_{t,z}f(t,z) - \cal Y(m^1,z)Y(a,t)m^2\, \iota_{z,t}f(t,z)\right)dt\right).
    \end{align*}
    This is precisely the Jacobi identity for intertwining operators. This shows the bijection.
\end{proof}
The above proof requires $\kk$ to be infinite for the grading argument to work. Indeed, if $\kk=\F_q$ is a finite field for some $q = p^m$ for a prime $p$, then $k^n = 1$ for all $k\in \F_q^{\times}$ implies that $n$ is divisible by $q-1$. Also, we needed to assume that $M^3$ has finite-dimensional graded components to dualize correctly.

\subsection{Sheaves of coinvariants on $\overline{\mathcal{M}}_{g,n}$}\label{sec:descent}
Let $\widehat{\mathcal{M}}_{g,n}$ be the moduli space of smooth, $n$-pointed curves with formal coordinates at the marked points. Equivalently, this is the restriction of $\widetriangle{\mathcal{M}}_{g,n}$ over the locus $\mathcal{M}_{g,n}$ of smooth, $n$-pointed curves. Because all of our constructions thus far commute with base change, we obtain a chiral Lie algebra sheaf $\mathcal{L}(V)$ on $\widetriangle{\mathcal{M}}_{g,n}$ whose pullback along a morphism $S\to \widetriangle{\mathcal{M}}_{g,n}$ corresponding to a stable relative curve $\mathcal{C}\to S$ is precisely the sheaf $\mathcal{L}_{\mathcal{C}\setminus P_\bullet}(V)$. Another way to obtain this sheaf is to first define it on $\widehat{\mathcal{M}}_{g,n}$ and glue via formal smoothings to obtain a sheaf on $\widetriangle{\mathcal{M}}_{g,n}$ (using the descent result of \cite{BL}). Either way, we can use this to obtain a sheaf of coinvariants on $\widetriangle{\mathcal{M}}_{g,n}$ associated to a collection of admissible $V$-modules $M^1,\ldots,M^n$, which we denote by $\widetriangle{\mathbb{V}}(V;M^\bullet)$.

Now, we will show that our setup can be used to induce sheaves of coinvariants on the moduli space $\overline{\mathcal{M}}_{g,n}$ over $\Spec(\kk)$.
For this section, we will work with $\Aut\mathcal{O}$-equivariant $V$-modules $W$ with a conformal dimension $c_W$.
Over $\C$, it is possible to set the conformal dimension $c_M$ to be complex since we can take complex powers. That said, we will only consider when the conformal dimension is a rational number since for any algebraically closed field we can only take rational powers in general.

The main goal of this section is to construct the sheaf of coinvariants on $\overline{\mathcal{M}}_{g,n}$, which is given by descent in two steps. These two steps correspond to the two groups involved in the semidirect product decomposition $\Aut\mathcal{O} = \Aut_+\mathcal{O}\rtimes \G_m$. Indeed, the $(\Aut\mathcal{O})^{\times n}$-torsor $\widetriangle{\mathcal{M}}_{g,n} \to\overline{\mathcal{M}}_{g,n}$ can be written as a composition of an $(\Aut_+\mathcal{O})^{\times n}$-torsor and a $\G_m^{\times n}$-torsor:
\begin{equation}
\label{eq:torsors}
\begin{tikzcd}
&\widetriangle{\mathcal{M}}_{g,n} \arrow[rightarrow]{dl}[swap]{(\Aut_+\mathcal{O})^{\times n}} \arrow[rightarrow]{dd}{(\Aut\mathcal{O})^{\times n}} \\
\overline{\mathcal{J}}_{g,n}^{\times} \arrow[rightarrow]{dr}[swap]{\G_m^{\times n}}&\\
& \overline{\mathcal{M}}_{g,n}.
\end{tikzcd}
\end{equation}
Here $\overline{\mathcal{J}}_{g,n}^{\times}$ is the space of tuples $(C,P_\bullet, \tau_\bullet)$ consisting of a stable, $n$-pointed curve $(C,P_\bullet)$ together with a tuple $\tau_\bullet=(\tau_1,\dots,\tau_n)$ with $\tau_i$ a 1-jet of a formal coordinate at $P_i$ for each $i$. The group scheme $(\Aut_+\mathcal{O})^{\times n}$ acts on the sheaf $\widetriangle{\mathbb{V}}(V;M^\bullet)$ on $\widetriangle{\mathcal{M}}_{g,n}$; descending along the $(\Aut_+\mathcal{O})^{\times n}$-torsor in \zcref[noname]{eq:torsors}, one obtains a sheaf of coinvariants on $\overline{\mathcal{J}}_{g,n}^{\times}$, which we denote as $\mathbb{V}^J(V;M^\bullet)$. See \cite[\textsection 6.3.1]{DGT21} for more details.

The descent of $\mathbb{V}^J(V;M^\bullet)$ from $\overline{\mathcal{J}}_{g,n}^{\times}$ to a sheaf over $\overline{\mathcal{M}}_{g,n}$ is obtained as follows. We tensor $\mathbb{V}^J(V;M^\bullet)$ with a specific line bundle to obtain a sheaf on which we have an action of $\G_m^{\times n}$. After descending this new sheaf, we then tensor with the dual of the line bundle, producing the apporpriate sheaf on $\overline{\mathcal{M}}_{g,n}$. In the case where $n>1$, we may simply iterate the procedure used for $n=1$, where we consider one marked point and a $V$-module $M$ with conformal dimension $c_M\in \Q$. 

We write $c_M$ as a reduced fraction $c_M = a/d$, where $a\in \Z$ and $d\in \Z_{>0}$. If $\Char\kk=p>0$, then we require that $d$ is not divisible by $p$. We denote $\cal{M} \eqdef \overline{\cal{M}}_{g,1}$ and $\cal{J} \eqdef \overline{\mathcal{J}}_{g,1}^{\times}$. We have a natural map $\pi\colon \cal{J} \to \cal{M}$ given by forgetting the extra data.
We define the line bundles $\mathcal{L}_{\cal{M}} = (\Psi^\lor)^{\otimes a}$, and $\mathcal{L}_{\cal{J}} = \pi^*\mathcal{L}_{\cal{M}}$. Here $\Psi =  P^*(\omega_{\widetriangle{\mathcal{M}}_{g,1}/\overline{\mathcal{M}}_{g,1}})$.

\subsubsection{Root stacks}
The root stack $\sqrt[\leftroot{-2}\uproot{2}d]{\mathcal{L}_{\mathcal{M}}/\mathcal{M}}$ is the stack parametrizing $d$-th roots of the line bundle $\mathcal{L}_\mathcal{M}$. That is, $\sqrt[\leftroot{-2}\uproot{2}d]{\mathcal{L}_{\mathcal{M}}/\mathcal{M}}$ represents the functor that associates to every $\mathcal{M}$-scheme $\phi \colon Y \to \mathcal{M}$ the groupoid of pairs $(\mathcal{N}, f)$ where $\mathcal{N}$ is a line bundle on $Y$ and $f \colon \mathcal{N}^{\otimes d}\to \phi^*\mathcal{L}_{\mathcal{M}}$ is an isomorphism. An isomorphism between $(\mathcal{N}_1, f_1)$ and $(\mathcal{N}_2, f_2)$ is an isomorphism of line bundles $g \colon  \mathcal{N}_1\to \mathcal{N}_2$  such that $f_2 \circ g^{\otimes d} = f_1$. The root stack $\sqrt[\leftroot{-2}\uproot{2}d]{\mathcal{L}_{\mathcal{J}}/\mathcal{J}}$ is defined similarly. 

Since $\mathcal{L}_{\mathcal{J}}$ is the pullback of $\mathcal{L}_{\mathcal{M}}$ along $\pi$, we can write $\sqrt[\leftroot{-2}\uproot{2}d]{\mathcal{L}_{\mathcal{J}}/\mathcal{J}}$ as the pullback of $\sqrt[\leftroot{-2}\uproot{2}d]{\mathcal{L}_{\mathcal{M}}/\mathcal{M}}$ and $\mathcal{J}$:
\begin{equation}\label{CDRS}
\begin{tikzcd}
\sqrt[\leftroot{-2}\uproot{2}d]{\mathcal{L}_{\mathcal{J}}/\mathcal{J}}  \arrow{r}{\sqrt[\leftroot{-2}\uproot{2}d]{\pi}} \arrow[swap]{d}{p_{\mathcal{J}}} & \sqrt[\leftroot{-2}\uproot{2}d]{\mathcal{L}_{\mathcal{M}}/\mathcal{M}} \arrow{d}{p_{\mathcal{M}}} \\
\mathcal{J} \arrow{r}{\pi} & \mathcal{M}.
\end{tikzcd}
\end{equation} 
In particular, $\sqrt[\leftroot{-2}\uproot{2}d]{\mathcal{L}_{\mathcal{J}}/\mathcal{J}}  \to \sqrt[\leftroot{-2}\uproot{2}d]{\mathcal{L}_{\mathcal{M}}/\mathcal{M}}$ is a $\mathbb{G}_m$-torsor. The stacks $\sqrt[\leftroot{-2}\uproot{2}d]{\mathcal{L}_{\mathcal{M}}/\mathcal{M}}$ and $\sqrt[\leftroot{-2}\uproot{2}d]{\mathcal{L}_{\mathcal{J}}/\mathcal{J}}$ admit universal line bundles $\mathcal{U}_{\mathcal{M}}$ and $\mathcal{U}_{\mathcal{J}}=\left(\sqrt[\leftroot{-2}\uproot{2}d]{\pi}\right)^* \mathcal{U}_{\mathcal{M}}$ such that $\mathcal{U}_{\mathcal{M}}^{\otimes d} = p_{\mathcal{M}}^* \, \mathcal{L}_{\mathcal{M}}$ and  $\mathcal{U}_{\mathcal{J}}^{\otimes d} = p_{\mathcal{J}}^* \, \mathcal{L}_{\mathcal{J}}$.
The category of quasi-coherent sheaves on the root stack $\sqrt[\leftroot{-2}\uproot{2}d]{\mathcal{L}_{\mathcal{J}}/\mathcal{J}}$ (resp. $\sqrt[\leftroot{-2}\uproot{2}d]{\mathcal{L}_{\mathcal{M}}/\mathcal{M}}$) admits an eigendecomposition with respect to the action of the inertia group $\mu_d$ of $d$-th roots of unity, and the degree zero component consists of pullbacks of quasi-coherent sheaves on $\mathcal{J}$ (resp. $\mathcal{M}$). This follows from Maschke's theorem.\footnote{This usage of Maschke's theorem is the only characteristic-dependent statement in this descent argument.} By \cite[Lemma 3.1.1.7]{Lieblich2008},
there is an identification between the eigensheaves for the trivial character on the root stack and sheaves on the base stack. Therefore, taking the pullback via $p_{\mathcal{J}}$ is fully faithful. That is, two quasi-coherent sheaves on $\mathcal{J}$ are isomorphic if and only if they are isomorphic when pulled back to $\sqrt[\leftroot{-2}\uproot{2}d]{\mathcal{L}_{\mathcal{J}}/\mathcal{J}}$.

\subsubsection{The final descent}  \label{sec:finaldescent}
Let $\mathbb{V}^{J}:=\mathbb{V}^J(V;M )$. The quasi-coherent sheaf $p_{\mathcal{J}}^*\mathbb{V}^{J}\otimes \mathcal{U}_{\mathcal{J}}$
on $\sqrt[\leftroot{-2}\uproot{2}d]{\mathcal{L}_{\mathcal{J}}/\mathcal{J}}$  has an action of $\mathbb{G}_m$ as in \cite[\S\S 4.2.1, 6.3.2]{DGT21}.\footnote{For this action, one uses the filtration on $\mathbb{V}^J(V;M)$ induced by the $\mathbb{Z}_{\ge 0}$-grading of the module $M$.} 
Descending along the $\mathbb{G}_m$-torsor $\sqrt[\leftroot{-2}\uproot{2}d]{\mathcal{L}_{\mathcal{J}}/\mathcal{J}}  \to \sqrt[\leftroot{-2}\uproot{2}d]{\mathcal{L}_{\mathcal{M}}/\mathcal{M}}$, we obtain a sheaf $\mathcal{F}$ on $\sqrt[\leftroot{-2}\uproot{2}d]{\mathcal{L}_{\mathcal{M}}/\mathcal{M}}$ for which
\begin{equation*}\label{rootdescent}
\left(\sqrt[\leftroot{-2}\uproot{2}d]{\pi}\right)^*\mathcal{F}\cong p_{\mathcal{J}}^*\mathbb{V}^{J}\otimes \mathcal{U}_{\mathcal{J}}.
\end{equation*}

Tensoring the above with $\left(\sqrt[\leftroot{-2}\uproot{2}d]{\pi}\right)^*\mathcal{U}_{\mathcal{M}}^{\vee}= \mathcal{U}_{\mathcal{J}}^{\vee}$, we have  
\begin{equation}\label{rootdiagram}
\left(\sqrt[\leftroot{-2}\uproot{2}d]{\pi}\right)^*\left(\mathcal{F} \otimes \mathcal{U}_{\mathcal{M}}^{\vee} \right) 
\cong p_{\mathcal{J}}^*\mathbb{V}^{J}.
\end{equation}
It follows that $\mathcal{F} \otimes \mathcal{U}_{\mathcal{M}}^{\vee}$ lives in the degree zero component, hence it descends to 
 a sheaf $\mathbb{V}(V;M)$ on $\mathcal{M}$ such that
$p_{\mathcal{M}}^* \, \mathbb{V}(V;M) \cong \mathcal{F} \otimes \mathcal{U}_{\mathcal{M}}^{\vee}$. By its construction and the commutativity of \zcref[noname]{CDRS}, one has 
\begin{equation*}\label{FF}
\left(\pi \circ p_{\mathcal{J}}\right)^*  \mathbb{V}(V;M)  = \left(p_{\mathcal{M}} \circ \sqrt[\leftroot{-2}\uproot{2}d]{\pi} \right)^*  \mathbb{V}(V;M) =p_{\mathcal{J}}^*\mathbb{V}^{J}.
\end{equation*} Since the pullback of sheaves to a root stack is fully faithful, we deduce that $\mathbb{V}^{J} = \pi^* \, \mathbb{V}(V;M)$. In particular, $\mathbb{V}^{J}$ descends to a well-defined sheaf $\mathbb{V} = \mathbb{V}(V;M)$ on $\overline{\mathcal{M}}_{g,1}$.

The arguments of this section allow us to conclude with the following:
\begin{theorem}\label{sheaves of coinvariants descent}
    Let $V$ be an $\N$-graded vertex algebra over an algebraically closed field $\kk$. Let $M^1,\ldots,M^n$ be a collection of $V$-modules with $\Aut\mathcal{O}$-actions (cf. \zcref{def:Aut O V-module}) and conformal dimensions $c_{M^i} = \frac{a_i}{d_i}$ such that $d_i$ is not divisible by $p = \Char \kk$ if $p>0$. Then the sheaf $\widetriangle{\mathbb{V}}(V;M^\bullet)$ on $\widetriangle{\mathcal{M}}_{g,n}$ descends to a well-defined quasi-coherent sheaf $\mathbb{V}(V;M^\bullet)$ on $\overline{\mathcal{M}}_{g,n}$.
\end{theorem}

\section{Finiteness and smoothing of coinvariants}\label{sec:finiteness-smoothing}
Our results so far impose little to no restrictions on the underlying vertex algebra $V$ or its admissible $V$-modules. For example, all of our results above apply to any (not necessarily CFT-type) VOA over any algebraically closed field $\kk$ of characteristic $0$. We will now impose various restrictions on the underlying vertex algebra $V$ to obtain finiteness properties of coinvariants. Many of the arguments in this section will correspond to results from \cite{DGT23} or \cite{DGK23}, and their proofs may often be repeated here. We will omit proofs when the original argument is sufficient, as is the case for the most part.

\subsection{Finite-dimensionality of coinvariants}
Here we will use the same arguments as Abe and Nagatomo to show that spaces of coinvariants associated to a smooth, $n$-pointed curve $C$ are finite-dimensional over $\kk$.

We introduce some notation first. Given a smooth, $n$-pointed curve $(C,P_\bullet)$ over $\kk$ and an effective divisor $D = \sum_{i=1}^m n_iQ_i$ on $C$ not supported at $P_\bullet$, let
\[\mathcal{L}_{C\setminus P_\bullet}(V,D) \eqdef H^0\left(C\setminus P_\bullet, \frac{\mathcal{V}_C\otimes \omega_C(-D)}{\Im \nabla}\right).\]
That is, $\mathcal{L}_{C\setminus P_\bullet}(V,D)$ is the subspace of $\mathcal{L}_{C\setminus P_\bullet}(V)$ consisting of sections vanishing with order at least $n_i$ at $Q_i$ for each $i$. This is a Lie subalgebra of $\mathcal{L}_{C\setminus P_\bullet}(V)$.

We fix the following data for our next result. Let $V$ be a vertex $\kk$-algebra with an $\Aut\mathcal{O}$-action. Let $C$ be a smooth, projective curve with distinct points $P_\bullet=(P_1,\ldots,P_n)$, and let $D$ be an effective divisor on $C$ not supported at $P_\bullet$. Fix formal coordinates $t_i$ at $P_i$ for all $i$. 
\begin{proposition}\label{prop:fin-dim-coinvars}
    Suppose that $V$ is CFT-type and $C_2$-cofinite. For any $n$-tuple of $V$-modules $M^\bullet = (M^1,\ldots,M^n)$ that are finitely generated over $\UV$ by their degree 0 part, the space of coinvariants 
\[M^\bullet_{\mathcal{L}_{C\setminus P_\bullet}(V,D)} = M^\bullet/\left(\mathcal{L}_{C\setminus P_\bullet}(V,D)\cdot M^\bullet\right)\]   
    is finite-dimensional.
\end{proposition}
\begin{proof}
    Let $U$ be a finite-dimensional subspace of $V$ such that $V = U\oplus C_2(V)$. Let $d_U$ be the maximum degree of the homogeneous elements in $U$. 
    We claim that there exists an integer $N$ such that $h^0(C,\omega_C^{\otimes 1-k}(lP_i-D))\neq 0$ for all $k\le d_U$, $l\ge N$, and $i\in \{1,\ldots,n\}$.
    To start, we have
    \[\deg(\omega_C^{\otimes 1-k}(lP_i-D)) = (1-k)\deg(\omega_C) - \deg (D) + l = (1-k)(2g-2) - \deg D + l.\]
    By Riemann-Roch, we have
    \[h^0(C,\omega_C^{\otimes 1-k}(lP_i-D)) - h^1(C,\omega_C^{\otimes 1-k}(lP_i-D)) = \deg(\omega_C^{\otimes 1-k}(lP_i-D)) + 1 - g.\]
    By Serre duality, we have $h^1(C,\omega_C^{\otimes 1-k}(lP_i-D)) = 0$ if $\deg(\omega_C^{\otimes 1-k}(lP_i-D)) > 2g-2$. Moreover, 
    \[h^0(C,\omega_C^{\otimes 1-k}(lP_i-D)) = \deg(\omega_C^{\otimes 1-k}(lP_i-D)) + 1 - g \ge  0.\]
    Equivalently,
    \[\deg(\omega_C^{\otimes 1-k}(lP_i-D)) + 1 - g \ge  0\iff l \ge  1-g + k(2g-2) + \deg D\]
    To have $h^0(C,\omega_C^{\otimes 1-k}(lP_i-D))\neq 0$ for any $k$ and $i$, we impose the inequality
    \[l > 1-g + k(2g-2) + \deg D.\]
    So we pick $N = 2-g + d_U(2g-2) + \deg D$, which implies that $h^0(C,\omega_C^{\otimes 1-k}(lP_i-D)) \neq 0$ for $k\le d_U$, $l\ge N$, and $i=1,\ldots,n$.
    
    The rest of the proof follows from the arguments of Step 2 of the proof of \cite[Proposition 5.1.1]{DGT23}, making use of \zcref{AN Prop 4.5} and the fact that finitely generated $V$-modules for a $C_2$-cofinite CFT-type vertex algebra are themselves $C_2$-cofinite.
\end{proof}
We can carry out the same argument in families as in \cite{DGT23}.
\begin{corollary}
\label{cor:CoherenceOnMgn} 
Let $\mathcal{C}\to S$ be a smooth, $n$-pointed family of curves. Suppose that $V$ is CFT-type and $C_2$-cofinite. For any collection of $V$-modules $M^1, \dots, M^n$ that are finitely generated by their degree 0 parts, the sheaf of coinvariants $\mathbb{V}(V; M^{\bullet})_{\left(\mathcal{C}/S, P_\bullet, t_\bullet\right)}$ is a coherent $\mathcal{O}_{S}$-module.
\end{corollary}

\subsection{Strong identity elements and smoothing}
The relationship between the mode transition algebra $\mathfrak{A}$ and sheaves of coinvariants associated to a vertex algebra $V$ as described in \cite{DGK23} may be carried out in a similar manner in positive characteristic. We have already seen the appearance of $\mathfrak{A}$ in the factorization theorem, but we only used the basic parts of its definition to prove the result. We will now study the relationship between smoothing conditions on $V$ associated to formal smoothings $\mathscr{C}\to S$ and strong identity elements.
\begin{definition}
    Given a smoothing family $(\mathscr{C},P_\bullet,t_\bullet)$ and a collection of $V$-modules $M^1,\ldots,M^n$, an element $\one = \sum_{d\ge 0}\one_d q^d \in \Ac[\![q]\!]$ defines a \textit{smoothing map} for $M^\bullet$ over $(\mathscr{C},P_\bullet,t_\bullet)$ if $\one_0 = 1\in \Ac_0$ and the map $M^\bullet\to M^\bullet \otimes \Ac[\![q]\!]$, $u\mapsto u\otimes \one$ extends by linearity and $q$-adic continuity to an $\mathcal{L}_{\mathscr{C}\setminus P_\bullet}(V)$-module morphism $\alpha\colon M^\bullet[\![q]\!]\to (M^\bullet\otimes \Ac)[\![q]\!]$. We say that $\one$ defines a \textit{smoothing map} for $V$ if it defines a smoothing map for all collections of $V$-modules $M^\bullet$ and all families $(\mathscr{C},P_\bullet,t_\bullet)$.
    
    We say smoothing holds for $M^\bullet$ over the family $(\mathscr{C},P_\bullet,t_\bullet)$ if there is an element $\one \in \Ac[\![q]\!]$ giving a smoothing map for $M^\bullet$ over $(\mathscr{C},P_\bullet,t_\bullet)$. We say $V$ \textit{satisfies smoothing} if smoothing holds for all collections of $V$-modules $M^\bullet$ and all families $(\mathscr{C},P_\bullet,t_\bullet)$.
\end{definition}
First, we can characterize smoothing maps using the strong identity equations:
\begin{proposition}
    Let $V$ be a $\Z$-graded vertex $\kk$-algebra with an $\Aut\mathcal{O}$-action, and let $\one_d\in \Ac_d$ for all $d\in \N$. Then $\one = \sum_{d\ge 0}\one_dq^d$ defines a smoothing map for $M^\bullet$ over $(\mathscr{C},P_\bullet,t_\bullet)$ if and only if the sequence $(\one_d)$ satisfies the strong identity equations.
\end{proposition}
\begin{proof}
    This may be proven in the exact same way as in \cite[Proposition 5.1.2]{DGK23}. Their proof uses \cite[Lemma 4.4.3]{DGK23}, and \zcref{prop: exists section chiral Lie alg} is a generalization of that result, so it can be applied in our setting.
\end{proof}
Because one side of this equivalence does not depend on the modules $M^\bullet$, we have the following:
	
\begin{corollary}[\cite[Corollary 5.1.3]{DGK23}]
    Smoothing holds for a collection of $V$-modules $M^\bullet$ over a family $(\mathscr{C},P_\bullet,t_\bullet)$ if and only if $V$ satisfies smoothing.
\end{corollary}
Combining the above two results with \zcref{lem:strong identity eqns equiv} and \zcref{lem:strong identity eqns reparam}, we obtain the following:
\begin{theorem}[\cite[Theorem 5.0.3]{DGK23}]
    Let $V$ be a $\Z$-graded vertex $\kk$-algebra with an $\Aut\mathcal{O}$-action. Then $V$ satisfies the strong identity condition if and only if $V$ satisfies smoothing.
\end{theorem}
\begin{proof}
    Similar to the previous statements, the argument is the same as in \cite[Theorem 5.0.3]{DGK23}. Their proof uses \cite[Lemma 5.1.4, Lemma 5.1.5]{DGK23}, which hold in our case (cf. \zcref{lem:strong identity eqns equiv} and \zcref{lem:strong identity eqns reparam}).
\end{proof}
Once again, all of the cited results above (and the preliminary results cited within those papers) may be proven in exactly the same way over an arbitrary algebraically closed field $\kk$.

The main consequence of the smoothing characterization is the smoothing property, which we now describe. We need two preliminary results:
\begin{lemma}[\cite[Corollary 4.3.1]{DGK23}]\label{lem:cor 4.3.1}
    Suppose that $\mathbb{V}(V;(M^\bullet\otimes \mathfrak{A})[\![q]\!])_{\widetilde{\mathscr{C}},P_\bullet\sqcup Q_\pm,t_\bullet\sqcup s_\pm}$ is coherent over $S$. Then there are isomorphisms
    \begin{align*}
        \mathbb{V}(V;(M^\bullet\otimes \mathfrak{A})[\![q]\!])_{(\widetilde{\mathscr{C}}, P_\bullet\sqcup Q_\pm, t_\bullet\sqcup s_\pm)}
        &\cong \mathbb{V}(V;M^\bullet\otimes \mathfrak{A})_{(\widetilde{\mathscr{C}}, P_\bullet\sqcup Q_\pm, t_\bullet\sqcup s_\pm)}[\![q]\!]\\
        &= \mathbb{V}(V;M^\bullet\otimes \mathfrak{A})_{(\widetilde{\mathscr{C}}, P_\bullet\sqcup Q_\pm, t_\bullet\sqcup s_\pm)}\otimes_\kk \kk[\![q]\!].
    \end{align*}
\end{lemma}
\begin{proof}
    The proof is the same as in \cite[Corollary 4.3.1]{DGK23}.
\end{proof}
\begin{proposition}[\cite[Proposition 4.3.4]{DGK23}]
    Suppose that $\mathbb{V}(V;M^\bullet)_{(\mathscr{C}_0,P_\bullet,t_\bullet)}$is a finite-dimensional vector space over $\kk$. Then both 
    \[\mathbb{V}(V;M^\bullet[\![q]\!])_{(\mathscr{C},P_\bullet,t_\bullet)} \quad \text{and}\quad \mathbb{V}(V;(M^\bullet\otimes\mathfrak{A})[\![q]\!])_{(\widetilde{\mathscr{C}}, P_\bullet\sqcup Q_\pm, t_\bullet\sqcup s_\pm)}\]
    are coherent over $S$.
\end{proposition}
\begin{proof}
    The proof is the same as in \cite[Proposition 4.3.4]{DGK23}. Their proof uses the factorization theorem, which we have proven for our setting. Their proof also makes use of \cite[Lemma 4.3.3, Lemma 4.1.1]{DGK23}, and a generalization of these results holds for our case as well. Therefore, the argument passes through.
\end{proof}
The above two results and the smoothing characterization gives us the smoothing property:
\begin{corollary}[Smoothing]\label{smoothing property}
    Suppose that $V$ satisfies the strong identity condition, where $(\one_d)_{d\in \N}$ is the set of strong identity elements $\one_d\in \Ac_d$. 
    Let $M^\bullet$ be $V$-modules such that the sheaf $\mathbb{V}(V;(M^\bullet\otimes \Ac)[\![q]\!])_{(\widetilde C,P_\bullet\sqcup Q_\pm,t_\bullet\sqcup s_\pm)}$ is coherent over $S$. 
    Set $\one = \sum_{d\ge 0}\one_dq^d$, and let $\alpha\colon M^\bullet[\![q]\!]\to (M^\bullet\otimes \Ac)[\![q]\!]$ be the map induced by $u\mapsto u\otimes \one$. Then we have a commutative diagram
    \[ \begin{tikzcd}[column sep=3cm]
    {\mathbb{V}(V;M^{\bullet}[\![q]\!])_{(\Cs, P_\bullet, t_\bullet)}} {\arrow[r, "{[\alpha]}"]} {\arrow[d , two heads ]}
        & {\mathbb{V}(V;M^{\bullet}\otimes \Ac)_{(\widetilde{\Cs}_0 , P_\bullet \sqcup Q_\pm, t_\bullet \sqcup s_\pm)}[\![q]\!]} {\arrow[d, two heads]} \\
        {\mathbb{V}(V;M^{\bullet})_{(\Cs_0, P_\bullet, t_\bullet)}} \arrow[r,  "{[\alpha_0]}" ]
    & {\mathbb{V}(V;M^\bullet \otimes \Ac )}_{(\widetilde{\Cs}_0 , P_\bullet\sqcup Q_\pm, t_\bullet \sqcup s_\pm)}
    \end{tikzcd}\]
    in which $\alpha_0\colon M^\bullet \to M^\bullet\otimes \Ac$ is given by $w\mapsto w\otimes \one_0$.
\end{corollary}
\begin{proof}
    The argument is the same as in \cite[Corollary 5.2.1]{DGK23}. The vertical maps are given by setting $q=0$, so they are surjective. The horizontal maps $[\alpha]$ and $[\alpha_0]$ are well-defined by using essentially the same argument as in the beginning of the proof of the factorization theorem.
\end{proof}
\begin{proposition}
    Suppose that $V$ satisfies the strong identity equations and the spaces of coinvariants $\mathbb{V}(V;M^\bullet[\![q]\!])_{\mathscr{C},P_\bullet,t_\bullet}$ and $\mathbb{V}(V;M^\bullet[\![q]\!])_{\widetilde{\mathscr{C}},P_\bullet\sqcup Q_\pm,t_\bullet\sqcup s_\pm}$ are coherent over $S$. Then the map $[\alpha]$ defined in \zcref{smoothing property} is an isomorphism.
\end{proposition}
\begin{proof}
    The proof is the same as \cite[Lemma 5.2.5]{DGK23}, except we do not need to assume $C_1$-cofiniteness since that was required for their factorization result. Note that the source of $[\alpha]$ is finitely generated and the target is locally free of finite rank (the finite rank part follows from \zcref{lem:cor 4.3.1}). The fact that $[\alpha]$ is an isomorphism follows from a corollary of Nakayama's lemma, which is that a surjective endomorphism $f\colon M\to M$ of a finitely generated $R$-module $M$ is an isomorphism.
\end{proof}


\bibliographystyle{amsalpha}
\bibliography{bibfile}


  





\end{document}